\documentclass[reqno, 11pt]{amsart}
\usepackage{amsmath,amssymb,amsthm, amscd, braket}

\usepackage{amsbsy}
\usepackage[initials]{amsrefs}

\BibSpec{article}{%
    +{}  {\PrintAuthors}                {author}
    +{,} { \textit}                     {title}
    +{.} { }                            {part}
    +{:} { \textit}                     {subtitle}
    +{,} { \PrintContributions}         {contribution}
    +{.} { \PrintPartials}              {partial}
    +{,} { }                            {journal}
    +{}  { \textbf}                     {volume}
    +{}  { \PrintDatePV}                {date}
    +{,} { \issuetext}                  {number}
    +{,} { \eprintpages}                {pages}
    +{,} { }                            {status}
    +{,} { }                            {eprint}
    +{}  { \PrintTranslation}           {translation}
    +{;} { \PrintReprint}               {reprint}
    +{.} { }                            {note}
    +{.} {}                             {transition}
}
\BibSpec{book}{%
    +{}  {\PrintPrimary}                {transition}
    +{,} { \textit}                     {title}
    +{.} { }                            {part}
    +{:} { \textit}                     {subtitle}
    +{,} { \PrintEdition}               {edition}
    +{}  { \PrintEditorsB}              {editor}
    +{,} { \PrintTranslatorsC}          {translator}
    +{,} { \PrintContributions}         {contribution}
    +{,} { }                            {series}
    +{,} { \voltext}                    {volume}
    +{,} { }                            {publisher}
    +{,} { }                            {organization}
    +{,} { }                            {address}
    +{,} { \PrintDateB}                 {date}
    +{,} { }                            {status}
    +{}  { \PrintTranslation}           {translation}
    +{;} { \PrintReprint}               {reprint}
    +{.} { }                            {note}
    +{.} {}                             {transition}
}
\BibSpec{collection.article}{%
    +{}  {\PrintAuthors}                {author}
    +{,} { \textit}                     {title}
    +{.} { }                            {part}
    +{:} { \textit}                     {subtitle}
    +{,} { \PrintContributions}         {contribution}
    +{,} { \PrintConference}            {conference}
    +{}  {\PrintBook}                   {book}
    +{,} { }                            {booktitle}
    +{,} { \PrintDateB}                 {date}
    +{,} { pp.~}                        {pages}
    +{,} { }                            {status}
    +{,} { }                            {eprint}
    +{}  { \PrintTranslation}           {translation}
    +{;} { \PrintReprint}               {reprint}
    +{.} { }                            {note}
    +{.} {}                             {transition}
}

\usepackage{footnote}
\usepackage{latexsym}
\usepackage[all]{xy}
\usepackage{mathrsfs}
\usepackage{dsfont}
\usepackage{bm}
\usepackage{stmaryrd}
\usepackage{ulem}
\usepackage{ifpdf}
\ifpdf\else
\PassOptionsToPackage{dvipdfmx}{hyperref}
\PassOptionsToPackage{dvipdfmx}{graphicx}
\PassOptionsToPackage{dvipdfmx}{xcolor}
\fi
\usepackage{hyperref}

\hypersetup{
    colorlinks=true,
    citecolor=blue,
    linkcolor=blue,
  urlcolor=blue,
}

\newcommand{\Tw}{\mathrm{Tw}}
\newcommand{\rH}{\mathrm{H}}
\newcommand{\rf}{\mathrm{f}}
\newcommand{\bT}{\mathbb{T}}

\newcommand{\scrS}{\mathscr{S}}
\newcommand{\rel}{\mathrm{rel}}
\newcommand{\str}{\mathrm{str}}
\newcommand{\coLie}{\mathrm{coLie}}
\usepackage{amscd,amsthm,amsfonts,amssymb,amsmath}
\usepackage{amsmath,tikz-cd}
\usepackage{fullpage}
\usepackage{mathtools}
\usepackage{nccmath}
 
\newcommand{\F}{\mathbb{F}}
\newcommand{\pst}{\mathrm{pst}}

\newcommand{\per}{\mathrm{per}}
\newcommand{\bfH}{\mathbf{H}}

     \newcommand{\CL}{{\mathcal {L}}}
     
    \newcommand{\CO}{{\mathcal {O}}}

    \newcommand{\fa}{{\mathfrak{a}}} \newcommand{\fb}{{\mathfrak{b}}}
    \newcommand{\fc}{{\mathfrak{c}}} 
     \newcommand{\ff}{{\mathfrak{f}}}
    \newcommand{\fg}{{\mathfrak{g}}}

    \newcommand{\fm}{{\mathfrak{m}}} 
     \newcommand{\fp}{{\mathfrak{p}}}
    \newcommand{\fq}{{\mathfrak{q}}}

     \newcommand{\fz}{{\mathfrak{z}}}

    \newcommand{\Aut}{{\mathrm{Aut}}}

    \newcommand{\Ch}{{\mathrm{Ch}}}
    
    \newcommand{\cris}{{\mathrm{cris}}}
    \newcommand{\corank}{{\mathrm{corank}}}

    \newcommand{\cyc}{{\mathrm{cyc}}}
    \newcommand{\dR}{{\mathrm{dR}}}

     \renewcommand{\div}{{\mathrm{div}}}
    \newcommand{\End}{{\mathrm{End}}} 
    
    \newcommand{\Fr}{{\mathrm{Fr}}}

    \newcommand{\Gal}{{\mathrm{Gal}}} \newcommand{\GL}{{\mathrm{GL}}}
    
    \newcommand{\Hom}{{\mathrm{Hom}}}

    \newcommand{\Ind}{{\mathrm{Ind}}}

    \newcommand{\loc}{{\mathrm{loc}}}

    \newcommand{\ord}{{\mathrm{ord}}} \newcommand{\rank}{{\mathrm{rank}}}

    \newcommand{\Res}{{\mathrm{Res}}}
    \newcommand{\Sel}{{\mathrm{Sel}}}

    \newcommand{\Spec}{{\mathrm{Spec}}}

    \newcommand{\sgn}{{\mathrm{sgn}}}

    \newcommand{\tor}{{\mathrm{tor}}}
    
    \newcommand{\ur}{{\mathrm{ur}}}

    \newcommand{\Q}{\mathbb{Q}}
    \newcommand{\Z}{\mathbb{Z}}

    \newcommand{\ac}{\mathrm{ac}}

\newcommand{\C}{\mathbb{C}}

\renewcommand{\det}{\mathrm{det}}

    \newcommand{\ov}{\overline}

    \theoremstyle{plain}
    \newtheorem{thm}{Theorem}[section] \newtheorem{cor}[thm]{Corollary}
    \newtheorem{lem}[thm]{Lemma}  \newtheorem{prop}[thm]{Proposition}
     
        \newtheorem{assumption}[thm]{Assumption}

\theoremstyle{remark} \newtheorem{remark}[thm]{Remark}
\theoremstyle{remark} \newtheorem{defn}[thm]{Definition}
\theoremstyle{remark} 
\theoremstyle{remark}

    \newcommand{\et}{\'{e}t}
    \renewcommand{\et}{{\text{\'{e}t}}}

    \numberwithin{equation}{section}

\makeatletter
\@namedef{subjclassname@2020}{
 \textup{2020} Mathematics Subject Classification}
\makeatother 

\begin{document}

\title{Anticyclotomic Iwasawa theory of CM elliptic curves at ramified primes}
\author{Ashay A. Burungale, Shinichi Kobayashi, Kentaro Nakamura, Kazuto Ota}
\address{Ashay A. Burungale:  The University of Texas at Austin, Austin, TX 78712, USA.
} 
\email{ashayburungale@gmail.com}

\address{Shinichi Kobayashi: Faculty of Mathematics,
Kyushu University, 744, Motooka, Nishi-ku, Fukuoka, 819-0395, Japan.}
\email{kobayashi@math.kyushu-u.ac.jp}

\address{Kentaro Nakamura: 
Faculty of Mathematics,
Kyushu University, 744, Motooka, Nishi-ku, Fukuoka, 819-0395, Japan.}
\email{nakamura.kentaro.858@m.kyushu-u.ac.jp}

\address{Kazuto Ota: Department of Mathematics, Graduate School of Science, Osaka University Toyonaka, Osaka 560-0043, Japan} 
\email{kazutoota@math.sci.osaka-u.ac.jp
}
\begin{abstract}
We propose an integral framework for the anticyclotomic Iwasawa theory of CM elliptic
curves $E$ at primes $p$ ramified in the CM field. The $\varepsilon$-constants of the
geometric specialisations of the associated $p$-adic conjugate symplectic self-dual
deformation equidistribute between $\pm 1$ within every layer of the anticyclotomic
tower, and none of the geometric specialisations are trianguline at $p$. We define signed Selmer groups via the Lagrangian local
conditions arising from the local sign decomposition established in the prequel
\cite{BKNO}, and our central result is the formulation and proof of an integral
Iwasawa main conjecture relating one of them to the $p$-adic $L$-function
$\mathscr{L}_p(E)$ constructed there. We further show that $\mathscr{L}_p(E)$
interpolates the central Hecke $L$-values of the twists with $\varepsilon$-constant
$+1$, including twists of arbitrary infinity type, and relate its values at twists
with $\varepsilon$-constant $-1$ to the $p$-adic logarithm of certain Selmer
elements. This provides the first Iwasawa main conjecture in terms of a $p$-adic
$L$-function and Selmer groups for a $p$-adic deformation admitting no trianguline
geometric specialisation.

The proofs rest on our resolution of a Rubin-type conjecture for the underlying local
deformation, together with a theory of plus/minus local points along the
anticyclotomic tower, based on the Gaussian plus/minus cyclotomic polynomials rooted
in Gauss' \textit{Disquisitiones Arithmeticae}.
\end{abstract}

\maketitle

\tableofcontents

\section{Introduction}
In this paper, we propose an integral framework for anticyclotomic CM Iwasawa theory
at ramified primes and establish some foundational results. It is a sequel to
\cite{BKNO}, in which we obtained, as a consequence of the local sign decomposition, a
ramified Rubin-type decomposition for the local anticyclotomic CM deformation: a
decomposition into signed Lagrangian submodules that encode Bloch--Kato subgroups via
completed local $\varepsilon$-constants. Here we globalise this local structure. While~Theorem~\ref{intro-local-result} is our new local result, 
the main global results are Theorems~\ref{thm, pL},~\ref{main-result-Intro-0}, and~\ref{thm, pL-oi}. 
\subsection{Context}
The Iwasawa theory of CM elliptic curves is a classical subject: already in the
mid-1970s, the seminal work of Coates and Wiles led to the first theoretical results
towards the Birch and Swinnerton-Dyer conjecture \cite{CW}. Since then, the subject
has continued to progress, revealing new phenomena, giving rise to systematic
theories explaining them, and yielding arithmetic applications. In turn, these have
contributed tools and perspectives to the Iwasawa theory of motives over number
fields.

A primary goal of Iwasawa theory is to formulate and study an Iwasawa main conjecture
for a $p$-adic deformation of a motive over a number field. A main conjecture may be
stated in terms of zeta elements and determinants of Galois cohomology
(cf.~\cite{K,PRbook}), with neither local conditions at $p$ nor a $p$-adic
$L$-function involved; for the CM deformations of this paper, a main conjecture of
this type --- independent of the reduction type --- is supplied by Rubin's
fundamental work on elliptic units \cite{Ru91} (see \S\ref{section-Selmer-group}).
This paper concerns a main conjecture in terms of a $p$-adic $L$-function: a
statement relating a $p$-adic $L$-function, interpolating the $L$-values of the
geometric specialisations normalised by periods, to a Selmer group over the Iwasawa
algebra, defined via a choice of local condition at $p$ --- a Selmer structure at $p$
--- along the deformation. At each geometric specialisation there is a canonical
local condition, the Bloch--Kato subgroup; over the deformation, however, there is no a priori local condition interpolating these subgroups. 
Consequently, the existence of a Selmer structure at $p$ related to the Bloch--Kato subgroups at
the geometric points is far from clear. It is this form of the main conjecture that
ties the deformation to the arithmetic of the individual specialisations, and that
underlies the strategy of proving main conjectures via congruences of automorphic
forms going back to Ribet and Mazur--Wiles. Beyond the cyclotomic setting, it is
largely a mystery: the existence of the $p$-adic $L$-function, its interpolation
formula, and that of such a Selmer structure are unknown in general, and their nature
depends crucially on the prime $p$. (In the CM setting, elliptic units suffice for
the proof, as Rubin's work shows; the $p$-adic $L$-function is needed here not for
the proof but for the theory, and the present case serves as a test case where none
of the existing frameworks applies.)

Our approach is guided by a principle which emerged from our study of the local sign
decomposition \cite{BKNO}. For a (conjugate) symplectic self-dual deformation, the
local Iwasawa cohomology carries a canonical perfect symmetric pairing arising
from local Tate duality, and one may consider Lagrangian submodules of the
cohomology in family. A Lagrangian family alone does not determine an arithmetic local condition; what matters is its relation to the Bloch–Kato subgroups at the geometric specialisations, and—this is the upshot of {\it loc. cit.}—that relation is governed by the local epsilon constants. When the epsilon constants are biased along the deformation, a Lagrangian complementary to the Bloch–Kato subgroups at almost all specialisations may serve as a Selmer structure at $p$. When they equidistribute, no single choice suffices, and several Lagrangians must be considered together.

The existing frameworks are instances of the biased case. For cyclotomic deformations
at ordinary and Panchishkin primes, the Selmer structure at $p$ arises from a
Galois-stable subrepresentation, following Greenberg \cite{Gr91}. At non-ordinary
semistable primes, Perrin-Riou's theory \cite{PRbook,PR-semistable} rests on her big
exponential map, which interpolates the Bloch--Kato exponential maps along the
deformation; as the image of the Bloch--Kato exponential is the Bloch--Kato subgroup,
her theory may be regarded as an interpolation of the Bloch--Kato subgroups
themselves, in a form allowing appropriate denominators. A Selmer structure at $p$ in
this setting arises from a trianguline sub-$(\varphi,\Gamma)$-module family over the
Robba ring, following Pottharst \cite{Pt}. In these settings the epsilon constants are essentially fixed along the deformation, and 
and the available local structures naturally single out local conditions compatible with the Bloch–Kato subgroups.
The anticyclotomic
deformations of CM elliptic curves at inert primes leave this paradigm: the root
numbers vary with the parity of the anticyclotomic layer \cite{Gr83,Gr01}, and the
theory envisioned by Rubin \cite{Ru} in the late 1980s and developed after the
resolution of his conjecture \cite{BKO21} in 2021 rests on a pair of Selmer
structures at $p$
\cite{BHKO,BKO24,BKOd,BKOe,BKOY} (see also \cite{AH2,Bu}). The present viewpoint also helps explain the difficulty of the inert case: the specialisations at non-trivial finite-order twists---those at which the Mordell--Weil growth occurs---are non-trianguline, and hence fall outside the scope of the theories built around crystalline specialisations at powers of Lubin--Tate characters (cf.~\cite{ScT}).

Our subject is the ramified case, where the phenomena are the most drastic. Let
$E/\Q$ be an elliptic curve with CM by the ring of integers $\CO_K$ of an imaginary
quadratic field $K$, let $p$ be an odd prime ramified in $K$, and let
$K_\infty^{\rm ac}$ be the anticyclotomic $\Z_p$-extension of $K$ with $n$-th layer
$K_n^{\rm ac}$. The global $\varepsilon$-constants satisfy
\begin{equation}\label{lem, ram eps, bis}
\varepsilon(\phi\chi^{a}) = \left(\frac{a}{p}\right) \varepsilon(\phi\chi),
\end{equation}
where $\phi$ is the Hecke character over $K$ associated to $E$, $\chi$ an
anticyclotomic character of order $p^n >1$ and $a$ an integer prime to $p$: the root
numbers equidistribute within every layer of the anticyclotomic tower. Accordingly
the Mordell--Weil ranks grow systematically --- for $n\gg 0$,
\begin{equation}\label{eq, MW-r}
\rank_{\CO_K} E(K_n^{\rm ac}) = \frac{p^n-1}{2}+c
\end{equation}
for a constant $c$ (cf.~Theorem~\ref{vGZK}); we return to this growth in
\S\ref{ss, lr}, in connection with the local points. Moreover --- in contrast with
the inert case --- not a single geometric specialisation of the deformation is
trianguline at $p$; in particular, none is semistable or crystabelline, and no
trianguline locus remains to build on. A
$p$-adic $L$-function has recently been constructed by analytic methods
(cf.~\cite{AI}); an Iwasawa-theoretic framework --- a Selmer structure at $p$, Selmer
groups over the Iwasawa algebra, a main conjecture --- has remained open. To our
knowledge, no Iwasawa main conjecture in terms of a $p$-adic $L$-function and Selmer
groups has previously been formulated for a $p$-adic deformation admitting no
trianguline geometric specialisation.

We propose such a framework and establish it. In our rank-two setting the principle
takes its minimal form: by the local sign decomposition --- our
resolution of a Rubin-type conjecture \cite{BKNO} --- the local Iwasawa cohomology
decomposes into two Lagrangian summands; the Bloch--Kato subgroups oscillate between
the two, and at each de Rham specialisation the completed local
$\varepsilon$-constant determines the summand specialising to the Bloch--Kato
subgroup. The plus/minus structure of the theory is thus forced by the rank and the
equidistribution of the root numbers. We show that the integral $p$-adic $L$-function
$\mathscr{L}_{p}(E)$ of \cite{BKNO} interpolates the central Hecke $L$-values of the
twists $\phi\chi$ with $\varepsilon(\phi\chi)=+1$, for geometric
characters $\chi$ of arbitrary infinity type
(Theorem~\ref{thm, pL}); we formulate and prove the signed Iwasawa main conjecture,
the central result of the paper (Theorem~\ref{main-result-Intro-0}); and we relate
the values of $\mathscr{L}_p(E)$ outside the interpolation range, at twists with
$\varepsilon(\phi\chi)=-1$, to Selmer elements (Theorem~\ref{thm, pL-oi}). The
finite-level tool is given by the \textit{Gaussian plus/minus cyclotomic polynomials}: unlike the
classical plus/minus polynomials \cite{Ko0,Po}, they partition the anticyclotomic
characters according to the signs of the $\varepsilon$-constants, and since these
mirror the quadratic residue symbol \eqref{lem, ram eps, bis}, our construction
recovers the quadratic-residue factors of the cyclotomic polynomial in Gauss'
\textit{Disquisitiones Arithmeticae} --- the constants in the trace relations encode
the class numbers of $\Q(\sqrt{\pm p})$ and the fundamental unit of $\Q(\sqrt{p})$
(Lemma~\ref{value-at-one}). The Gaussian polynomials isolate, at finite level, the
plus/minus Bloch--Kato subgroups, and lead to the plus/minus local points
(Theorem~\ref{intro-local-result}), the local counterparts of the growth
\eqref{eq, MW-r}.

\subsection{Main results}

\subsubsection{Set-up}
Let $E$ be an elliptic curve defined over $\Q$ with CM by an order of an imaginary
quadratic field ${K}$. Let $\phi$ be the associated Hecke character over $K$ of
infinity type $(1,0)$ such that
$$L({\phi},s)=L(E_{/\Q},s),$$
where we normalise the $L$-functions to have center at $s=1$. Let $p\geq 5$ be a
prime ramified in $K$, $\mathfrak{p}$ the prime of $K$ above $p$ and
$\Psi=K_{\mathfrak{p}}$.

The main text considers a conjugate symplectic self-dual Hecke character of infinity
type $(1,0)$ over a general imaginary quadratic field $K$, an associated CM abelian
variety over $K$ and any odd prime $p$. Accordingly, some notation of the
introduction differs from that of the main text.

Fix embeddings $\iota_{\infty}:\overline{\Q} \hookrightarrow  \C$ and
$\iota_{p}:\overline{\Q}\hookrightarrow  \overline{\Q}_p$.
Let ${\phi}$ also denote the associated $p$-adic Galois character over $K$.
Put $\Gamma_{\rm ac}=\mathrm{Gal}(K_\infty^{\rm ac}/K)$ and
$\Lambda=\mathcal{O}_\Psi[\![\Gamma_{\rm ac}]\!]$. 

For the variation \eqref{lem, ram eps, bis} and its extension to the twists of
$\phi$ by the infinite-order de Rham characters of $\Gamma_{\rm ac}$, see
Lemma~\ref{density-of-gxi}.

\begin{defn}
 For $n\in \Z_{\geq 1}$, define a partition of the set $\Xi_{n}$ of anticyclotomic
 characters over $K$ of order $p^n$ by
 \[
 \Xi_{n}=\Xi_{\phi, n}^+ \cup \Xi_{\phi, n}^-=\{\chi \,|\, \varepsilon(\phi\chi)=+1\}\cup
 \{\chi \,|\, \varepsilon(\phi\chi)=-1\}.
 \]
 \end{defn}
Note that $|\Xi_{\phi, n}^+| = |\Xi_{\phi, n}^-|$ by \eqref{lem, ram eps, bis}.
For $\circ \in\{\emptyset, +,-\}$, put  $\Xi^\circ_{\phi}=\cup_n \Xi_{\phi, n}^\circ$.

Put
\begin{equation}\label{eq, st}
\varepsilon=\varepsilon_\phi
:=\sgn({\varepsilon}({\phi})/\varepsilon_p({\rm Ind}_{\Psi/\Q_p}\phi_p)),
\qquad
\bT_{\phi}=T_{{\phi}}^{\otimes -1}(1)\otimes_{\CO_K} \Lambda^{\iota},
\end{equation}
where $\iota$ denotes the involution of $\Lambda$ arising from
$\gamma \mapsto \gamma^{-1}$ for $\gamma \in \Gamma_{\rm ac}$. 

\subsubsection{Local results}\label{ss, lr}
To develop an integral theory of Perrin-Riou maps for a given $p$-adic deformation,
as in the cyclotomic supersingular case \cite{Ko0}, and to define global Selmer
groups, two local ingredients are essential: a systematic
characterisation of how Bloch--Kato subgroups vary at de Rham points, and the
construction of a primitive system of local points within these subgroups satisfying
certain trace relations.

Let $\Psi/\mathbb{Q}_{p}$ be a ramified quadratic extension and $\fp$ the maximal
ideal of $\CO_{\Psi}$. Let $\CO$ be the integer ring of a $p$-adic local field
containing $\sqrt{(-1)^{(p-1)/2}p}$, and let
$\psi: \Psi^{\times} \to \CO^{\times}$ be a continuous character that is conjugate
symplectic self-dual, i.e.\ its restriction to $\Q_p^{\times}$ satisfies
\begin{equation}\label{local-self-dual, intro}
\psi(a)=\omega_{\Psi/\Q_p}(a)|a|_{\Q_p^{\times}}\quad (a \in \Q_p^{\times}),
\end{equation}
where $\omega_{\Psi/\Q_p}$ denotes the quadratic character associated to
$\Psi/\Q_p$. We regard $\psi$ as a Galois character via local class field theory.
Let $T_{\psi}$ denote the associated free rank one $\mathcal{O}$-module, on which
$G_{\Psi}$ acts by $\psi$. Suppose that $T_{\psi} \otimes_{\Z_p} \Q_p$ is de Rham.

Let $\Psi_{\infty}/\Psi$ be the anticyclotomic $\mathbb{Z}_{p}$-extension with
$n$-th layer $\Psi_{n}$, and
$\Lambda = \mathcal{O}[\![\mathrm{Gal}(\Psi_{\infty}/\Psi)]\!]$ the associated
Iwasawa algebra. For the $p$-adic deformation
$\bT_{\psi}:=T_{\psi}^{\otimes-1}(1) \otimes_{\mathcal{O}} \Lambda$, the Iwasawa
cohomology $H^1(\Psi, \bT_{\psi})$ is a free $\Lambda$-module of rank two. By the
proof of the Rubin-type conjecture in \cite[Thm.~1.18]{BKNO}, this cohomology
decomposes into free, rank-one Lagrangian $\Lambda$-submodules:
\begin{equation}\label{eq, RuC}
H^{1}(\Psi, \bT_{\psi}) = H_{+}^{1}(\Psi, \bT_{\psi}) \oplus H_{-}^{1}(\Psi, \bT_{\psi}).
\end{equation}
The Lagrangian submodules $H_{\pm}^{1}(\Psi, \bT_{\psi})$ interpolate the Bloch--Kato
subgroups at de Rham specialisations in terms of the associated completed
$\varepsilon$-constants $\in \{ \pm 1 \}$ (cf.~\cite[\S1.3.3]{BKNO}).

Accordingly, we define the plus/minus Bloch--Kato subgroups at finite layers by
$$H_{\rm f,\pm}^{1}(\Psi_{n}, T_{\psi}^{\otimes-1}(1)) :=
H_{\rm f}^{1}(\Psi_{n}, T_{\psi}^{\otimes-1}(1)) \cap
H_{\pm}^{1}(\Psi_{n}, T_{\psi}^{\otimes-1}(1)),$$
where $H_{\pm}^{1}(\Psi_{n}, T_{\psi}^{\otimes-1}(1))$ denotes the image of
$H_{\pm}^{1}(\Psi, \bT_{\psi})$ under the natural projection.

We partition the set $\Xi_{k}$ of characters of $\mathrm{Gal}(\Psi_{k}/\Psi)$ of
order $p^{k}$ by\footnote{Note that $\mathrm{Ind}_{\Psi/\mathbb{Q}_{p}}(\psi\chi)$ is
symplectic self-dual and so the $\varepsilon$-constant is well defined, independent
of various choices.}
\[
\Xi_{\psi,k}^{\pm} = \{ \chi \in \Xi_{k} \mid
\varepsilon(\mathrm{Ind}_{\Psi/\mathbb{Q}_{p}}(\psi\chi)) = \pm 1 \}.
\]
\begin{defn}
Fix a topological generator $\gamma \in \mathrm{Gal}(\Psi_{\infty}/\Psi)$. For an
integer $k \geq 1$, define the \textit{Gaussian plus/minus cyclotomic polynomials} by
\[
\Phi_{k,\psi}^{\pm}(\gamma) = \prod_{\chi \in \Xi_{\psi,k}^{\pm}} (\gamma - \chi(\gamma))
\in \CO[\Gal(\Psi_{k}/\Psi)].
\]
\end{defn}

A key local result of this paper is the following (cf.~Theorem~\ref{c-generates-fpm}).

\begin{thm}\label{intro-local-result}
Let $p\geq 5$ be a prime, $\Psi$ a ramified quadratic extension of $\Q_p$ and $\psi$
a conjugate symplectic self-dual Galois character over $\Psi$ as above. There exists
a system of plus/minus local points
$c_{n}^{\pm} \in H_{\rm f, \pm}^{1}(\Psi_{n}, T_{\psi}^{\otimes-1}(1))$, $n\geq 0$,
with the following properties.
\begin{enumerate}
    \item  For the corestriction map $\mathrm{Cor}_{n+1/n}: H^{1}(\Psi_{n+1},
    T_{\psi}^{\otimes-1}(1)) \rightarrow H^{1}(\Psi_{n}, T_{\psi}^{\otimes-1}(1))$,
    we have
    \[
    \mathrm{Cor}_{n+1/n}(c_{n+1}^{\pm}) = \Phi_{1,\psi}^{\pm\epsilon(\Psi)}(1) c_{n}^{\pm},
    \]
    where $\epsilon(\Psi)$ is as in Definition~\ref{def, Psi}.
    \item   As an $\mathcal{O}[\mathrm{Gal}(\Psi_{n}/\Psi)]$-module,
    the elements $c_{n}^{\pm}$ generate
    $H_{\rm f,\pm}^{1}(\Psi_{n}, T_{\psi}^{\otimes-1}(1))$.
    \item  As $\mathcal{O}[\mathrm{Gal}(\Psi_{n}/\Psi)]$-modules, we have
    \[
    H_{\rm f}^{1}(\Psi_{n}, T_{\psi}^{\otimes-1}(1)) = H_{\rm f,+}^{1}(\Psi_{n},
    T_{\psi}^{\otimes-1}(1)) \oplus H_{\rm f,-}^{1}(\Psi_{n}, T_{\psi}^{\otimes-1}(1)).
    \]
\end{enumerate}
\end{thm}
The elements $c_{n}^{\pm}$ are local counterparts of the global points of infinite
order contributing to the growth of Mordell--Weil ranks as in \eqref{eq, MW-r}; as
explained in \S\ref{ss, oi}, the relation is more than an analogy.

\begin{remark}The constant $\Phi_{1,\psi}^{\pm\epsilon(\Psi)}(1)$ appearing in the
above trace relation encodes the class numbers of the quadratic fields
$\Q(\sqrt{\pm p})$ and the fundamental unit of $\Q(\sqrt{p})$ (see
Lemma~\ref{value-at-one}).
\end{remark}

\subsubsection{$p$-adic $L$-function: interpolation of central Hecke $L$-values}

Based on the Rubin-type conjecture \eqref{eq, RuC}, the following integral
anticyclotomic $p$-adic $L$-function was introduced in \cite{BKNO}.

For $\varepsilon \in \{\pm 1\}$ as in \eqref{eq, st}, fix a basis $v_{\varepsilon}$
of $H^1_{\varepsilon}(\Psi, \bT_{\phi})$ and an $\mathcal{O}$-basis $t$ of
$T_{\phi}^{\otimes -1}$. The $p$-adic $L$-function
$\mathscr{L}_{p, v_{\varepsilon}}(E) \in \Lambda_{}$ is defined by
\[
\loc_{p}(z_{p^\infty\mathfrak{f}}^{t, {\rm ac}}) =
\mathscr{L}_{p, v_{\varepsilon}}(E) \cdot v_{\varepsilon},
\]
where $z_{p^\infty\mathfrak{f}}^{t, {\rm ac}}$ is the anticyclotomic elliptic
unit class associated to $t$ and the ideal $\mathfrak{f}$
(see \S\ref{ss, emc-t}). Note that the ideal
$(\mathscr{L}_{p, v_{\varepsilon}}(E)) \subseteq \Lambda$ is independent of the
choices of $v_{\varepsilon}$ and $t$.

Let $\eta$ denote the finite order Hecke character over $K$ such that
$\phi\eta^{-1}$ factors through $\Gamma:=\Gal(K_{\infty}/K)$, and put
$$\phi_{\rm ac}=\phi/\phi^{c}, \quad \eta_{\rm ac}=\eta/\eta^{c},$$
where $\phi^{c}:=\phi\circ c$ for $c$ the non-trivial element of $\Gal(K/\Q)$. The de
Rham characters of $\Gamma_{\rm ac}$ are of the form
$$\chi=\phi_{\rm ac}^{k}\,\eta_{\rm ac}^{-k}\,\chi'$$
for an integer $k\geq 0$ and a finite order character $\chi'$ of $\Gamma_{\rm ac}$,
up to the involution $\iota$; the twist $\phi\chi$ is then a conjugate symplectic
self-dual Hecke character of infinity type $(k+1,-k)$.

The $p$-adic $L$-function interpolates the following central Hecke $L$-values
(cf.~Theorem~\ref{interpolation-rubin}).

\begin{thm}\label{thm, pL}
Let $E$ be a CM elliptic curve defined over $\Q$ with CM by an order of an imaginary
quadratic field $K$. Let $\phi$ be the associated Hecke character and $p\geq 5$ a
prime ramified in $K$. Let $v_{\varepsilon}$ be a basis of the $\Lambda$-module
$H^{1}_{\varepsilon}(\Psi, \bT_{\phi})$, and $\mathscr{L}_{p, v_{\varepsilon}}(E)$
the associated $p$-adic $L$-function. For any de Rham character
$\chi=\phi_{\rm ac}^{k}\eta_{\rm ac}^{-k}\chi'$ of $\Gamma_{\rm ac}$ such that
$$\varepsilon(\phi\chi)=+1,$$
we have
\[
\chi(\mathscr{L}_{p, v_{\varepsilon}}(E))
=\left(-\frac{2\pi}{\sqrt{|d_K|}}\right)^{\!k}
\frac{L_{p\mathfrak{f}}(\phi^{2k+1}\eta_{\rm ac}^{-k}\chi',\, k+1)}
{\delta_{v_{\varepsilon}}^{\omega_{E}}(\chi)\cdot \Omega_{\infty}^{2k+1}},
\]
where $d_K$ denotes the fundamental discriminant of $K$, $\omega_{E}$ the N\'eron
differential of a minimal Weierstrass model of $E$,
$\Omega_{\infty}\in\mathbb{C}^{\times}$ the CM period determined by $\omega_{E}$ and
$t$, and $\delta_{v_{\varepsilon}}^{\omega_{E}}(\chi)$ the $p$-adic period given
by the image of $v_{\varepsilon}$ under the dual exponential map associated to $\chi$
and $\omega_{E}^{\otimes(2k+1)}$ (see \S\ref{interpolation-section}). Moreover,
$\mathscr{L}_{p, v_{\varepsilon}}(E)$ is non-zero.
\end{thm}

\begin{remark}
The non-vanishing of the $p$-adic period $\delta_{v_{\varepsilon}}^{\omega_{E}}(\chi)$ in the above
interpolation formula is a consequence of the Rubin-type decomposition \eqref{eq, RuC}
(cf.~\cite[Cor.~3.2]{BKNO}).
\end{remark}

For finite order characters $\chi$ (i.e.\ $k=0$), the above interpolation formula
goes back to \cite{BKNO}. In the main text, the hypothesis appears in the equivalent
form $\hat{\varepsilon}_p(\Ind_{K/\Q}(\phi\chi))=\varepsilon$
(see Lemma~\ref{density-of-gxi} and Theorem~\ref{interpolation-rubin}).

\subsubsection{Iwasawa main conjecture}

We now globalise the local theory. 

Following the principles of plus/minus Iwasawa
theory \cite{Ko0}, we define the plus/minus anticyclotomic Selmer groups
$\mathrm{Sel}_{\pm}(K_{\infty}^{\rm ac}, W(E))$ via the signed local conditions at
$\mathfrak{p}$ arising from the Rubin-type decomposition \eqref{eq, RuC}
(see~\S\ref{ss, lr}), where $W(E)$ denotes the associated divisible module
(see \S\ref{section-Selmer-group}). Let $X_{\pm}(E)$ denote their Pontryagin duals.
The signed main conjecture relates the $p$-adic $L$-function
$\mathscr{L}_{p, v_{\varepsilon}}(E)$, constructed via the $\varepsilon$-Lagrangian
of the local Iwasawa cohomology, to the Selmer group $X_{-\varepsilon}(E)$ governed
by the opposite local condition.

The central result of this paper is the following (cf.~Theorems~\ref{rank} and
~\ref{main-result}).

\begin{thm}\label{main-result-Intro-0}
Let $E$ be a CM elliptic curve defined over $\Q$ with CM by an order of an imaginary
quadratic field $K$. Let $\phi$ be the associated Hecke character and $p\geq 5$ a
prime ramified in $K$. Let $v_{\varepsilon}$ be a basis of the $\Lambda$-module
$H^{1}_{\varepsilon}(\Psi, \bT_{\phi})$, and $\mathscr{L}_{p, v_{\varepsilon}}(E)$
the associated $p$-adic $L$-function.
\begin{enumerate}
    \item The signed anticyclotomic Selmer group $X_{-\varepsilon}(E)$ is a torsion
    $\Lambda$-module.
    \item As ideals in $\Lambda$, we have
    $$ \Ch_{\Lambda}(X_{-\varepsilon}(E)) =
    (\mathscr{L}_{p, v_{\varepsilon}}(E)). $$
\end{enumerate}
\end{thm}

\subsubsection{$p$-adic $L$-function: outside the range of interpolation}\label{ss, oi}

We now turn to the arithmetic of the $p$-adic $L$-function outside its defining
range of interpolation.

While Theorem~\ref{thm, pL} governs the arithmetic of
$\mathscr{L}_{p, v_{\varepsilon}}(E)$ for the twists with
$\varepsilon(\phi\chi)=+1$, the central Hecke $L$-values vanish whenever
$\varepsilon(\phi\chi)=-1$ by the functional equation. In light of the $p$-adic
Beilinson and BDP-type formulas (cf.~\cite{BDP, BKO24, AI}), this region is expected
to be related to non-torsion elements in the associated Selmer groups.

Along these lines, our main result is the following (cf.~Theorem~\ref{p-adic-L-with-global-class}).

\begin{thm}\label{thm, pL-oi}
Let $E$ be a CM elliptic curve defined over $\Q$ with CM by an order of an imaginary
quadratic field $K$. Let $\phi$ be the associated Hecke character and $p\geq 5$ a
prime ramified in $K$. Let $v_{\varepsilon}$ be a basis of the $\Lambda$-module
$H^{1}_{\varepsilon}(\Psi, \bT_{\phi})$, $\mathscr{L}_{p, v_{\varepsilon}}(E)$ the
associated $p$-adic $L$-function, and $v_{-\varepsilon}$ the basis of the
$\Lambda$-module $H^{1}_{-\varepsilon}(\Psi, \bT_{\phi})$ normalised by
$(v_{\varepsilon}, v_{-\varepsilon})=1$ with respect to the natural $\Lambda$-adic
pairing (see \S\ref{s, pLo}). For any de Rham character
$\chi=\phi_{\rm ac}^{k}\eta_{\rm ac}^{-k}\chi'$ of $\Gamma_{\rm ac}$ such that
$$\varepsilon(\phi\chi)=-1,$$
we have $\chi(z_{p^{\infty}\mathfrak{f}}^{t, {\rm ac}}) \in
H^{1}_{\rm f}(K, T_{\phi\chi}^{\otimes -1}(1))$, and for any non-zero
$\omega \in D_{\rm dR}^{0}\bigl(\Ind_{K/\Q}(V_{\phi\chi}^{\otimes -1}(1))\bigr)$
\[
\chi(\mathscr{L}_{p, v_{\varepsilon}}(E)) =
\exp_{\omega}^{*}(\chi(v_{-\varepsilon})) \cdot
\log_{\omega}(\chi(z_{p^{\infty}\mathfrak{f}}^{t, {\rm ac}})),
\]
where $T_{\phi\chi}^{\otimes -1}(1)$ denotes the specialisation of $\bT_{\phi}$ at
$\chi$ and $V_{\phi\chi}^{\otimes -1}(1)$ its rational counterpart, and
$\log_{\omega}$ and $\exp_{\omega}^{*}$ denote the Bloch--Kato logarithm and dual
exponential maps normalised by $\omega$ (see \S\ref{s, pLo}).
\end{thm}

For $k = 0$, the above Selmer classes are --- granting the finiteness of the
Tate--Shafarevich groups --- the very classes producing the Mordell--Weil rank
growth \eqref{eq, MW-r}; moreover, for $n \gg 0$, their localisations at $p$
coincide with the specialisations of the local points of
Theorem~\ref{intro-local-result}, up to scalars. Theorem~\ref{thm, pL-oi} thus links
the $p$-adic $L$-function to the systematic appearance of points of infinite order
along the anticyclotomic tower, and the local points capture these global classes at
$p$.

\subsection{About the proofs}\label{ss, outline}
The following gives a brief outline of the proofs of the main results.

\subsubsection{Local points}
We first outline the construction of the plus/minus local points.

Fix a $\Lambda$-basis $v_{\pm}$ of the Lagrangian submodule
$H^{1}_{\pm}(\Psi, \bT_{\psi})$, and let $v_{n,\pm}$ denote its projection to
$H^{1}(\Psi_{n}, T_{\psi}^{\otimes-1}(1))$. The local points are defined by
$$ c_{n}^{\pm} = \Phi_{n,\psi}^{\pm}(\gamma)\, v_{n,\pm}. $$
The Gaussian polynomial annihilates the specialisations at characters of the
opposite sign, and so
$c_{n}^{\pm} \in H_{\rm f,\pm}^{1}(\Psi_{n}, T_{\psi}^{\otimes-1}(1))$.
An elementary renormalisation of the system $(c_{n}^{\pm})_{n}$ then leads to a
system satisfying the trace relation as in Theorem~\ref{intro-local-result}\,(1)
(cf.~Proposition~\ref{norm-relation-of-c} and the subsequent remark). That the local
points generate the plus/minus Bloch--Kato subgroups as in
Theorem~\ref{intro-local-result}\,(2) rests on the Rubin-type decomposition
\eqref{eq, RuC} (see Theorem~\ref{c-generates-fpm}).

\subsubsection{Global results}
The proof of Theorem~\ref{main-result-Intro-0} rests on the Euler system of elliptic
units and the Rubin-type decomposition \eqref{eq, RuC}. The key inputs are an
anticyclotomic elliptic unit main conjecture, deduced from Rubin's fundamental work
\cite{Ru91} via a descent (see Theorem~\ref{anticyc-zetaIMC}), and the sign of the
elliptic unit
$$\loc_{p}(z_{p^{\infty}\mathfrak{f}}^{t, {\rm ac}})
\in H^{1}_{\varepsilon}(\Psi, \bT_{\phi}),$$
whereby the local sign decomposition enters (see
Proposition~\ref{prop, zeta-in-epsilon}). Since
$\mathscr{L}_{p, v_{\varepsilon}}(E)$ is the coordinate of
$\loc_{p}(z_{p^{\infty}\mathfrak{f}}^{t, {\rm ac}})$ with
respect to $v_{\varepsilon}$, the Poitou--Tate duality then transforms the elliptic
unit main conjecture into the equality
$\Ch_{\Lambda}(X_{-\varepsilon}(E))
=(\mathscr{L}_{p, v_{\varepsilon}}(E))$. Along the way, we determine the structure
of the signed Selmer groups: $X_{\varepsilon}(E)$ is of $\Lambda$-rank one and
$X_{-\varepsilon}(E)$ is torsion, the essential input being the non-vanishing of
$\mathscr{L}_{p, v_{\varepsilon}}(E)$ and the freeness of the signed Lagrangian
submodules.

The interpolation formula of Theorem~\ref{thm, pL} is deduced from an explicit
reciprocity law for elliptic units at the de Rham specialisations of arbitrary
infinity type (see Proposition~\ref{higher-explicit-reciprocity-law}), which extends
Kato's reciprocity law \cite{K}, together with the non-vanishing of the $p$-adic
periods arising from the Rubin-type decomposition. As for Theorem~\ref{thm, pL-oi}, the sign of the elliptic unit shows that the
relevant specialisations are Selmer classes; the formula then follows since the
local Tate pairing at such a specialisation is computed by the Bloch--Kato
logarithm and dual exponential maps.

\subsection{Vistas}

In the inert case, the resolution of Rubin's conjecture catalysed the development of
anticyclotomic CM Iwasawa theory, including the determination of the $p$-adic
valuation of the periods appearing in the interpolation formula of Rubin's $p$-adic
$L$-function \cite{BKOd}, a BDP-type formula \cite{BKO24}, an asymptotic formula for
the size of Tate--Shafarevich groups along the anticyclotomic tower \cite{BKOe}, and
the determination of the $\mu$-invariant of Rubin's $p$-adic $L$-function
\cite{BHKO}. It also led to applications to the Birch and Swinnerton-Dyer
conjecture, such as a $p$-converse theorem \cite{BKO24}. Moreover, the corresponding
local points play a key role in Sangiovanni Vincentelli's construction of an Euler
system for an elliptic curve over an imaginary quadratic field with $p$ inert \cite{SV}.

The integral framework developed in this paper lays the foundations for the ramified
counterparts of these problems, on which we will report elsewhere.

\subsection{Plan} In Section \ref{section-local-pt}, we introduce the signed
Bloch--Kato subgroups and construct the signed local points. Section
\ref{section-Selmer-group} presents the elliptic unit main conjectures, which
underlie our main theorems. In Section \ref{general-interpolation}, we study the
interpolation property of the Rubin-type $p$-adic $L$-function at the de Rham
specialisations. Section \ref{s, rMC} introduces the signed Selmer groups, and
formulates and proves the signed Iwasawa main conjecture. In Section \ref{s, SMW},
we present a control theorem for the signed Selmer groups and its consequences for
asymptotic Selmer and Mordell--Weil ranks, including the formula \eqref{eq, MW-r}
(cf.~Theorem \ref{vGZK}). Finally, Section \ref{s, pLo} studies the values of the
$p$-adic $L$-function outside the defining range of interpolation, relating them to
Selmer classes.

\subsubsection*{Acknowledgments}
We thank K\^az\i m B\"uy\"ukboduk, Antonio Lei, Marco Sangiovanni Vincentelli,
Christopher Skinner and especially Ye Tian for discussions about related topics.
This work was partially supported by the NSF grant DMS 2302064,
and the JSPS KAKENHI grants JP22H00096, JP22K03231, JP25K06935, JP21K13774,
JP25K06953.

\subsection*{Conventions}
Let $G$ be a profinite group and $H$ a closed normal subgroup of $G$. Let $R$ be a
commutative Noetherian complete local $\Z_p$-algebra with finite residue field.

For an $R$-representation $M$ of $H$ (i.e.\ a free $R$-module of finite rank with a
continuous $R$-linear $H$-action), put
\[
\Ind_{H}^G(M)=\Set{f: G\to M \ | \ f \text{ continuous},\ f(hg)=hf(g)\ (h \in H,\ g\in G) },
\]
on which $G$ acts by $(gf)(x)=f(xg)$ for $g, x \in G$.

Suppose moreover that $M$ is an $R$-representation of $G$ and that $G/H$ is finite
and abelian. Then we have an isomorphism of $R[G]$-modules
\begin{equation}\label{ind-group-ring}
\Ind_{H}^G(M) \xrightarrow{\ \cong\ } M\otimes_{R}R[G/H],\quad
f \mapsto \sum_{x \in G/H} \tilde{x}^{-1} f(\tilde{x}) \otimes [x],
\end{equation}
where $\tilde{x}\in G$ denotes a lift of $x$ (the sum being independent of the
choice of lifts), and $G$ acts on $M\otimes_{R}R[G/H]$ by
$g\bigl(m\otimes \sum_{x\in G/H}a_x[x]\bigr)=gm \otimes \sum_{x\in G/H}a_x[xg^{-1}]$.

In turn, we have $R$-module isomorphisms
\[
\rH^i(H, M) \xleftarrow{\ \cong\ } \rH^i(G, \Ind_H^G(M))
\xrightarrow{\ \cong\ } \rH^i(G, M\otimes_R R[G/H]),
\]
where the first arrow is the Shapiro isomorphism, induced by the evaluation map
$\Ind_{H}^G(M) \to M$, $f\mapsto f(1)$, and the second is induced by
\eqref{ind-group-ring}. The composite yields an isomorphism
\[
\rH^i(H, M) \xrightarrow{\ \cong\ } \rH^i(G, M\otimes_R R[G/H])
\]
of $R[G/H]$-modules, where $G/H$ acts on $\rH^i(H, M)$ through the natural action
induced by the $G$-action on $M$ and conjugation on $H$, and on the right-hand side
through multiplication by the group-like elements.

\section{System of local points}\label{section-local-pt}

This section is purely local. We construct a system of local points for the
anticyclotomic $\Z_p$-deformation of the $p$-adic Tate module of a CM elliptic curve
over a quadratic extension of $\Q_p$, in a somewhat more general setting. The main
result is
 the construction of the plus/minus local points generating the signed
Bloch--Kato subgroups in Theorem~\ref{c-generates-fpm}.

\subsection{The local sign decomposition}\label{subsection-local-decomposition}
\subsubsection{Set-up} 
Let $\Psi$ be a quadratic field extension of $\Q_p$.
Let $\omega_{\Psi/\Q_p}$ denote the quadratic character associated to $\Psi/\Q_p$. 

Let $\psi: \Psi^{\times} \to \CO^{\times}$ be a continuous conjugate symplectic self-dual character,  
 i.e.
 \begin{equation}\label{local-self-dual}
\psi(a)=\omega_{\Psi/\Q_p}(a)|a|_{\Q_p^{\times}}\quad (a \in \Q_p^{\times}),
\end{equation}
 where  
$\CO$ denotes the integer ring of a sufficiently large finite extension $\CL$ of $\Psi$ such that $\CO[1/p]$ contains $\sqrt{(-1)^{(p-1)/2}p}$. 
Denote by $T_{\psi}$ the free $\CO$-module of rank one on which $G_{\Psi}:=\Gal(\ov{\Q}_{p}/\Psi)$ acts by $\psi$, which is regarded as a character of $G_{\Psi}$ via the arithmetically normalised local Artin map.
Suppose that $T_{\psi}[1/p]$ is de Rham and that $\Ind_{\Psi/\Q_p}(T_{\psi}[1/p])$ has Hodge--Tate weights $\{1,0\}$ (by convention, the Hodge--Tate weight of the $p$-adic cyclotomic character is one).

Throughout this section, unless stated otherwise, we assume the following.
\begin{assumption}\label{generic}
If $\Psi $ is isomorphic to $\Q_3(\sqrt{-3})$,
then $\rH^0(\Psi, T_{\psi}^{\otimes -1}(1) \otimes_{\CO}\CO/\fm)=0$,
where $\fm$ denotes the maximal ideal of $\CO$. 
\end{assumption}
\begin{remark}\leavevmode 
\begin{enumerate}
\item[i)] If $\Psi \not\cong \Q_3(\sqrt{-3})$,
then the above vanishing is automatic (cf.~\cite[Lem.\ 7.22]{BKNO}).
\item[ii)] Without Assumption \ref{generic},
the arguments in this section work after taking $\otimes_{\Z_p}\Q_p$.
\end{enumerate}
\end{remark}

Put $\bT_{\psi}=T_{\psi}^{\otimes -1}(1)\otimes \Lambda_{\ac}$ and 
\[
\tilde{\bT}_{\psi}=\mathrm{Ind}_{\Psi/\Q_p}(\bT_{\psi}).
\]
The $\Lambda_{\ac}$-representation $\tilde{\bT}_{\psi}$  of $G_{\Q_p}$ is symplectic self-dual of rank two (cf.~\cite[\S7]{BKNO}). 
For a continuous character $\chi:\Gal(\Psi_{\infty}^{}/\Psi) \to \CO^{\times}$,
the $\CO$-representation $\tilde{\bT}_{\psi}\otimes_{\Lambda_{\ac}, \chi}\CO=\Ind_{\Psi/\Q_p}(T_{\psi\chi}^{\otimes -1}(1))$
is also symplectic self-dual of rank two. 
If $\chi$ is a de Rham character, 
then 
$$\varepsilon(\Ind_{\Psi/\Q_p}(T_{\psi\chi}^{\otimes -1}(1))):= \varepsilon(\Ind_{\Psi/\Q_p}(D_{\mathrm{pst}}(T_{\psi\chi}^{\otimes -1}(1)[1/p])))\in \{\pm 1\}$$
and likewise for $\varepsilon(\Ind_{\Psi/\Q_p}(T_{\psi\chi}))$,
where $\varepsilon(\Ind_{\Psi/\Q_p}(D_{\mathrm{pst}}(T_{\psi\chi}^{\otimes -1}(1)[1/p])))$ denotes the $\varepsilon$-constant of the associated Weil--Deligne representation (cf.~\cite[\S 2.5]{BKNO} or \S \ref{local-epsilon-section}). 
Noting that  $\Ind_{\Psi/\Q_p}(T_{\psi\chi}^{\otimes -1}(1))\cong \Ind_{\Psi/\Q_p}(T_{\psi\chi})$,
we put $\varepsilon(\Ind_{\Psi/\Q_p}(\psi\chi)):=\varepsilon(\Ind_{\Psi/\Q_p}(T_{\psi\chi}^{\otimes -1}(1)))
=\varepsilon(\Ind_{\Psi/\Q_p}(T_{\psi\chi}))$. 

Let $\Psi_{\infty}/\Psi$ be the anticyclotomic $\Z_p$-extension and $\Psi_n$ the $n$-th layer. 
Denote by $\Xi$ the set of finite order characters of $\Gal(\Psi_{\infty}/\Psi)$
and by $\Xi_n$ the subset of $\Xi$ consisting of characters of order $p^n$.
Put 
\[
\Xi_n^{\pm} = \Set{\chi \in \Xi_n | \varepsilon(\Ind_{\Psi/\Q_p}(\psi\chi))=\pm 1},\quad \Xi^{\pm}=\bigcup_{n\ge 0}\Xi_n^{\pm}. 
\]

\begin{lem}\label{density-of-xi}
Both $\Xi^{+}$ and $\Xi^{-}$ are infinite. 
\end{lem}
\begin{proof}
Let $\ff$ denote the conductor of $\psi$.
If $p\CO_{\Psi}\nmid \ff$, then\footnote{This does not mean that $\ff$ is trivial in the ramified case.} the lemma follows from \cite[Prop.\ 7.4 \& Cor.\ 7.8]{BKNO}.

Hence, we assume that $p\CO_{\Psi}\mid \ff$.
Let $\delta \in \Psi^{\times}$ be a uniformiser such that $\delta^2\in \Q_p$.
Let $\psi_0: \Psi^{\times} \to \CO^{\times}$ be a continuous character such that $\psi_0|_{\Q_p^{\times}}=\omega_{\Psi/\Q_p}|\cdot |_{\Q_p}$,
 $\psi_0(a)=a \ (a \in 1+\fp )$ (\textit{i.e.} the conductor of $\psi_0$ coincides with the prime ideal $\fp$ of $\CO_{\Psi}$)
 and $\psi_0(\delta)=\psi(\delta)$,
whose existence follows from \cite[Lemmas 7.3 and 7.5]{BKNO}.
Then the character $\psi_0/\psi$ has conductor divisible by $p\CO_{\Psi}$ and is trivial on $\Q_p^{\times}$.
Hence,  $\psi_0/\psi$ is of $p$-power order (cf.\ \cite[Lem.\ 7.14]{BKNO}),
 and we may regard $\psi_0/\psi$ as an element in $\Xi.$
 
In the ramified case, for $\chi \in \Xi$ whose conductor is strictly contained in the conductor of $\psi_0/\psi$ and $b\in \Z_p^{\times}$,
\cite[Prop.\ 7.7 (ii)]{BKNO} implies that
\begin{equation}\label{epsilon-behaviour}
\begin{split}\varepsilon(\Ind_{K_{\fp}/\Q_p}(\psi\chi^b))&=\varepsilon(\Ind_{K_{\fp}/\Q_p}(\psi_0\chi^b\psi^{-1}_{0}\psi))=\left(\frac{b}{p}\right)\varepsilon(\Ind_{K_{\fp}/\Q_p}(\psi_0\chi\psi^{-1}_0\psi))\\
&=\left(\frac{b}{p}\right)\varepsilon_p(\Ind_{K_{\fp}/\Q_{p}}(\psi\chi)),
\end{split}
\end{equation}
where for a $p$-adic character $\phi$ of $K_{\fp}^{\times}$ we denote by $\Ind_{K_{\fp}/\Q_p}$ the induction of the $p$-adic representation of $G_{K_{\fp}}$ associated to $\phi$.  
In particular, both $\Xi^{+}$ and $\Xi^{-}$ are infinite.
In the unramified case, for $\chi \in \Xi$ whose conductor $\fp^k$ is strictly contained in the conductor of $\psi_{0}/\psi$,
\cite[Prop.\ 7.4]{BKNO} implies that 
\[
\varepsilon(\Ind_{K_{\fp}/\Q_p}(\psi\chi))=\varepsilon(\Ind_{K_{\fp}/\Q_p}(\psi_0\chi\psi\psi_0^{-1}))=(-1)^{k},
\]
which completes the proof. 
\end{proof}

\begin{defn}\label{def, pm}\leavevmode 
\begin{enumerate}
\item 
Define $\rH^1_{\pm}(\Psi, T_{\psi}^{\otimes -1}(1)\otimes_{\CO} \Lambda_{\ac})$ as the subgroup of $\rH^{1}(\Psi, T_{\psi}^{\otimes -1}(1)\otimes_{\CO} \Lambda_{\ac})$ consisting of classes whose image under the composite 
\begin{equation}
\begin{split}
\rH^{1}(\Psi, T_{\psi}^{\otimes -1}(1)\otimes_{\CO} \Lambda_{\ac})& \to 
\rH^{1}(\Psi, T_{\psi}^{\otimes -1}(1)\otimes \Lambda_{\ac})\otimes_{\Lambda_{\ac}, \chi} \CO[\mathrm{Im}(\chi)] \to \rH^1_{}(\Psi, T_{\psi\chi}^{\otimes -1}(1) )\\ &\xrightarrow{\exp^*}  D_{\dR}^0(T_{\psi\chi}^{\otimes -1}(1)[1/p])
\end{split}
\end{equation}
equals $0 \in D_{\dR}^0(T_{\psi\chi}^{\otimes -1}(1)[1/p])$ for all $\chi\in \Xi^{\mp},$
where the last arrow is given by the dual exponential map.
\item For $n\ge 0$,
we define 
$\rH^1_{\pm}(\Psi_{n}, V_{\psi}^{\otimes -1}(1)) \subseteq\rH^1_{}(\Psi_{n}, V_{\psi}^{\otimes -1}(1))$
as the image of $\rH^1_{\pm}(\Psi, T_{\psi}^{\otimes -1}(1)\otimes_{\CO} \Lambda_{\ac})[1/p]$
under the composite 
\[
\rH^1(\Psi, T_{\psi}^{\otimes -1}(1)\otimes_{\CO} \Lambda_{\ac})[1/p] \to \rH^1(\Psi, T_{\psi}^{\otimes -1}(1)\otimes_{\CO} \Lambda_{\ac})\otimes_{\Lambda_{\ac}}\CO[1/p][\Gal(\Psi_n/\Psi)] \cong \rH^1(\Psi_n, V_{\psi}^{\otimes -1}(1)),
\]
where $V_{\psi}^{\otimes -1}:=T_{\psi}^{\otimes -1}[1/p]$, the last isomorphism arises from Shapiro's lemma and \cite[Prop.\ 2.12]{BKNO}. 
We define $\rH^1_{\pm}(\Psi_{n}, T_{\psi}^{\otimes -1}(1)) \subseteq\rH^1_{}(\Psi_{n}, T_{\psi}^{\otimes -1}(1))$
as the inverse image of $\rH^1_{\pm}(\Psi_{n}, V_{\psi}^{\otimes -1}(1))$
under the natural map $\rH^1_{}(\Psi_{n}, T_{\psi}^{\otimes -1}(1))\to  \rH^1_{}(\Psi_{n}, V_{\psi}^{\otimes -1}(1))$. 
\end{enumerate}
\end{defn}
\subsubsection{Rubin's conjecture}

\begin{thm}\label{higher-weight-decomposition}\leavevmode 
Suppose that Assumption \ref{generic} holds.
\begin{enumerate}
\item  The $\Lambda_{\ac}$-module $\rH^1_{\pm}(\Psi, T_{\psi}^{\otimes -1}(1)\otimes_{\CO} \Lambda_{\ac})$ is free of rank one, 
and we have a decomposition of $\Lambda_{\ac}$-modules
\[
\rH^1_{}(\Psi, T_{\psi}^{\otimes -1}(1)\otimes_{\CO} \Lambda_{\ac})=\rH^1_{+}(\Psi, T_{\psi}^{\otimes -1}(1)\otimes_{\CO} \Lambda_{\ac}) \oplus \rH^1_{-}(\Psi, T_{\psi}^{\otimes -1}(1)\otimes_{\CO} \Lambda_{\ac}).
\]
\item For $n\ge 1$, we have $\rH^1_{\pm}(\Psi, T_{\psi}^{\otimes -1}(1)\otimes_{\CO} \Lambda_{\ac})\otimes_{\Lambda_{\ac}}\CO[\Gal(\Psi_n/\Psi)]=\rH^1_{\pm}(\Psi_{n}^{}, T_{\psi}^{\otimes -1}(1)),$
and  
\[
\rH^1(\Psi_{n}^{}, T_{\psi}^{\otimes -1}(1))=\rH^1_+(\Psi_{n}^{}, T_{\psi}^{\otimes -1}(1))\oplus \rH^1_{-}(\Psi_{n}^{}, T_{\psi}^{\otimes -1}(1)).
\]
\end{enumerate}
\end{thm}
\begin{remark}\label{signed-condition-p-invert}
As is proved in \cite{BKNO},
without Assumption \ref{generic},
Theorem \ref{higher-weight-decomposition} still holds after inverting $p$.
\end{remark}
\begin{proof}
This is a consequence of \cite[Thm.\ 7.25 and Prop.\ 2.12]{BKNO} and 
Lemma \ref{density-of-xi}. 
\end{proof}
\subsection{Construction of local points}
In the rest of this section,
we assume that $\Psi/\Q_p$ is a \textit{ramified} quadratic field extension and that the prime ideal  $\fp$ of $\CO_{\Psi}$ exactly divides\footnote{Note that there is no  $\psi$ satisfying (\ref{local-self-dual}) whose conductor is $\CO_{\Psi}$.} the conductor of $\psi$. We also indicate how to relax the latter assumption. 

\subsubsection{Gaussian plus/minus cyclotomic polynomials}
Let $\delta \in \Psi$ be a uniformiser such that $\delta^2\in \Q_p$,
so that $-\delta^2=N_{\Psi/\Q_p}(\delta).$ (Note that $\delta$ is unique up to $\Z_p^{\times}$.)

In view of \cite[Cor.\ 7.8 ii)]{BKNO},
 there exists a system of finite order characters $(\chi_n)_{n\ge 1}$ of $\Gal(\Psi_{\infty}/\Psi)$ 
such that $\chi_n\in \Xi_n^+$  
and $\chi_{n+1}^{-\delta^2}=\chi_{n}$, which we fix. 
Enlarge $\CO$ so that $\CO[1/p]$ contains the quadratic extension of $\Q$ inside the $p$-th cyclotomic field $\Q(\mathrm{Im}(\chi_1))$. 
Fix a topological generator $\gamma$ of $\Gal(\Psi_{\infty}/\Psi)$. 
\begin{defn}
For $k \ge 1$, define the \textit{Gaussian plus/minus cyclotomic polynomials} 
\[
\Phi_{k}^{\pm}(\gamma)=\prod_{\chi \in \Xi_k^{\pm}}(\gamma-\chi(\gamma)) \in \Q(\mathrm{Im}(\chi_k))[\![\gamma]\!] \simeq \Q({\rm Im}(\chi_{k}))[\![\Gal(\Psi_{\infty}/\Psi)]\!].
\]
\end{defn}
\begin{lem}\label{half-polynomial}
For $k\ge 1$, we have $\Phi_{k}^{\pm}(\gamma) \in \CO[\gamma]$.
\end{lem}
\begin{proof}
 For a character $\chi=\chi_k^b \in \Xi_k$ with $b\in \Z_p^{\times}$,
note that $\chi \in \Xi_k^{\pm}$ if and only if $\displaystyle
\left(\frac{b}{p}\right)=\pm 1$ (cf.~\cite[Cor.\ 7.8 i)]{BKNO}).
Hence, 
we have
\begin{equation}\label{another-description-psi}
\Phi_{k}^{\pm}(\gamma)=\prod_{ b \in (\Z/p^k\Z)^{\times},  \left(\frac{b}{p}\right)=\pm 1}(\gamma-\chi_k^b(\gamma)). 
\end{equation}
It follows that $\Phi_{k}^{\pm}(\gamma) \in M[\gamma]$, where $M$ denotes the quadratic extension of $\Q$ inside $\Q(\mathrm{Im}(\chi_1))$.
Since the coefficients of $\Phi_{k}^{\pm}(\gamma)$ are algebraic integers, they lie in $\CO_M$, and the lemma follows from the assumption that $\CO_M \subseteq \CO.$
\end{proof}
For $n\ge 1$,
put
\begin{equation}\label{def-omega}
\omega_n^+=\prod_{0\le k\le n}\Phi_{k}^+(\gamma),\quad \omega_n^- =\prod_{0\le k\le n}\Phi_{k}^-(\gamma) \quad \in \CO[\gamma],
\end{equation}
where $\Phi_0^{\pm}$ is given by $$\Phi_0^{\varepsilon(\mathrm{Ind}_{\Psi/\Q_p}\psi )}=\gamma-1, \quad\Phi_0^{-\varepsilon(\mathrm{Ind}_{\Psi/\Q_p}\psi )}=1.$$
For $\chi \in \Xi_{\le n}$, note that\footnote{This notably differs from \cite[\S 2.3]{BKO24}, where the sign on $\omega_n^{\pm}(\gamma)$ is determined by the parity of the exponent of order of characters.}
$\chi\in \Xi^{\pm}$ if and only if $\chi(\omega_{n}^{\pm}(\gamma))=0.$ 

 \begin{defn}\label{def, Psi}
Put $d=p/(-\delta^2) \in \Z_p^{\times},$ and 
$\epsilon_{}(\Psi)=\left(\frac{d}{p}\right) \in \{\pm 1\}$,
which is independent of the choice of $\delta$.
\end{defn}
\begin{lem}\label{signed-cyc-polynomial-mod}
For $n\ge 0$, we have 
$\Phi_{n+1}^{\pm}(\gamma)\equiv \Phi_1^{\pm\epsilon(\Psi)^n}(1) \bmod (\gamma^{p^n}-1),$
where we note that $\epsilon(\Psi)^n=\left(\frac{d^n}{p}\right)\in \{\pm 1 \}.$
\end{lem}
\begin{proof}
Note that 
\[
\Phi_{n+1}^+(\gamma)\Phi_{n+1}^{-}(\gamma)=\frac{\gamma^{p^{n+1}}-1}{\gamma^{p^n}-1}=\Phi^+_1(\gamma^{p^n})\Phi^-_1(\gamma^{p^n}), 
\]
and for $a \in \Z_p^{\times}$, putting $X=\chi_{n+1}(\gamma)^a$, we have
\[
X^{p^n}=\chi_{n+1}(\gamma)^{p^na}=\chi_{n+1}(\gamma)^{a(-\delta^2)^{n}d^{n}}=\chi_1(\gamma)^{ad^{n}}. 
\]
Hence,  (\ref{another-description-psi}) implies that  $\Phi_{n+1}^{\pm}(\gamma)=\Phi_1^{\pm\epsilon(\Psi)^n}(\gamma^{p^n})$, 
concluding the proof.
 \end{proof}

The invariant $\Phi_1^{\pm}(1)$ has the following arithmetic interpretation, where
we regard the subfield $\Q(\chi_1(\gamma)) \subset \CO[1/p](\chi_1(\gamma))$
as the subfield $\Q(e^{2\pi i/p}) \subset \C$ by identifying $\chi_1(\gamma) $ with  $e^{2\pi i/p}$.
\begin{lem}\label{value-at-one}\leavevmode
\begin{enumerate}
\item If $p>3$ and $p  \equiv 3  \bmod 4 $,
then $$\Phi_1^{+}(1)=(-1)^{\frac{h(\Q(\sqrt{-p}))+1}{2}}\sqrt{-p},$$
where $\sqrt{-p}\in\CO$ denotes the square root identified with $i\sqrt{p}\in\C$ under $\chi_1(\gamma)\leftrightarrow e^{2\pi i/p}$, equivalently $\sqrt{-p}=\sum_{a\in \mathbb{F}_p^{\times}}\left(a/p\right)\chi_1(\gamma)^a$.
\item If $p  \equiv 1 \bmod 4 $,
then $$\Phi_1^{+}(1)=\sqrt{p} \cdot u_p^{-h(\Q(\sqrt{p}))/2},$$
where $\sqrt{p}=\sum_{a\in \mathbb{F}_p^{\times}}\left(a/p\right)\chi_1(\gamma)^a$ and $u_p=(a+b\sqrt{p})/2 \in \Z[(1+\sqrt{p})/2]^{\times}$, with $a,b$ the least positive integers satisfying $a^2-pb^2=4$.
\item We have $\Phi_1^{+}(1)\Phi_1^{-}(1)=p$; consequently $\Phi_1^{-}(1)=-\Phi_1^{+}(1)$ if $p\equiv 3\bmod 4$, and $\Phi_1^{-}(1)=\sqrt{p}\cdot u_p^{h(\Q(\sqrt{p}))/2}$ if $p\equiv 1\bmod 4$.
\end{enumerate}
\end{lem}
\begin{proof}
Part 1) follows from \cite[pp.\ 221--222]{L06}.

By \cite[p.\ 232]{L06}, 
we have\footnote{Our $\Phi_1^+(1)$ (resp.\ $\Phi_1^-(1)$) corresponds to $A(1)$ (resp.\ $B(1)$) and $u_p$ corresponds to $E(p)$ of \cite{L06} (cf.\ \cite[p. 340]{L05}).}
\[
h(\Q(\sqrt{p}))\log u_p = 2\log \frac{|\Phi_1^+(1)|^{-1}p}{\sqrt{p}}.
\]
and so $|\Phi_1^+(1)|=\sqrt{p} \cdot u_p^{-h(\Q(\sqrt{p}))/2}$.
Since $-1$ is a square in $\F_p^{\times}$,
we have 
\[
\Phi_1^+(1)=\prod_{1\le a \le  (p-1)/2,\ 
(a/p)=1}(1-e^{2a\pi i/p})(1-e^{-2a\pi i/p})=\prod_{1\le a \le  (p-1)/2,\ 
(a/p)=1} \left(2-2\cos(2a\pi /p)\right)>0,
\]
yielding part 2). 
Part 3) follows from
$$\Phi_1^{+}(1)\Phi_1^{-}(1)=\prod_{b=1}^{p-1}(1-\chi_1(\gamma)^b)=\Phi_p(1)=p,$$
where the first equality is (\ref{another-description-psi}) evaluated at $\gamma=1$ and $\Phi_p$ denotes the $p$-th cyclotomic polynomial; combined with parts 1) and 2) (and, in the case $p\equiv 3\bmod 4$, the fact that $h(\Q(\sqrt{-p}))$ is odd), this gives the stated values of $\Phi_1^{-}(1)$.
\end{proof}
\subsubsection{Plus/minus local points}
Fix a $\Lambda_{\ac}$-basis $v^{\pm}$ of $\rH^1_{\pm}(\Psi, T_{\psi}^{\otimes -1}(1)\otimes_{\CO} \Lambda_{\ac})$.
\begin{defn}
We define plus/minus local points by 
\begin{equation}\label{local-point-def}
c_n^{\pm}=\omega_n^{\pm}v^{\pm}_{ n} \in \rH^1_{\pm}(\Psi_n, T_{\psi}^{\otimes -1}(1)),
\end{equation}
where $v^{\pm}_{ n}$ denotes the image of $v^{\pm}_{}$ under 
$
\rH^1_{\pm}(\Psi, T_{\psi}^{\otimes -1}(1)\otimes \Lambda_{\ac}) \to
 \rH^1_{\pm}(\Psi_n, T_{\psi}^{\otimes -1}(1)).
$
\end{defn}
\begin{prop}\label{finiteness-of-c}
For $n\ge 0$,
we have 
\[
c_n^{\pm} \in \rH^1_{\rf}(\Psi_n, T_{\psi}^{\otimes -1}(1))
:= \mathrm{Ker}\left(\rH^1_{}(\Psi_n, T_{\psi}^{\otimes -1}(1))\to \frac{\rH^1_{}(\Psi_n, V_{\psi}^{\otimes -1}(1))}{\rH^1_{\rf}(\Psi_n, V_{\psi}^{\otimes -1}(1))}\right),
\]
where $V_{\psi}^{\otimes -1}:=T_{\psi}^{\otimes -1}[1/p]$, and
$\rH^1_{\rf}(\Psi_n, V_{\psi}^{\otimes -1}(1))$ denotes the Bloch--Kato subgroup.
\end{prop}
\begin{proof}
It suffices to show that: for a character $\chi$ of $ \Gal(\Psi_n/\Psi)$, 
the image of $c_n^{\pm}$ under the composite 
\[
\rH^1_{}(\Psi_n, T_{\psi}^{\otimes -1}(1)) \to \rH^1_{}(\Psi_n, V_{\psi}^{\otimes -1}(1))\otimes_{\CO[\Gal(\Psi_n/\Psi)], \chi}\CO[\mathrm{Im}(\chi)] =\rH^1_{}(\Psi, V_{\psi\chi}^{\otimes -1}(1))
\]
lies in $\rH^1_{\rf}(\Psi, V_{\psi\chi}^{\otimes -1}(1))$.
If $\chi \in \Xi^{\mp}$, then this follows from the inclusion $c_n^{\pm}\in \rH^1_{\pm}(\Psi_n, T_{\psi}^{\otimes -1}(1))$,
and if $\chi \in \Xi^{\pm}$, it
follows from the vanishing $\chi(\omega^{\pm}_n)=0$.
\end{proof}
\begin{prop}\label{norm-relation-of-c}
For $n\ge 0$,
we have $$\mathrm{Cor}_{n+1/n}c^{\pm}_{n+1}=\Phi_1^{\pm\epsilon(\Psi)^n}(1)c_n^{\pm},$$
where $\mathrm{Cor}_{n+1/n}: \rH^1_{}(\Psi_{n+1}, T_{\psi}^{\otimes -1}(1))\to \rH^1_{}(\Psi_{n}, T_{\psi}^{\otimes -1}(1))$ denotes the corestriction map.
\end{prop}
\begin{proof}
Under Shapiro's lemma, note that 
 $\mathrm{Cor}_{n+1/n}$ is identified with the map
 \[
 \rH^1_{}(\Psi, T_{\psi}^{\otimes -1}(1)\otimes_{\CO}\CO[\Gal(\Psi_{n+1}/\Psi)])\to \rH^1_{}(\Psi, T_{\psi}^{\otimes -1}(1)\otimes_{\CO}\CO[\Gal(\Psi_n/\Psi)])
 \]
induced by the natural projection $\CO[\Gal(\Psi_{n+1}/\Psi)] \to \CO[\Gal(\Psi_{n}/\Psi)]$. 
So we have 
\begin{equation}\label{naive-norm}
\mathrm{Cor}_{n+1/n}c^{\pm}_{n+1}=\Phi_{n+1}^{\pm}(\gamma)c_n^{\pm} 
\end{equation}
and Lemma \ref{signed-cyc-polynomial-mod} concludes the proof. 
\end{proof}
\begin{remark}\label{norm-c-frak}If we define \[
\mathfrak{c}_n^{\pm}
=\begin{cases}
\prod_{1\le k\le n-1}(\Phi_1^{\pm\epsilon(\Psi)}(1)/\Phi_1^{\pm\epsilon(\Psi)^k}(1))c_n^{\pm} & (n\ge 2),\\
c_n^{\pm} & (n\le 1),
\end{cases}
\]
then $\mathrm{Cor}_{n+1/n}\fc^{\pm}_{n+1}=\Phi_1^{\pm\epsilon(\Psi)}(1)\fc_n^{\pm}.$
For each $k$, the ratio $\Phi_1^{\pm\epsilon(\Psi)}(1)/\Phi_1^{\pm\epsilon(\Psi)^k}(1)$ is a unit in $\CO$: it equals $1$ if $\epsilon(\Psi)^k=\epsilon(\Psi)$, and otherwise equals $p/\Phi_1^{\pm\epsilon(\Psi)}(1)^2$, which is a unit since $\Phi_1^{+}(1)^2=-p$ if $p\equiv 3\bmod 4$ and $\Phi_1^{+}(1)^2=p\cdot u_p^{-h(\Q(\sqrt{p}))}$ (a unit multiple of $p$) if $p\equiv 1\bmod 4$, by Lemma~\ref{value-at-one}. Hence $\mathfrak{c}_n^{\pm}$ generates the same $\CO[\Gal(\Psi_n/\Psi)]$-module as $c_n^{\pm}$, and Theorem~\ref{c-generates-fpm} shows that this common system generates $\rH^1_{\rf,\pm}(\Psi_n, T_{\psi}^{\otimes-1}(1))$ as well. Thus the single system $(\mathfrak{c}_n^{\pm})_n$ satisfies both the trace relation above and the generation property of Theorem~\ref{c-generates-fpm}, as asserted in Theorem~\ref{intro-local-result}.
\end{remark}

\subsection{$c_n^{\pm}$ as generators}
\subsubsection{Local sign decomposition for induced representations}\label{ss, lsd-ind} In preparation for the main result of this section (cf.~Theorem \ref{c-generates-fpm}), we recall some results of \cite{BKNO} which underlie Theorem~\ref{higher-weight-decomposition}.

For $n \ge 0$,
put 
\[
T_n=\mathrm{Ind}_{\Psi_n/\Psi}(T_{\psi}^{\otimes -1}(1))=T_{\psi}^{\otimes -1}(1)\otimes_{\CO}\CO[\Gal(\Psi_n/\Psi)],\quad \tilde{T}_n=\mathrm{Ind}_{\Psi/\Q_p}(T_n). 
\]
Then for $\chi \in \Xi_{\le n}:=\bigcup_{0\le k \le n}\Xi_{k}$,
the $\CO_{\chi}[G_{\Q_p}]$-representation 
\[
\tilde{T}_{\chi}:=\tilde{\bT}_{\psi}\otimes_{\Lambda_{\ac},\chi}\CO_{\chi}=\tilde{T}_n\otimes_{\CO[\Gal(\Psi_n/\Psi)],\chi^{}}\CO_{\chi}=\Ind_{\Psi/\Q_p}(T_{\psi\chi}^{\otimes -1}(1))
\]
is symplectic self-dual,
where $\CO_{\chi}$ denotes the integer ring of a  finite extension of $\CO[1/p]$
containing the image of $\chi$. 

\begin{thm}\label{lsd-ind} There exist canonical $\CO_{\chi}$-submodules $\rH^1_{\pm}(\Q_p, \tilde{T}_{\chi})$ of $\rH^1(\Q_p, \tilde{T}_{\chi})$ with the following properties. 
\begin{enumerate}
\item The $\CO_{\chi}$-submodules $\rH^1_{\pm}(\Q_p, \tilde{T}_{\chi}) \subset \rH^1(\Q_p, \tilde{T}_{\chi})$ are maximal isotropic submodules with respect to the symmetric Tate pairing on $\rH^1(\Q_p, \tilde{T}_{\chi})$.
\item As $\CO_\chi$-modules, we have 
$$\rH^1(\Q_p, \tilde{T}_{\chi})=\rH^1_{+}(\Q_p, \tilde{T}_{\chi}) \oplus \rH^1_{-}(\Q_p, \tilde{T}_{\chi}).$$
\item If $\chi \in \Xi^{\pm}$, then $\rH^1_{\mp}(\Q_p, \tilde{T}_{\chi})$ coincides with  $\rH^1_{\rf}(\Q_p, \tilde{T}_{\chi})$.
\item Under $\rH^1_{}(\Psi, T_{n})\otimes_{\CO[\Gal(\Psi_n/\Psi)],\chi}\CO_{\chi}\cong \rH^1_{}(\Q_p, \tilde{T}_{\chi})$ (cf.~\cite[Prop.\ 2.22]{BKNO}), we have 
 $$\rH^1_{\pm}(\Psi, T_{n})\otimes_{\CO[\Gal(\Psi_n/\Psi)],\chi}\CO_{\chi}=\rH^1_{\pm}(\Q_p, \tilde{T}_{\chi}).$$

\end{enumerate}
\end{thm}
\begin{proof}This is a special case of \cite[Thm.\ 1.3]{BKNO}.  
\end{proof}
Put
$\tilde{V}_{\chi}=\tilde{T}_{\chi}[1/p]$. 
Then we also have $\rH^1_{\pm}(\Q_p, \tilde{V}_{\chi})\subseteq \rH^1_{}(\Q_p, \tilde{V}_{\chi})$ with properties as in Theorem~\ref{lsd-ind},
and for  $\bullet \in \Set{+, -, \rf}$,
\begin{equation}\label{saturated}
\rH^1_{\bullet}(\Psi_n, T_{\psi}^{\otimes -1}(1))=\Set{x \in \rH^1_{}(\Psi_n, T_{\psi}^{\otimes -1}(1)) | \chi(x) \in \rH^1_{\bullet} (\Q_p, \tilde{V}_{\chi}) \ (\chi \in \Xi_{\le n}^{}) },
\end{equation}
where  $\chi(x)$ denotes the image of $x\in \rH^1_{}(\Psi_n, T_{\psi}^{\otimes -1}(1))$ under  the map
\[
 \rH^1(\Psi_n, T_{\psi}^{\otimes -1}(1))\to \rH^1(\Psi_n, T_{\psi}^{\otimes -1}(1))\otimes_{\CO[\Gal(\Psi_n/\Psi )],\chi}\CO_{\chi}[1/p]=\rH^1(\Q_p, \tilde{V}_{\chi}).
\]
For $\bullet=\rf$, (\ref{saturated}) is immediate from the definition of the Bloch--Kato subgroup; for $\bullet=\pm$, it follows from Theorem~\ref{lsd-ind} and Theorem~\ref{higher-weight-decomposition}, identifying $\rH^1_{\pm}(\Psi_n, T_{\psi}^{\otimes -1}(1))\otimes\Q_p$ with $\rH^1_{\pm}(\Q_p, \tilde{T}_n)\otimes \Q_p$ via Shapiro's lemma and noting that the quotient of $\rH^1(\Psi_n, T_{\psi}^{\otimes -1}(1))$ by each side of (\ref{saturated}) is $p$-torsion-free.
\begin{defn}
 The plus/minus subgroups of the Bloch--Kato subgroups $\rH^1_{\rf}(\Psi_{n}, T_{\psi}^{\otimes -1}(1))$ are defined by 
\[
\begin{split}
\rH^1_{\rf, \pm}(\Psi_{n}, T_{\psi}^{\otimes -1}(1))=&\Set{x \in \rH^1_{\rf}(\Psi_{n}, T_{\psi}^{\otimes -1}(1)) | \chi(x)=0 \ \   (\chi \in \Xi_{\le n}^{\pm} )}
\\
=&\rH^1_{\rf}(\Psi_{n}, T_{\psi}^{\otimes -1}(1)) \cap \rH^1_{\pm}(\Psi_{n}, T_{\psi}^{\otimes -1}(1)),
\end{split}
\]
where $\Xi_{\le n}^{\pm}:=\Xi_{\le n} \cap \Xi^{\pm}$, and the second equality follows from the above properties of $\rH^1_{\pm}$  and (\ref{saturated}).
\end{defn}
\subsubsection{}
\begin{thm}\label{c-generates-fpm}
\leavevmode
\begin{enumerate}
\item The element $c_n^{\pm}$ lies in $\rH^1_{\rf, \pm}(\Psi_n, T_{\psi}^{\otimes -1}(1))$ and generates it over $\CO[\Gal(\Psi_n/\Psi)].$
\item As $\CO[\Gal(\Psi_n/\Psi)]$-modules, 
we have 
\begin{equation}\label{direct-sum-finite-part}
 \rH^1_{\rf}(\Psi_n, T_{\psi}^{\otimes -1}(1)) =\rH^1_{\rf, +}(\Psi_n, T_{\psi}^{\otimes -1}(1)) \oplus\rH^1_{\rf, -}(\Psi_n, T_{\psi}^{\otimes -1}(1))
\end{equation}
\end{enumerate}
\end{thm} 
\begin{proof}
In view of (\ref{saturated}), we have 
 \begin{equation}\label{description-f-pm}
\begin{split}
&\rH^1_{\rf, \pm}(\Psi_{n}, T_{\psi}^{\otimes -1}(1))\\
=&\Set{x \in \rH^1_{}(\Psi_{n}, T_{\psi}^{\otimes -1}(1)) | \chi(x)=0 \ (\chi \in \Xi_{\le n}^{\pm}),  \ \chi(x) \in \rH^1_{\rf}(\Q_p, \tilde{V}_{\chi})  \   (\chi \in \Xi_{\le n}^{\mp} )}\\
=&\Set{x \in \rH^1_{}(\Psi_{n}, T_{\psi}^{\otimes -1}(1)) | \chi(x)=0 \ (\chi \in \Xi_{\le n}^{\pm}), \ \chi(x) \in \rH^1_{\pm}(\Q_p, \tilde{V}_{\chi})  \   (\chi \in \Xi_{\le n}^{\mp} )}\\
=&\Set{x \in \rH^1_{}(\Psi_{n}, T_{\psi}^{\otimes -1}(1)) | \chi(x)=0 \ (\chi \in \Xi_{\le n}^{\pm}), \ \chi(x) \in \rH^1_{\pm}(\Q_p, \tilde{V}_{\chi})  \   (\chi \in \Xi_{\le n}^{} )}\\
=&\Set{x \in \rH^1_{\pm}(\Psi_{n}, T_{\psi}^{\otimes -1}(1)) | \chi(x)=0 \ (\chi \in \Xi_{\le n}^{\pm})}\\
=&\CO[\Gal(\Psi_n/\Psi)] \cdot \omega_{n}^{\pm}v_n^{\pm}. 
\end{split} 
\end{equation}
Here the second equality follows from the properties of $\rH^1_{\pm}(\Q_p, \tilde{V}_{\chi})$ in \S\ref{ss, lsd-ind},
and  the last equality from the fact that $\rH^1_{\pm}(\Psi_{n}, T_{\psi}^{\otimes -1}(1))$ is free of rank one over $\CO[\Gal(\Psi_n/\Psi)]$ and generated by $v_n^{\pm}$, and $\rH^1_{}(\Psi_{n}, T_{\psi}^{\otimes -1}(1))/\rH^1_{\pm}(\Psi_{n}, T_{\psi}^{\otimes -1}(1))$ is also $\CO[\Gal(\Psi_n/\Psi)]$-free (cf.~Theorem~\ref{higher-weight-decomposition}).  Since $c_{n}^{\pm}=\omega_{n}^{\pm}v_n^{\pm}$, part 1) follows. 

As for part 2), we first note that the equality holds after inverting $p$. 
In view of (\ref{saturated}) and (\ref{description-f-pm}),
the quotient of $\rH^1(\Psi_n, T_{\psi}^{\otimes -1}(1))$ by the left or right hand side of 
\eqref{direct-sum-finite-part}
is $p$-torsion-free. 
Hence, part 2) follows.
\end{proof}

\section{Elliptic unit main conjectures}\label{section-Selmer-group} 
In this section we develop some preliminaries on elliptic units and $p$-adic $L$-functions needed for the formulation and proof of the plus/minus Iwasawa main conjectures. The main result is the anticyclotomic elliptic unit main conjecture of Theorem~\ref{anticyc-zetaIMC}, which is deduced from Rubin's work and underlies subsequent sections. 
\subsection{Notation}\label{ss, nt3}
\subsubsection{} 
Let $p\ge 3$ be a prime.
Let $K$ be an imaginary quadratic field such that
$p$ does not divide its class number $h_K$ and
\begin{equation*}\label{n-spl}\tag{n-spl}
\text{ $p$ does \textit{not} split in $K$.} 
 \end{equation*}
 Let $\fp$ denote the unique prime of $K$ above $p$. 
Fix an embedding $\iota: K\hookrightarrow \C$ and denote by $\overline{\Q}$ the algebraic closure of $K$ inside $\C$ via $\iota$.
We regard abelian extensions of $K$ as subfields of $\overline{\Q} \subset \C$.

Let $K_{\infty}/K$ be the $\Z_p^{2}$-extension and $K_{\infty}^{\ac}$ the anticyclotomic $\Z_p$-extension. 
Let $K_n^{\rm ac}$ denote the $n$-th layer of $K_{\infty}^{\ac}/K$ and $K_{n,p}^{\ac}:=K_{n}^{\ac}\otimes_K K_{\fp}$. 
For a non-zero ideal $\fg$ of $\CO_K$, denote by $K(\fg)$ the ray class field of $K$ of conductor $\fg$.
For a finitely generated $\Z_p$-module $T$  with continuous action of $G_{K}:=\Gal(\overline{\Q}/K)$,
put
\[
\bfH^1_{\fg}(T)=\varprojlim_{n}\rH^1(K(p^n\fg), T),\quad \bfH^1_{\fg}(T\otimes \Q_p)=\bfH^1_{\fg}(T) \otimes \Q_{p}.
\]

\subsubsection{} For an algebraic  extension $F$ of $K$ and a prime $v$ of $F$,
we put $F_{v}=\bigcup_{K\subseteq M \subseteq F}M_v$,
where $M$ ranges over all finite extensions of $K$ inside $F$, and $M_v$ denotes the completion at the prime of $M$ below $v$. 

Put 
\[
\begin{split}
&\rH^1_{\rf}(F_{v}, T\otimes \Q_p)=
\begin{cases}
\ker\left(\rH^1(F_v, T\otimes \Q_p) \to \rH^1(F_v^{\ur}, T\otimes \Q_p)\right) & (v\nmid p),\\
\ker\left(\rH^1(F_v, T\otimes \Q_p) \to \rH^1(F_v^{}, T\otimes_{\Q_p} B_{\cris})\right) & (v\mid p),
\end{cases}\\
&\rH^1_{\rel}(F_{v}, T\otimes \Q_p)=\rH^1_{}(F_{v}, T\otimes \Q_p),\\
&\rH^1_{\str}(F_{v}, T\otimes \Q_p)=\{0\},
\end{split}
\] 
and denote by $\rH^1_{\bullet}(F_{v}, T) \subseteq \rH^1_{}(F_{v}, T) $ the inverse image of $\rH^1_{\bullet}(F_{v}, T\otimes \Q_p)$
under the natural map $\rH^1_{}(F_{v}, T)\to \rH^1_{}(F_{v}, T\otimes \Q_p)$.
Moreover,  denote by $\rH^1_{\bullet}(F_{v}, T\otimes \Q_p/\Z_p)$ the image of
$\rH^1_{\bullet}(F_{v}, T\otimes \Q_p)$ under the natural map
$\rH^1_{}(F_{v}, T\otimes \Q_p)\to \rH^1_{}(F_{v}, T\otimes \Q_p/\Z_p)$.

For $M \in \Set{T\otimes \Q_p/\Z_p, T}$  and $\bullet \in \Set{\rel, \str, \rf}$,
we define Selmer groups 
\[
\begin{split}
\Sel_{\bullet}(F,M)&=\ker\left(\rH^1(F, M) \to \frac{\rH^1(F_{p}, M)}{\rH^1_{\bullet}(F_{p}, M)} \times \prod_{v\nmid p}\frac{\rH^1(F_v, M)}{\rH^1_{\rf}(F_v, M)}\right),
\end{split}
\]
where $v$ ranges over all places of $F$ not dividing $p$,
and $ \rH^1_{\bullet}(F_{p}, -):=\prod_{v\mid p}\rH^1_{\bullet}(F_{v}, -)$.
We also put
$\Sel_{}(F,M)=\Sel_{\rf}(F,M).$

\subsubsection{Main conjecture}
Let  $K_{\infty}^{\ac}$ be the anticyclotomic $\Z_p$-extension of $K$. 
Let $\varphi$ be an algebraic Hecke character of $K$
of conductor $\ff$
 whose infinity  type is given by a fixed $\iota: K\hookrightarrow \C$
such that 
\begin{equation}\label{self-dual-character}
\varphi|_{\mathbb{A}_{\Q}^{\times}}=\omega_{K/\Q}|\cdot |_{\mathbb{A}_{\Q}^{\times}},
\end{equation}
where $\mathbb{A}_{\Q}^{\times}$ denotes the idele group of $\Q$, 
 $|\cdot |_{\mathbb{A}_{\Q}^{\times}}$ denotes the norm map,
 and $\omega_{K/\Q}$ denotes the quadratic character associated to $K/\Q$. 
Let $\CO$ be the integer ring of a sufficiently large finite extension of $\Q_p$,
and $V=V_{\varphi}$ the $\CO[1/p]$-representation of $G_{K}$ associated to $\varphi$
and $T$   an $\CO$-lattice of $V$ (see \S \ref{subsection-setup} for details). 

We list the hypotheses that will be assumed in parts of the paper. 

\begin{assumption}\label{for-integrality}\leavevmode
\begin{enumerate}
\item If  $K_p:=K\otimes_{\Q} \Q_{p}\cong \Q_3(\sqrt{-3})$ (in particular, $p=3$), then $\rH^0(K_p, T^{\otimes -1}(1)\otimes \CO/\fm )=\{0\}$,
where $T^{\otimes -1}=\Hom_{\CO}(T, \CO)$.
\item There exists a finite abelian extension $H$ of $K$ containing the Hilbert class field of $K$, unramified outside $\ff$ 
such that $p\nmid [H:K]$ and
$\varphi \circ N_{H/K}$ is the Hecke character associated to $A$ over $H$, taking a model of $A$ over $H$.  
\end{enumerate}
\end{assumption}
\begin{remark}
A typical example satisfying (2) is given by a canonical Hecke character, and in this case $H$ may be taken to be the Hilbert class field.     
\end{remark}

\subsection{Selmer groups}\label{subsection-setup}
\subsubsection{Galois representations} 

 Let $L=L_{\varphi}$ be the subfield of $\C$ generated by $\varphi(\widehat{K}^{\times})$ over $K$, where we regard $\varphi$ as a $\C$-valued character of the idele class group of $K$.
 In the following, we fix a prime $\lambda$ of $\overline{\Q}$ dividing $p\CO_K$, 
and let $\lambda$ also denote the  place of $L$ below $\lambda$.
Let $\fg$ be a non-zero integral ideal of $K$  contained in $\ff$,
 and $(A, \alpha)$ be the canonical CM pair of modulus $\fg$ over $K(\fg)$ in the sense\footnote{To use arguments in \cite[\S 15]{K}, note that the natural map $\CO_K^{\times} \to (\CO_K/\ff)^{\times}$ is always injective:  
for $a \in \CO_K$ with $a\equiv 1 \bmod \ff$ we have $\varphi(a \CO_K)=a$ when $\varphi$ is regarded as a function on the fractional ideals.} of \cite[\S 15]{K}. 
 Then $(A, \alpha)$ is canonically isomorphic to the CM pair $(\C/\fg, 1\bmod \fg)$ over $\C$. 
 
Put \[
V(\varphi)=V_{L_{\lambda}}(\varphi)=\rH^1_{\et}(A\otimes_{K(\fg)}\overline{\Q}, \Q_p)\otimes_{K\otimes \Q_p}L_{\lambda}. 
\]
Then, as in \cite[15.8]{K},
 $V_{L_{\lambda}}(\varphi)$ is a one-dimensional $L_{\lambda}$-space naturally equipped with $L_{\lambda}$-linear $\Gal(K(p^{\infty}\ff)/K)$-action such that 
for a non-zero ideal $\fa$ of $\CO_K$ relatively prime to $p\ff$,
the Artin symbol $(\fa, K(p^{\infty}\ff)/K )$ acts by the multiplication by $\varphi(\fa)^{-1}$. 
Here $\varphi$ is regarded as a character of the group of fractional ideals,
and the reciprocity map $(-, K(p^{\infty}\ff)/K)$ is normalised so that it sends a prime ideal to the arithmetic Frobenius. 
Note that $V_{L_{\lambda}}(\varphi)$ is independent of the choice of $\fg$ and $(A,\alpha)$ (cf.~\cite[pp.\ 257--258]{K}). 

Put
\begin{equation}\label{our-V}
V =V_{\varphi}=V_{\varphi, \lambda}
=V_{L_{\lambda}}(\varphi)^{\otimes -1}:=\Hom_{L_{\lambda}}(V_{L_{\lambda}}(\varphi), L_{\lambda}).
\end{equation}
Denote by $\CO$ the integer ring of $L_{\lambda}$ and by $\fm$ the maximal ideal. 
Fix an $\CO$-lattice $T$ of $V$. 

Under Assumption \ref{for-integrality} (2), note that  $T|_{G_H}$ is isomorphic to $T_p(A) \otimes_{\CO_K\otimes \Z_p} \CO$ as an $\CO[G_H]$-module.  

\begin{lem}\label{non-anomalous}
Suppose that Assumption \ref{for-integrality} (1) holds.
For a finite $p$-extension $F$ of $K_{\fp}$,
we have $H^0(F, T^{\otimes -1}(1)/\fm_{})=0.$
\end{lem}

\begin{proof}
We may assume that $K\otimes\Q_p$ is not isomorphic to $\Q_3(\sqrt{-3})$.
It suffices to show that $H^0(K_{\fp}, T^{\otimes -1}(1)/\fm_{})=\{0\},$
which follows from  \cite[Lem.\ 7.22]{BKNO}.
\end{proof}
\begin{lem}\label{rel-gal}For a finite extension $F$ of $K$ unramified outside $p\ff$, we have 
$$\Sel_{\rel}(F, T^{\otimes -1}(1))=\rH^1(K_S/F, T^{\otimes -1}(1)),$$
where $K_S$ denotes the maximal extension of $K$ unramified outside $p\ff$.
\end{lem}
\begin{proof}
For $v\nmid p$,
 we have  
 \begin{equation}\label{local-vanishing}
\rH^1_{}(F_v, T^{\otimes -1}(1)\otimes \Q_{p})=0
 \end{equation} (cf.\ \cite[Cor.\ 1.3.3]{Ru00}),
and hence 
\begin{equation}\label{vanishing-ramified}
\rH^1_{\rf}(F_v, T^{\otimes -1}(1))=\rH^1_{}(F_v, T^{\otimes -1}(1)).
\end{equation}
For $v\mid p$,  we also have  $\rH^1_{\rel}(F_v, T^{\otimes -1}(1))=\rH^1_{}(F_v, T^{\otimes -1}(1))$ by definition. 
Hence, \cite[Lem.\ 1.5.3]{Ru00} concludes the proof.
\end{proof}

For simplicity of notation, 
also denote by $\varphi: \Gal(K(p^{\infty}\ff )/K) \to \CO^{\times}$ the character corresponding to the Galois representation $T$,
which may be regarded as the $p$-adic avatar of $\varphi$.

\subsubsection{Selmer groups}
For $* \in \Set{\ac, \emptyset}$ and $\bullet \in \Set{\emptyset, \str, \rel}$,
put $\Lambda_{*}=\CO[\![\Gal(K_{\infty}^{*}/K )]\!]$
 and 
\[
\scrS_{\bullet}^{*}=\varprojlim_{K\subseteq F \subseteq K_{\infty}^{*}}\Sel_{\bullet}(F, T^{\otimes -1}(1)), \quad X_{\bullet}^{*}=
\left(\varinjlim_{K\subseteq F \subseteq K_{\infty}^{*}}
 \Sel_{\bullet}(F, W)\right)^{\vee},
\]
which naturally have a structure of $\Lambda_{*}$-module.
Here $F$ ranges over all finite extensions of $K$ inside $K_{\infty}^{*},$
and $M^{\vee}:=\Hom_{\CO}(M, L_{\lambda}/\CO).$

\begin{lem}\label{finiteness-at-ramified}
For an algebraic extension $F$ of $K$ inside $K_{\infty}$ and a prime $v\mid \ff$ of $F$ with $v\nmid p$, 
$\rH^0(F_v, T^{\otimes -1}(1)\otimes \Q_p/\Z_p)$
is finite.
\end{lem}
\begin{proof}
Since the representation $T^{\otimes -1}(1)\otimes \Q_p/\Z_p$
of $G_{F_v}$ is ramified and $F_v/K_v$ is unramified,
we have $\rH^0(F_v, T^{\otimes -1}(1)\otimes \Q_p/\Z_p)\subsetneq T^{\otimes -1}(1)\otimes \Q_p/\Z_p$.
Noting that $T^{\otimes -1}(1)\otimes \Q_p/\Z_p \cong L_{\lambda}/\CO$
has no non-trivial $\CO$-submodule of infinite order, this concludes the proof. 
\end{proof}

\begin{prop}\label{rank-rel}
For $* \in \Set{\emptyset, \ac}$,
the $\Lambda_{*}$-module $\scrS_{\rel}^*$ is torsion-free, and 
$\rank_{\Lambda_*}(\scrS_{\rel}^*)=\rank_{\Lambda_*}(X_{\rel}^*).$
\end{prop}
\begin{proof}
By Shapiro's lemma and Lemma \ref{rel-gal}, we have $\scrS_{\rel}^{*}\cong \rH^1(K_S/K, T^{\otimes -1}(1)\otimes \Lambda_{*}).$
In the following, we shall prove the torsion-freeness of $\rH^1(K_S/K, T^{\otimes -1}(1)\otimes \Lambda_{*})$ by a similar argument to that in \cite[13.8]{K}.

It suffices to show that for $f\in \Lambda_{*}\setminus\{0\}$, we have 
\begin{equation}\label{for-torsionfree}
\rH^0(K, T^{\otimes -1}(1)\otimes \Lambda_*/f\Lambda_*)=0. 
\end{equation}
Since  there exist $g\in \Lambda_*$  and $n\ge 0$ such that $f=p^ng$ and $\Lambda_{*}/g\Lambda_*$ is $p$-torsion-free, we may assume that either $\Lambda_*/f\Lambda_*$ is $p$-torsion-free or $f=p$. 

Let us first consider the case $f=p$.
For each element $a \in \CO_K$  such that $ a \equiv 1 \bmod p\ff$,
$(a, K(\ff p^{\infty})/K) \in \Gal(K(p^{\infty}\ff)/K)$
acts on $T^{\otimes -1}(1)\otimes \Lambda_*/p$ by $ [(a, K_{\infty}^*/K)] \bmod p \in (\Lambda_*/p)^{\times},$
where $[ \ ]$ denotes the associated group-like element.
Hence, (\ref{for-torsionfree}) follows. 
Let us next assume that $\Lambda_*/f\Lambda_*$ is $p$-torsion-free. 
Since the Galois action on $T^{\otimes -1}(1)[1/p]$ cannot factor through $\Gal(K_{\infty}^{\ac}/K)$,
(\ref{for-torsionfree}) holds for $ * = \ac$. 
Moreover,  the image of the character $G_K \to ((\Lambda_*/f\Lambda_*)[1/p])^{\times}$ associated to the $G_K$-action on $T^{\otimes -1}(1)\otimes(\Lambda_*/f\Lambda_*)[1/p]$
is not contained in $\CO[1/p]^{\times}$, and so (\ref{for-torsionfree}) holds for $* = \emptyset $.

It remains to prove $\rank_{\Lambda_*}(\scrS_{\rel}^*)=\rank_{\Lambda_*}(X_{\rel}^*).$

By (\ref{local-vanishing}), for a place $v\nmid p$ of $K$ and a finite extension $F\subseteq K_{\infty}^*$ of $K$,
we have $\rH^1_{\rf}(F\otimes K_v, W)=0$.
For a place $v\mid p$ of $K$,
 $\rH_{\rel}^1(F\otimes K_v, W)=\rH^1(F\otimes K_v, W)_{\div},$
 where $\div$ denotes the maximal divisible part.
Then we have a short exact sequence
\begin{equation}\label{sel-coh}
0\to \Sel_{\rel}(F, W) \to \rH^1(K_S/F, W) \to \bigoplus_{v\mid \ff^{(p)}} \rH^1(F_v, W) \oplus \frac{\rH^1(F_{\fp}, W)}{\rH^1(F_{\fp}, W)_{\div}},
\end{equation}
where $\ff^{(p)}$ denotes the prime-to-$p$ part of $\ff$, 
and $v$ ranges over all primes of $K$ dividing $\ff^{(p)}$.  
By (\ref{local-vanishing}) and the local Tate duality, for $v\nmid p$ we have 
\[
\rH^1(F_v, W)=\rH^2(F_v, T)=\rH^0(F_v, T^{\otimes -1}(1)\otimes \Q_p/\Z_p)^{\vee},
\]
and 
\[
\frac{\rH^1(F_{\fp}, W)}{\rH^1(F_{\fp}, W)_{\div}}=\rH^2(F_{\fp}, T)_{\tor}=\left(\frac{\rH^0(F_{\fp}, T^{\otimes -1}(1)\otimes\Q_p/\Z_p)}{\rH^0(F_{\fp}, T^{\otimes -1}(1)\otimes\Q_p/\Z_p)_{\div}}\right)^{\vee}=\rH^0(F_{\fp}, T^{\otimes -1}(1)\otimes\Q_p/\Z_p)^{\vee},
\]
where the last equality follows from the finiteness of $\rH^0(F_{\fp}, T^{\otimes -1}(1)\otimes\Q_p/\Z_p)$ (cf.~Lemma~\ref{finiteness-at-ramified}).

 Hence, taking the dual of the limit of (\ref{sel-coh}), we obtain the following exact sequence of $\Lambda_{*}$-modules 
\[
 \bigoplus_{v\mid p\ff }\varprojlim_{K\subseteq F\subseteq K_{\infty}^{*}}\rH^0(F\otimes_{K} K_v,
  T^{\otimes -1}(1)\otimes \Q_p/\Z_p)  \to \rH^1(K_S/K_{\infty}^{*}, W)^{\vee} \to X_{\rel}^*\to 0.
\]
Since 
for $v\mid p\ff$,
the $\Lambda_{*}$-module $\varprojlim_{K\subseteq F\subseteq K_{\infty}^{*}}\rH^0(F\otimes_{K} K_v,
  T^{\otimes -1}(1)\otimes \Q_p/\Z_p)$ is torsion,
it thus follows that 
\begin{equation}\label{dual-lambda-rel-gal}
\begin{split}
\Hom_{\Lambda_*}(X_{\rel}^*, \Lambda_*)=\Hom_{\Lambda_*}(\rH^1(K_S/K_{\infty}^{*}, W)^{\vee}, \Lambda_*)=&\Hom_{\CO}(\rH^1(K_S/K_{\infty}^{*}, W)^{\vee}, \CO)\\
=&\varprojlim_{F\subseteq K_{\infty}^*}\Hom_{\CO}(\rH^1(K_S/F, W)^{\vee}, \CO),
\end{split}
\end{equation}
where the penultimate equality arises from composition with the augmentation $\Lambda_{*} \to \CO.$

Note that $\Hom_{\CO}(\rH^1(K_S/F, W)^{\vee}, \CO)=\varprojlim_{m}\rH^1(K_S/F, W)[p^m],$
and for $m\ge 1$ we have a short exact sequence
\[
0\to \rH^0(K_S/F, W)/p^m \to \rH^1(K_S/F, W[p^m]) \to \rH^1(K_S/F, W)[p^m] \to 0,
\]
leading to the short exact sequence
\[
0\to \rH^0(K_S/F, W) \to \rH^1(K_S/F, T) \to \varprojlim_{m}\rH^1(K_S/F, W)[p^m] \to 0.
\]
Noting that $\rH^0(K_S/F, W)$ is finite, (\ref{dual-lambda-rel-gal}) yields the $\Lambda_{*}$-module exact sequence
\[
0\to \varprojlim_{F\subseteq K_{\infty}^*}\rH^0(K_S/F, W) \to \scrS_{\rel}^{*} \to  \Hom_{\Lambda_*}(X_{\rel}^*, \Lambda_*) \to 0.
\]
Since $\varprojlim_{F\subseteq K_{\infty}^*}\rH^0(K_S/F, W)$ is $\Lambda_*$-torsion\footnote{Under Assumption \ref{for-integrality} (1), it is zero.}, this concludes the proof. 

\end{proof}

\subsection{Two-variable elliptic unit main conjecture}\label{ss, emc-t}
\subsubsection{Elliptic units} For an $\CO$-basis $t$ of $T^{\otimes -1}$, a non-zero ideal $\fa \subseteq \CO_K$ prime to $6p\ff$ and $* \in \{\emptyset, \ac\} $,
denote by $\!_{\fa}z_{p^{\infty}\ff}^{t, *}$ the image of the elliptic units  $(\!_{\fa}z_{p^{n}\ff})_n \in \bfH^1_{\ff}(\Z_p(1))$ of \cite[\S 15.5]{K}
under the composite 
\[
\bfH^1_{\ff}(\Z_p(1)) \xrightarrow{\otimes t}\bfH^1_{\ff}(\Z_p(1))\otimes_{\Z_p}T^{\otimes -1}\cong \bfH^1_{\ff}(T^{\otimes -1}(1))\to \varprojlim_{K\subseteq F \subseteq K_{\infty}^{*}}\rH^1(F, T^{\otimes -1}(1)),
\]
where the last map is induced by the corestriction maps.

By \cite[Cor.\ B.3.5]{Ru00}, we have $\!_{\fa}z_{p^{\infty}\ff}^{t, *}\in \scrS_{\rel}^{*}.$
Denote by $\fz$ the $\Lambda$-submodule of $\scrS_{\rel}^{}$ generated by $\Set{\!_{\fa}z_{p^{\infty}\ff}^{t} }_{\fa, t}.$
Fix a  $t$,
and put 
\[
z_{p^{\infty}\ff}^{\ac}=z_{p^{\infty}\ff}^{t, \ac}=(N_{K/\Q}(\fa)-\varphi^{-1}(\fa)(\fa, K(p^{\infty}\ff)/K) )^{-1}  \!_{\fa}z_{p^{\infty}\ff}^{t, \ac },
\]
which is independent of $\fa\neq \CO_K$ (cf.~\cite[(15.6.2)]{K}). 
Since there exists a positive integer $a$ prime to $6p\ff$ such that $\fa:=a\CO_{K}$ satisfies 
\begin{equation}\label{invertible-a}
(N_{K/\Q}(\fa)-\varphi^{-1}(\fa)(\fa, K(p^{\infty}\ff)/K) ) \in \Lambda_{\ac}^{\times},
\end{equation}
 it follows that \begin{equation}\label{lie-in-rel}
 z_{p^{\infty}\ff}^{\ac}\in \scrS_{\rel}^{\ac}
 \end{equation} (cf.\ \cite[Lem.\ 8.7]{BKNO}).

\begin{prop}\label{half-IMC}\leavevmode 
\begin{enumerate}
\item The localisation of $z_{p^{\infty}\ff}^{\ac}$ at $\fp$, i.e.\ its image in $\varprojlim_n \rH^1(K_{n,p}^{\ac}, T^{\otimes -1}(1))$, is $\Lambda_{\ac}$-non-torsion. 
\item The $\Lambda_{\ac}$-module  $X^{\ac}_{\str}$
is torsion.
\item The $\Lambda_{\ac}$-module $\scrS^{\ac}_{\rel}$ 
is  of rank one, and $\scrS_{\str}^{\ac}=\{0\}.$
\item We have 
$$\Ch_{\Lambda_{\ac}}
 (\scrS_{\rel}^{\ac}/ \Lambda_{\ac} z_{p^{\infty}\ff}^{\ac})\subseteq 
 \Ch_{\Lambda_{\ac}}(X^{\ac}_{\str}),$$
where $\Ch_{\Lambda_{*}}$ denotes the characteristic ideal over $\Lambda_{*}$.
\end{enumerate}
\end{prop}
\begin{remark}
It is subsequently shown that (4) is an equality (cf.~Theorem \ref{anticyc-zetaIMC}). 
\end{remark}
\begin{proof}
 In view of the assumption on $\varphi$,
 for a finite character $\chi$ of $\Gamma_{\rm ac}$, 
we have $\varphi\chi(\overline{\fa})=\overline{\varphi\chi({\fa})}.$
By Rohrlich \cite{Ro} (see also \cite{J}),
there are infinitely many finite characters $\chi$ of $\Gamma_{\rm ac}$ such that $$L(\varphi\chi,1)\neq 0.$$
Hence,
by the explicit reciprocity law for $z_{p^{\infty}\ff}^{\ac}$ (cf.~Proposition \ref{higher-explicit-reciprocity-law} below), part 1) follows.

 Part 2) is a consequence of part 1) since $z_{p^{\infty}\ff}^{\ac}$ extends to an Euler system (cf.~\cite[\S 9.2]{Ru00}).

 As for part 3), the assertion on the rank of $\scrS_{\rel}^{\ac}$ follows from part (2) and the global Euler--Poincar\'e characteristic formula.
As for the vanishing of $\scrS_{\str}^{\ac}$, 
suppose that $\scrS_{\str}^{\ac}\neq \{0\}$. Since $\scrS_{\str}^{\ac}$ is $\Lambda_{\ac}$-torsion-free (cf.~Proposition \ref{rank-rel}), we then have $\rank_{\Lambda_{\ac}}(\scrS_{\str}^{\ac})=1.$ 
Note that the localisation map at $p$ induces an injection $$\scrS_{\rel}^{\ac}/\scrS_{\str}^{\ac} \hookrightarrow \varprojlim_n\rH^1(K_{n,p}^{\ac}, T^{\otimes -1}(1))$$
and that $\varprojlim_n\rH^1(K_{n,p}^{\ac}, T^{\otimes -1}(1))=\rH^1(K_p, T^{\otimes -1}(1)\otimes_{\CO} \Lambda_{\ac})$ is $\Lambda_{\ac}$-torsion-free.
Hence, it follows that  $\scrS_{\rel}^{\ac}=\scrS_{\str}^{\ac}$.
However, this implies that $z_{p^{\infty}\ff}^{\ac}\in \scrS_{\str}^{\ac}$,
which contradicts part 1).

Analogously to the proof of part 2), part 4) is a consequence of part 1) and the usual Euler system argument. 
\end{proof}
\subsubsection{The main conjecture} 
The following is a standard consequence of Rubin's foundational work \cite{Ru91}, for which we include the details, as we know of no reference.   
\begin{prop}\label{RubinIMC} \leavevmode
\begin{enumerate}
\item  As ideals of $\Lambda\otimes_{\Z_p} \Q_p$,  we have $$\Ch_{\Lambda}(\scrS_{\rel}/\fz)\otimes_{\Z_p} \Q_p \subseteq \Ch_{\Lambda}(X_{\str})\otimes_{\Z_p} \Q_p.$$

\item Suppose that Assumption \ref{for-integrality} (2) holds and  that $p$ does not divide the number of  roots of unity in the Hilbert class field of $K$.
Then we have 
$$\Ch_{\Lambda}(\scrS_{\rel}/\fz)= \Ch_{\Lambda}(X_{\str}).$$
\end{enumerate}
\end{prop}
\begin{proof}
We first consider part 2).  

Denote by 
$\eta:G_K\to \CO^{\times}$
the finite character induced by the composition of the Teichm\"uller lift 
and the character $G_K \to (\CO/\fm)^{\times}$ corresponding to the Galois representation $T/\fm$.
Since $p\nmid [H:K]$ and $T|_{G_H} \cong T_p(A)\otimes_{\CO_K\otimes \Z_p}\CO$,
$p$ does not divide the order of the image of $\eta$,
and  $\mathrm{Im}(\varphi\eta^{-1}) \subseteq \CO^{\times}$ is a pro-$p$ group without $p$-torsion\footnote{
Here we denote by the same symbol the $p$-adic avatar (the character attached to the associated Galois representation) of an algebraic Hecke character.}. 
Hence, $\varphi^{-1}\eta$ factors through $\Gal(K_{\infty}/K)$.
 Denote by $K_0$ the finite abelian extension of $K$ corresponding to $\eta$. Then $p\nmid [K_0:K]$, and
 $\Gal(K_{\infty}K_0/K)=\Gal(K_{\infty}/K)\times \Gal(K_0/K)$ is a quotient of $\Gal(K(p^{\infty}\ff)/K)$ through which  $\varphi$ factors.
 
As $\Lambda$-modules, we have 
\begin{equation}\label{compact-twist}
\begin{split}
\scrS_{\rel}&=\varprojlim_{K\subseteq F \subseteq K_{\infty}}\rH^1(F, T^{\otimes -1}(1))
=\rH^0\left(K_0/K, \varprojlim_{K\subseteq F \subseteq K_{\infty}K_0}\rH^1(F, T^{\otimes -1}(1))\right)\\
&\cong \rH^0\left(K_0/K, \varprojlim_{K\subseteq F \subseteq K_{\infty}K_0}\rH^1(F, \Z_p(1))\otimes T^{\otimes -1}\right)\\
&\cong \left(\varprojlim_{K\subseteq F \subseteq K_{\infty}K_0}\rH^1(F, \Z_p(1))\right)^{\eta}(\varphi^{-1}\eta).
\end{split}
\end{equation}
 Here $F$ ranges over finite extensions of $K$ inside $K_{\infty}$ or $K_{\infty}K_0$, 
 the first equality follows from \cite[Lem.\ 1.3.5 and Cor.\ B.3.6]{Ru00},
  the third from  $T^{\otimes -1}\cong \CO(\varphi^{-1}\eta\eta^{-1})$,  
 the superscript ${\eta}$ denotes the submodule  on which $\Gal(K_0/K)$ acts by $\eta$,
   and $(\varphi^{-1}\eta)$ in the last term denotes the Galois-twist.
As $\Lambda$-modules, we also have 
\[
\begin{split}
\rH^1(K_{\infty},W)
&=\rH^1(K_{\infty}K_0, W)^{\Gal(K_0/K)}
\cong 
\rH^0\left(K_0/K, \rH^1(K_{\infty}K_0, L_{\lambda}/\CO)\otimes_{\CO}  T \right)\\
&=\rH^0\left(K_0/K, \rH^1(K_{\infty}K_0, L_{\lambda}/\CO)(\eta)\right)(\varphi\eta^{-1})\\
&=\rH^1(K_{\infty}, L_{\lambda}/\CO(\eta))(\varphi\eta^{-1}),
\end{split}
\]
and likewise
\begin{equation}\label{strong-twist}
\begin{split}
X_{\str}^{\vee}&=\varinjlim_{K \subseteq F \subseteq K_{\infty}}\Sel_{\str}(F, L_{\lambda}/\CO(\eta))(\varphi\eta^{-1})
=\varinjlim_{K \subseteq F \subseteq K_{\infty}}\Sel_{}(F, L_{\lambda}/\CO(\eta))(\varphi\eta^{-1})\\
&=\rH^0\left(K_0/K, \varinjlim_{K \subseteq F \subseteq K_{\infty}K_0}\Hom(A(F), L_{\lambda}/\CO(\eta))\right)(\varphi\eta^{-1})\\
&=\Hom\left(\left(\varprojlim_{K \subseteq F \subseteq K_{\infty}K_0}A(F)\right)^{\eta}, L_{\lambda}/\CO\right)(\varphi\eta^{-1}).
\end{split}
\end{equation}
Here 
$A(F)$ denotes the $p$-primary part of the ideal class group of $F$, 
the third equality follows from  \cite[Prop.\ 1.6.2]{Ru00}, 
and  the second from 
$\rH^1_{\rf}(F\otimes K_{\fp}, L_{\lambda}/\CO(\eta))=\{0\},$
which is a consequence of 
\[
\rH^1_{\rf}(F_{\fp_F}, L_{\lambda}(\eta))=\rH^1_{\ur}(F_{\fp_F}, L_{\lambda}(\eta))=
\rH^0(F_{\fp_F}^{\ur}, L_{\lambda}(\eta))/(\Fr_{\fp_{F}}-1)=0,
\]
where $\fp_F$ denotes the prime of $F \subseteq K_{\infty}$ above $\fp$. 
Moreover, the last equality holds as follows: 
suppose that $\rH^0(F_{\fp_F}^{\ur}, L_{\lambda}(\eta))/(\Fr_{\fp_{F}}-1) \neq 0,$ 
then $\eta|_{G_{F_{\fp_F}}}$ is trivial. Since $[F:K]$ and $[K_0:K] $ are relatively prime,
$\eta|_{G_{K_{\fp}}}$ is trivial, which implies that the image of $\varphi|_{G_{K_{p}}}$ is pro-$p$ without $p$-torsion,   contradicting Lemma \ref{non-anomalous}
(note that our assumption implies Assumption \ref{for-integrality} (1)).

Consequently, 
part 2) follows from (\ref{compact-twist}), (\ref{strong-twist}) and the twisting \cite[Thm.\ 4.1 (ii)]{Ru91}.
  (Note that 
   $\eta$ is non-trivial on the decomposition group at $\fp$ by Lemma~\ref{non-anomalous},  as required in \textit{loc.\ cit.})

Part 1) is the content of \cite[Thm.\ 15.2 (1)]{K}.
\end{proof}
\subsection{Anticyclotomic elliptic unit main conjecture} We consider the anticyclotomic variant of \S\ref{ss, emc-t}.
\subsubsection{Preliminaries}
\begin{lem}\label{inj-ram}\leavevmode
\begin{enumerate}
\item 
For a prime $v\mid \ff$  of $K_{\infty}^{\ac}$ not dividing $p$,
the kernel  of the restriction map
\begin{equation}\label{ac-to-full-at-f}
\rH^1(K_{\infty,v}^{\ac},W) \to \rH^1(K_{\infty}\otimes_{K_{\infty}^{\ac}}K_{\infty,v}^{\ac}, W )
\end{equation}
is annihilated by $p^k$ for some integer $k$. 
\item Under Assumption \ref{for-integrality} (2),
the map (\ref{ac-to-full-at-f}) is injective.
\end{enumerate}
\end{lem}
\begin{proof}

 The kernel of (\ref{ac-to-full-at-f}) is annihilated by the order of $\rH^0(K_{v}^{\ur},W),$
 which is finite\footnote{Note that  $\rH^0(K_{v}^{\ur},W)$ is a proper $\CO$-submodule of the divisible group $W$ of corank one.}, and so part 1) follows.  

Since $p\nmid [H:K]$, we have 
 $p\nmid [H(A[\fp]):K]$ and $W|_{G_{H}} \cong A[p^{\infty}]\otimes_{\CO_K}\CO$. 
So it suffices to show that the restriction map
\begin{equation}\label{infinity-infinity}
\rH^1(K_{\infty,v}^{\ac}\otimes_KH(A[\fp]),A[p^{\infty}]) \to \rH^1(K_{\infty}\otimes_{K_{\infty}^{\ac}}K_{\infty,v}^{\ac}\otimes_KH(A[\fp]), A[p^{\infty}] )
\end{equation}
is injective.

Note that $A$ over $H(A[\fp])$ has good reduction outside $p$.
Indeed, for a prime $w\nmid p$ of $H(A[\fp])$,
the image of the inertia subgroup $I_w$ under the natural homomorphism $G_{H(A[\fp])}\to \Aut_{\CO_K\otimes \Z_p}(T_p(A))$
is pro-$p$.
Let $l$ be  a rational prime such that $p\nmid |(\CO_K/l\CO_K)^{\times}|$. 
 Since \cite[Thm.\ 2 (ii)]{ST} implies that
 the image of $I_w$ under the homomorphism $G_{H(A[\fp])}\to \Aut_{\CO_K\otimes \Z_l}T_l(A)$ is pro-$p$, 
the image is trivial. Hence, $A$ has good reduction at $w$ by the Serre--Tate theorem.

 In turn, it suffices to show the injectivity of  homomorphisms 
\begin{equation}\label{FF}
\rH^1(F, A[p^{\infty}] ) \to \rH^1(F^{\prime}, A[p^{\infty}]), 
\end{equation}
where $F$ is a finite extension of $K_{\infty,v}^{\ac}$ over which $A$ has good reduction 
and $F^{\prime}$ is an unramified extension of $F$.
Moreover,  since $A(F)\otimes \Q_p/\Z_p$ and $A(F^{\prime})\otimes \Q_p/\Z_p$ are trivial, 
(\ref{FF})
is identified with 
\[
\rH^1(F, A)[p^{\infty}]  \to \rH^1(F^{\prime}, A)[p^{\infty}],
\]
whose kernel is isomorphic to $\rH^1(F^{\prime}/F, A(F^{\prime}))[p^{\infty}]$, which vanishes by 
\cite[Prop.\ 4.3]{M} and so part 2) follows.
\end{proof}

\begin{lem}\label{finiteness-at-p}\leavevmode 
\begin{enumerate}
\item 
The kernel of the restriction map 
\begin{equation}\label{ac-to-full-at-p}
\rH^1(K_{\infty,\fp}^{\ac},W) \to \rH^1(K_{\infty, \fp}, W )
\end{equation}
is finite,
 where we also write $\fp$ for the unique prime of $K_{\infty}^{\ac}$ above $\fp$.
\item Under Assumption \ref{for-integrality} (1),
the map (\ref{ac-to-full-at-p}) is injective.
\end{enumerate}
\end{lem}
\begin{proof}
As for part 1), 
it suffices to show that  $\rH^1(K_{\infty,\fp}/K_{\infty,\fp}^{\ac}, \rH^0(K_{\infty,\fp},W))$ is finite. 
If $\rH^0(K_{\infty,\fp},W)$ is finite, then 
the assertion is trivial.
If $\rH^0(K_{\infty,\fp},W)$ is infinite,
then $\rH^0(K_{\infty,\fp},W)=W$. 
Hence, the $p$-adic avatar of $\varphi$ factors through $\Gal(K_{\infty, \fp}/K_{\fp})$, and 
\[
\rH^1(K_{\infty,\fp}/K_{\infty,\fp}^{\ac}, \rH^0(K_{\infty,\fp},W))=\rH^1(K_{\infty,\fp}/K_{\infty,\fp}^{\ac}, W)
\]
is finite since the restriction of $\varphi$ to $G_{K_{\infty, \fp}^{\ac}}$ is non-trivial.

Recall that $\varphi(\fa^{c})=\overline{\varphi(\fa)}$, 
where $c$ denotes the complex conjugation. 
Since $c$ preserves $L$,
Lemma \ref{non-anomalous} implies that
\begin{equation}\label{vanishing-p}
\rH^0(K_{\infty, \fp}, W)=\rH^0(K_{\infty, \fp}, W^{\vee}(1))\otimes_{\CO_{L}, c}\CO_L=0.
\end{equation}
Hence, part 2) follows from the inflation-restriction exact sequence.
\end{proof}

\subsubsection{Anticyclotomic descent}
\begin{prop}\label{reduction-rel}\leavevmode
\begin{enumerate}
\item
The kernels of the following natural homomorphisms of $\Lambda_{\ac}$-modules
\begin{equation}\label{base-change-compact}
\scrS_{\rel}\otimes_{\Lambda}\Lambda_{\ac} \to \scrS_{\rel}^{\ac},\quad 
\scrS_{\rel}/\fz_{\fa}\otimes_{\Lambda}\Lambda_{\ac} \to \scrS_{\rel}^{\ac}/\Lambda_{\ac}z^{\ac}_{p^{\infty}\ff}
\end{equation}
are finite,
where $\fa$ is as in (\ref{invertible-a}),
and
$\fz_{\fa}$ denotes the $\Lambda$-submodule of $\scrS_{\rel}$ generated by $\!_{\fa}z_{p^{\infty}\ff}^{t,\emptyset}$.
\item
The characteristic ideals of the cokernels of the homomorphisms (\ref{base-change-compact})  equal $\Ch_{\Lambda_{\ac}}(X_{\str}[I])$ up to powers of $p\Lambda_{\ac}$. 
\item Under Assumption \ref{for-integrality},
the homomorphisms (\ref{base-change-compact}) are injective, and
 the characteristic ideals of their cokernels equal $\Ch_{\Lambda_{\ac}}(X_{\str}[I]),$
where  $I$ denotes the kernel of $\Lambda \to \Lambda_{\ac}.$
\end{enumerate}
\end{prop}
\begin{proof}
 In view of Lemma \ref{rel-gal},
 Shapiro's lemma implies that $\scrS_{\rel}^*=\rH^1(K_{S}/K, T^{\otimes -1}(1)\otimes \Lambda_{*}),$
where $K_{S}$ denotes the maximal extension of $K$  unramified outside $p\ff$. 
Since $I$ is principal, we thus have an exact sequence of $\Lambda_{\ac}$-modules: 
\begin{equation}\label{rel-h2}
0\to \scrS_{\rel}\otimes_{\Lambda}\Lambda_{\ac} \to \scrS_{\rel}^{\ac} \to \rH^2(K_{S}/K, T^{\otimes -1}(1)\otimes \Lambda)[I] \to 0.
\end{equation}
On the other hand, by Poitou--Tate duality,
we also have a $\Lambda_{\ac}$-exact sequence: 
\[
0\to X_{\str}[I]\to \rH^2(K_{S}/K, T^{\otimes -1}(1)\otimes \Lambda)[I] \to \bigoplus_{v\mid p\ff} \varprojlim_{K\subseteq F \subseteq K_{\infty}}\rH^2(F\otimes_K K_{v}, T^{\otimes -1}(1))[I],
\]
where $v$ ranges over all primes of $K$ dividing $p\ff$. 
Note that  the Pontryagin dual of the $v$-component of the rightmost term is isomorphic to
\[
\begin{split}
\rH^0(K_{\infty}\otimes K_v, W)/I&\cong \rH^1\left(\Gal(K_{\infty}/K_{\infty}^{\ac}), \rH^0(K_{\infty}\otimes K_v, W)\right)\\
&\cong \ker\left(
\rH^1(K_{\infty}^{\ac}\otimes K_v,  W)\to \rH^1(K_{\infty}\otimes K_v,  W)
\right).
\end{split}
\]
In view of Lemmas \ref{inj-ram} and \ref{finiteness-at-p},
the above kernel is finite and it vanishes under Assumption \ref{for-integrality}.
Consequently, (\ref{rel-h2}) implies the assertion for $\scrS_{\rel}\otimes_{\Lambda}\Lambda_{\ac} \to \scrS_{\rel}^{\ac}$.

The remaining assertion  follows by  applying the snake lemma to the commutative diagram of $\Lambda_{\ac}$-modules: 
\[
\xymatrix{
& \fz_{\fa} \otimes_{\Lambda}\Lambda_{\ac} \ar[r] \ar[d]^{(1)} & \scrS_{\rel}\otimes_{\Lambda}\Lambda_{\ac} \ar[r]\ar[d] &  \scrS_{\rel}/\fz_{\fa}\otimes_{\Lambda}\Lambda_{\ac} \ar[r]\ar[d] & 0\\
0 \ar[r] & \Lambda_{\ac}z_{p^{\infty}\ff}^{\ac}\ar[r]  & \scrS_{\rel}^{\ac} \ar[r] &  \scrS_{\rel}^{\ac}/\Lambda_{\ac}z_{p^{\infty}\ff}^{\ac} \ar[r] & 0,
}
\]
and noting that (1) is surjective.
\end{proof}

\begin{prop}\label{restriction-str}\leavevmode 
\begin{enumerate} 
\item The homomorphism  
$X_{\str}\otimes_{\Lambda}\Lambda_{\ac}\to X_{\str}^{\ac}$
  induced by the dual of the restriction map is an isomorphism after taking the tensor product with $\Q_p$ over $\Z_p$.
\item Under Assumption \ref{for-integrality}, the homomorphism in part (1) is an isomorphism of $\Lambda_\ac$-modules. 
\item The ideal $\Ch_{\Lambda_{\ac}}(X_{\str}[I])$ is non-zero 
and $\Ch_{\Lambda_{\ac}}(X_{\str}^{\ac}) = \Ch_{\Lambda}(X_{\str})\Ch_{\Lambda_{\ac}}(X_{\str}[I]).$
\end{enumerate}
\end{prop}
\begin{proof}
We have the following commutative diagram  of $\Lambda_\ac$-modules with exact rows: 
\[
\xymatrix{
0 \ar[r]& \Sel_{\str}(K_{\infty}^{\ac}, W) \ar[r]\ar[d]& \rH^1(K_S/K_{\infty}^{\ac}, W) \ar[r]\ar[d]^{(a)} &\bigoplus_{v\mid p\ff}\rH^1(K_{\infty,v}^{\ac}, W)\ar[d]^{(b)} \\
0 \ar[r]& \Sel_{\str}(K_{\infty}^{}, W)[I] \ar[r]& \rH^1(K_S/K_{\infty}^{}, W)[I] \ar[r] & \bigoplus_{v\mid p\ff}\rH^1(K_{\infty}\otimes_{K_{\infty}^{\ac}}K_{\infty,v}^{\ac}, W)[I], 
}
\]
where $v$ ranges over all primes of $K_{\infty}^{\ac}$ dividing $p\ff$.
Note that the kernel and cokernel of (a) are annihilated by some power of $p$, and vanish by (\ref{vanishing-p}) under Assumption \ref{for-integrality}. 
Moreover, by Lemmas \ref{inj-ram} and \ref{finiteness-at-p}, the kernel of the map (b) is annihilated by some power of $p$,
which vanishes under Assumption \ref{for-integrality}.
Hence, by the snake lemma, the restriction map gives rise to the isomorphism $\Sel_{\str}(K_{\infty}^{\ac}, W) \cong \Sel_{\str}(K_{\infty}^{}, W)[I]$
as in part 1), and also in part 2) under Assumption \ref{for-integrality}.

By Proposition \ref{half-IMC} (2) and parts 1) and 2),
the $\Lambda_{\ac}$-module $X_{\str}/I$ is torsion. Hence, part 3) follows from 
\cite[Lem.\ 6.2]{Ru91}.
\end{proof}

\begin{thm}\label{anticyc-zetaIMC}\leavevmode 
\begin{enumerate}
\item We have
$\Ch_{\Lambda}(\scrS_{\rel}/\fz_{\fa})\Lambda_{\ac}=\Ch_{\Lambda_{\ac}}((\scrS_{\rel}/\fz_{\fa})/I).$
\item We have $\Ch_{\Lambda_{\ac}}(\scrS_{\rel}^{\ac}/ \Lambda_{\ac}z^{\ac}_{p^{\infty}\ff})\otimes \Q_p=\Ch_{\Lambda_{\ac}}(X_{\str}^{\ac})\otimes \Q_p$ as ideals of $\Lambda_{\ac}[1/p].$
\item Suppose that Assumption \ref{for-integrality} holds, 
and that  $p$ does not divide the number of roots of unity in the Hilbert class field of $K$.
Then,  $\Ch_{\Lambda_{\ac}}(\scrS_{\rel}^{\ac}/ \Lambda_{\ac}z^{\ac}_{p^{\infty}\ff})=\Ch_{\Lambda_{\ac}}(X_{\str}^{\ac}).$
\end{enumerate}
\end{thm}
\begin{proof}
Since $I$ is principal, by the snake lemma,
we have an exact sequence
\[
0 \to \fz_{\fa}[I ]\to \scrS_{\rel}[I] \to (\scrS_{\rel}/\fz_{\fa})[I] \to \fz_{\fa}/I \to \scrS_{\rel}/I \to (\scrS_{\rel}/\fz_{\fa})/I\to 0,
\]
where $\fa$ is as in (\ref{invertible-a}).
Hence,  noting that $\fz_{\fa}\cong \Lambda$,
  Proposition \ref{rank-rel}
implies that  $(\scrS_{\rel}/\fz_{\fa})[I]$ injects into $\fz_{\fa}/I$ and hence is $\Lambda_{\ac}$-torsion-free.
By Propositions \ref{half-IMC}(1) and \ref{reduction-rel}(1),
 $(\scrS_{\rel}/\fz_{\fa})/I$ is $\Lambda_{\ac}$-torsion,
and so \cite[Lem.\ 6.2 (i)]{Ru91} implies that
$(\scrS_{\rel}/\fz_{\fa})[I]$ is $\Lambda_{\ac}$-torsion.
Therefore, $(\scrS_{\rel}/\fz_{\fa})[I]=0,$
and by \cite[Lem.\ 6.2 (ii)]{Ru91}
we have
$$\Ch_{\Lambda}(\scrS_{\rel}/\fz_{\fa})\Lambda_{\ac}=\Ch_{\Lambda_{\ac}}((\scrS_{\rel}/\fz_{\fa})/I),$$
yielding part 1). 

As for part 3), note that
\[
\Ch_{\Lambda_{\ac}}(\scrS_{\rel}^{\ac}/\Lambda_{\ac}z_{p^{\infty}\ff }^{\ac})=\Ch_{\Lambda}(\scrS_{\rel}/\fz_{\fa})\Ch_{\Lambda_{\ac}}(X_{\str}[I])
=\Ch_{\Lambda}(X_{\str})\Ch_{\Lambda_{\ac}}(X_{\str}[I])=\Ch_{\Lambda_{\ac}}(X^{\ac}_{\str}),
\]
where the first equality follows from Proposition \ref{reduction-rel},
the second from Proposition \ref{RubinIMC} (2) and \cite[(15.6.4)]{K},
and the last one from Proposition \ref{restriction-str} (2).

Finally, for part 2), it suffices to show that 
 $\Ch_{\Lambda}(\scrS_{\rel}/\fz_{\fa})\otimes \Q_p=\Ch_{\Lambda}(X_{\str})\otimes \Q_p$,
which is the content of \cite[Thm.\ 5.2]{JL-K}. 
\end{proof}

\section{$p$-adic $L$-functions at ramified primes}\label{general-interpolation}

This section studies the interpolation property of the Rubin-type $p$-adic $L$-function at the de Rham specialisations of arbitrary infinity type. The main result is Theorem~\ref{interpolation-rubin}, the interpolation formula for the $p$-adic $L$-function.

\subsection{Local and global epsilon constants associated to Galois characters}
Let the setting be as in \S\ref{ss, nt3}. 
We assume that $p\CO_K$ does not divide 
the conductor $\ff$ of $\varphi$ (in the ramified case,
the prime ideal $\fp$ of $K$ above $p$ may exactly divide $\ff$). 
By twisting by a finite order character of $\Gal(K_{\infty}^{\ac}/K)$,
we may always reduce to this case.

 Fix a $\Z_p$-basis $\zeta=(\zeta_{n})$ of $\Z_p(1):=\varprojlim_n \Gamma(\Spec(\overline{\Q}_p ), \Z/p^n\Z(1))$,
which gives rise to a continuous homomorphism $\mathbf{e}_{\Q_p}:\Q_p \to \overline{\Q}_p^{\times}$ characterised by $\mathbf{e}_{\Q_p}(1/p^n)=\zeta_n$ for $n\ge 1$. Write $\mathbf{e}_{K_{\fp}}=\mathbf{e}_{\Q_p}\circ \mathrm{Tr}_{K_{\fp}/\Q_p}.$
\subsubsection{Local case}\label{local-epsilon-section}
Let $F_{p} \in \{\Q_p, K_{\fp}\},$ and $G_{F_p}:=\Gal(\ov{\Q}_{p}/F_{p})$. 

For a finite extension $M$ of $\Q_p$ and a potentially crystalline $M$-representation $V$ of $G_{F_p}$, 
note that $D_{\pst}(V)$ is an $M\otimes_{\Q_p}\Q_p^{\ur}$-module equipped with a continuous and linear action of the Weil group $W_{F_p}$ of $F_p$ (cf.\ \cite[\S 2.4]{BKNO}),
where the topology on $D_{\pst}(V)$ is discrete.
Then, 
\[
\varepsilon(V)=
\left(\varepsilon(D_{\pst}(V)\otimes_{M\otimes \Q_p^{\ur}, \tau}\overline{\Q}_p, \mathbf{e}_{F_p}, dx_{F_p})\right)_{\tau} \in \prod_{\tau: M \hookrightarrow \overline{\Q}_p}\overline{\Q}_p^{\times}=(M\otimes_{\Q_p} \overline{\Q}_p)^{\times},
\]
where $dx_{F_p}$ denotes the Haar measure on $F_p$ (valued in $\overline{\Q}_p$) which is self-dual with respect to $\mathbf{e}_{F_p},$
 $\tau$ ranges over all homomorphisms $M\to \overline{\Q}_p$ of $\Q_p$-algebras,
and $ \varepsilon(D_{\pst}(V)\otimes_{M\otimes \Q_p^{\ur}, \tau}\overline{\Q}_p, \mathbf{e}_{F_p}, dx_{F_p})$ denotes the local epsilon constant of the smooth $\overline{\Q}_p$-representation
$D_{\pst}(V)\otimes_{M\otimes \Q_p^{\ur}, \tau}\overline{\Q}_p$ of the Weil group. 

Let $F \in \{\Q, K\}$ and $G_{F}:=\Gal(\ov{\Q}/F)$.
If $V$ is an  $M$-representation of $G_F$ 
such that $V|_{G_{F_{p}}}$ is potentially crystalline,
put $$\varepsilon_{p}(V)=\varepsilon(V|_{G_{F_{p}}}),$$
where $F_p=F\otimes_{\Q}\Q_p$.
 Although such a local epsilon constant is defined for general de Rham representations, we consider only potentially crystalline ones in this paper.

For a finite order character $\chi: \Gal(\overline{\Q}/K)\to \overline{\Q}^{\times} \subseteq \C^{\times}$, put
$\Ind_{K/\Q}(\varphi\chi)=\Ind_{K/\Q}(V_{\varphi\chi})$, 
following the notation (\ref{our-V}). 

For an anticyclotomic character $\chi$ which is de Rham at $p$,
the two-dimensional $p$-adic  representation $\Ind_{K/\Q}(\varphi \chi)$ of $G_{\Q}$ is symplectic self-dual 
and we have
\[
\varepsilon_{p}(\Ind_{K/\Q}(\varphi \chi))
\in \{\pm 1\},
\]
which is independent of the choice of $\zeta$ (cf.~\cite[Lem.\ 4.3.4]{BKNO} and \cite[(3.4.4), (3.4.7)]{T}). 

\subsubsection{Global case}

Let $\mathbf{e}_K: \mathbb{A}_K/K\to \C^{\times}$ be a non-trivial character defined as $\mathbf{e}_{\Q}\circ\mathrm{Tr}_{K/\Q}$ for some $\mathbf{e}_{\Q}: \mathbb{A}_{\Q}/\Q \to \C^{\times}$ such that for every rational prime $l$, $n(\mathbf{e}_{\Q}|_{\Q_l})$ in \cite[(3.2.6)]{T} equals  $1$.
Let $dx$ be the Haar measure on $\mathbb{A}_K$ such that $\int_{\mathbb{A}_K/K}dx=1.$
Let $dx=\prod_{v}dx_v$ be a factorisation of $dx$, where $v$ ranges over all places of $K$, and  $dx_v$ is a Haar measure on $K_v$ such that for almost all $v$ we have $\int_{\CO_v}dx_v=1.$

For an algebraic Hecke character $\eta:\mathbb{A}_K^{\times}/K^{\times}\to \C^{\times}$ of $K$ and for a place $v$ of $K$,
denote by $\varepsilon(\eta|_{K_v^{\times}}, \mathbf{e}_K|_{K_v}, dx_v)$ the local epsilon constant associated to $\eta|_{K_v^{\times}}$.
 The associated global epsilon constant is defined by
  $$\varepsilon(\eta)=\prod_v\varepsilon(\eta|_{K_v^{\times}}, \mathbf{e}_K|_{K_v}, dx_v),$$
which is independent of the choices of $\mathbf{e}_K$ and $dx$. 
If $\eta$ satisfies (\ref{self-dual-character}), then $\varepsilon(\eta)\in \{\pm 1\}$, and 
its sign is the same as that of the functional equation of $L(\eta,s)$. 

\subsubsection{Anticyclotomic variation} 
We first recall a special case of results of \cite{BKNO},
and give some consequences which will be used later.

\begin{lem}\label{conductor-of-anticyc}
For a finite character $\chi $ of $\Gal(K_{\infty}^{\ac}/K)$,
the conductor is of the form $(p^{n}).$ 
\end{lem}
\begin{proof}
We may assume that $p$ is ramified in $K/\Q$.
It suffices to show that if  $\chi$ is trivial on $1+\fp^{2n+1}\CO_{K_{\fp}} \subseteq \CO_{K_{\fp}}^{\times}$ for some $n\ge 0$
then  $\chi$ is trivial on $1+\fp^{2n}\CO_{K_{\fp}}=1+p^n\CO_{K_{\fp}}$ if $n\ge 1$ or on $\CO_{K_{\fp}}^{\times}$ if $n=0$. 

Suppose first that $n\ge 1$. 
Let $a=1+p^{n}x$ with $x \in \CO_{K_{\fp}}$. 
We shall show that $\chi(a)=1$.
Since $\CO_{K_{}}/\fp \simeq \Z/p\Z$, we may write $x=y+\pi z$ with $y \in \Z_p$ and $z \in \CO_{K_{\fp}}$,
 where $\pi$ denotes a uniformiser at $\fp$. 
Then, 
\[
\chi(a)=\chi(1+p^n(y+\pi z))=\chi(1+p^ny+p^n\pi z )=\chi(1+p^ny)\chi\left(1+p^n\pi z(1+p^ny)^{-1}\right)=1.
\]
If $n=0$, then the assertion follows by noting that $\#(\CO_{K_{\fp}}^{\times}/(1+\fp\CO_{K_{\fp}}))$ is coprime to $p$.
\end{proof}

We use the notation of \S \ref{subsection-local-decomposition} with $\psi$ taken to be
$\varphi|_{G_{K_{\fp}}}$.
\begin{defn}
Denote by $\Xi$ the set of finite order characters of $\Gamma:=\Gal(K_{\infty}^{\ac}/K)$
and by $\Xi_n$ the subset of $\Xi$ consisting of characters of  order $p^n$.
Put  
\[
\Xi_n^{\pm} = \Set{\chi \in \Xi_n | \varepsilon_p(\Ind_{K/\Q}(\varphi\chi))=\pm 1}
\]
and $\Xi^{\pm}=\bigcup_{n\ge 0}\Xi_n^{\pm}$.
\end{defn}
\begin{remark}\label{sign-convention} 
The above notation differs from \cite{BKNO}. 
 If $p$ is ramified in $K$, then the above $\Xi^{\pm}$ 
 equals $\Xi^{\mp}_{\varphi|_{G_{K_{\fp}}}}$
in \cite[Def.\ 7.34]{BKNO}.
Our sign convention is based on (\ref{base-change-rep}) below.
\end{remark}

For a finite order character $\chi$ of $\Gamma$, note that 
\begin{equation}\label{base-change-rep}
V^{\otimes -1}(1)\otimes_{\CO}\Lambda_{\ac}\otimes_{\Lambda_{\ac}, \chi}\CO_{\chi}=V(\varphi)(1)\otimes_{\CO}\Lambda_{\ac}\otimes_{\Lambda_{\ac}, \chi}\CO_{\chi}=V(\varphi\chi)(1)=V_{\varphi\chi}^{\otimes -1}(1),
\end{equation}
where $\CO_{\chi}=\CO[\mathrm{Im}(\chi)].$
Recall that $V(\varphi\chi)$ is the one-dimensional representation of $G_K$ on which  the action of the arithmetic Frobenius at $\fq\nmid p\ff$ is given by $\varphi^{-1}\chi^{-1}(\fq)$, 
and that $G_K$ acts on the component $\Lambda_{\ac}$ of $V^{\otimes -1}(1)\otimes_{\CO}\Lambda_{\ac}$ via $g\mapsto [g^{-1}]$,
the inverse of the group-like element.
In particular, $V(\varphi\chi)$ is isomorphic to the twist of $V^{\otimes -1}$ by $\chi^{-1}$. 

Let $\eta$ denote the finite order Hecke character over $K$ such that
$\varphi\eta^{-1}$ factors through $\Gamma$ (see \S\ref{interpolation-section} for its
modulus), and put $\varphi_{\ac}=\varphi/\bar{\varphi}$,
$\eta_{\ac}=\eta/\bar{\eta}$. The de Rham characters of $\Gamma$ are of the form
$\xi=\varphi_{\ac}^{k}\eta_{\ac}^{-k}\chi$ for $k\geq 0$ and $\chi$ finite order, up
to the involution $\iota$. Following \cite[Def.~2.10]{BKNO}, put
$\hat{\varepsilon}_p(\Ind_{K/\Q}(\varphi\xi))
:=\Gamma(\Ind_{K/\Q}(\varphi\xi))\,\varepsilon_p(\Ind_{K/\Q}(\varphi\xi))
=(-1)^{k}\varepsilon_p(\Ind_{K/\Q}(\varphi\xi)),$
the second equality by \cite[Cor.~2.9]{BKNO}. Note that
$\hat{\varepsilon}_p(\Ind_{K/\Q}(\varphi\chi))=\varepsilon_p(\Ind_{K/\Q}(\varphi\chi))$
for finite order $\chi$.

\begin{lem}\label{density-of-gxi}\leavevmode
\begin{enumerate}
\item For a de Rham character $\xi=\varphi_{\ac}^{k}\eta_{\ac}^{-k}\chi$ of $\Gamma$,
we have
$$\frac{\varepsilon(\varphi\xi)}{\varepsilon(\varphi)}
=\frac{\hat{\varepsilon}_p(\Ind_{K/\Q}(\varphi\xi))}
{\varepsilon_{p}(\Ind_{K/\Q}(\varphi))}.$$
\item For $b\in \Z_p^\times$ and a nontrivial finite order character $\chi$ of
$\Gamma$, we have
$$\varepsilon(\varphi\chi^{b})=\left(\frac{b}{p}\right)\varepsilon(\varphi\chi).$$
\item For each $k\geq 0$ and each sign $\pm$, there exist infinitely many finite
order characters $\chi$ of $\Gamma$ with
$\varepsilon(\varphi\,\varphi_{\ac}^{k}\eta_{\ac}^{-k}\chi)=\pm 1$.
\end{enumerate}
\end{lem}
\begin{proof}
For parts 1) and 2) with $k=0$, one may proceed as in the proof of
\cite[Prop.~8.2]{BKNO}. For part 1) in general, the finite places away from $p$
contribute trivially to the ratio by the same argument, while the archimedean
contribution equals $(-1)^{k}$, which coincides with the ratio of the
$\Gamma$-constants by \cite[Cor.~2.9]{BKNO} since $\Ind_{K/\Q}(\varphi\xi)$ has
Hodge--Tate weights $(k+1,-k)$. As for part 3), since $\xi|_{G_{K_{\fp}}}$ is
trivial on $\Q_p^{\times}$, the character $\varphi_p\,\xi|_{G_{K_{\fp}}}$ is
conjugate symplectic self-dual and de Rham; hence Lemma~\ref{density-of-xi} applies
with $\psi$ replaced by it, and part 1) concludes the proof.
\end{proof}

\begin{defn}\label{def, epsilon}
Put $$\varepsilon=\frac{\varepsilon(\varphi)}{\varepsilon_p(\Ind_{K/\Q}(\varphi))} \in \{\pm 1\}.$$
\end{defn}
\subsection{Rubin-type $p$-adic $L$-function: definition} 
Recall that we have fixed an $\CO$-basis $t$ of $T^{\otimes -1}$ to define the elliptic unit $z_{p^{\infty}\ff }^{t,\ac}$. 
We begin with the following preliminary. 
\begin{prop}\label{prop, zeta-in-epsilon}
We have 
$$
\loc_{p}(z_{p^{\infty}\ff }^{t,\ac}) \in \rH^1_{\varepsilon}(K_p, T^{\otimes -1}(1)\otimes \Lambda_{\ac}).
$$
\end{prop}
\begin{proof}
One may proceed just as in the proof of \cite[Cor.\ 3.3]{Ru} or \cite[Prop.\ 8.11]{BKNO}.
\end{proof}

\begin{defn}\label{def, RpL}\leavevmode
\begin{enumerate}
\item 
Suppose that Assumption \ref{for-integrality} (1) holds.
Fix a $\Lambda_{\ac}$-basis $v_\varepsilon$ of $\rH^1_{\varepsilon}(K_p, T^{\otimes -1}(1)\otimes \Lambda_{\ac})$ 
and define $\mathscr{L}_p(\varphi):=\mathscr{L}_{p,v_{\varepsilon}, t}(\varphi) \in \Lambda_{\ac}$ by
\[
\mathscr{L}_{p,v_{\varepsilon}, t}(\varphi) \cdot v_{\varepsilon} = \loc_{p}(z_{p^{\infty}\ff }^{t,\ac}).
\]
\item Without Assumption \ref{for-integrality} (1), 
fix a
$\Lambda_{\ac}[1/p]$-basis $v_\varepsilon$ of $\rH^1_{\varepsilon}(K_p, T^{\otimes -1}(1)\otimes \Lambda_{\ac})\otimes_{\Z_p} \Q_{p}$ (cf.\ Remark \ref{signed-condition-p-invert}), 
and define $\mathscr{L}_p(\varphi):=\mathscr{L}_{p,v_{\varepsilon}, t}(\varphi) \in \Lambda_{\ac}\otimes_{\Z_p} \Q_p$ by 
\[
\mathscr{L}_{p,v_{\varepsilon}, t}(\varphi) \cdot v_{\varepsilon} = \loc_{p}(z_{p^{\infty}\ff }^{t,\ac}). 
\]
\end{enumerate}
\end{defn}
We refer the reader to Theorem \ref{interpolation-rubin} for an interpolation property of $\mathscr{L}_p(\varphi)$ in terms of Hecke $L$-values.

\subsection{CM motives}
\subsubsection{} Let $\psi$ be a Hecke character over $K$ of infinity type  $(i, j)$, where the first component corresponds to the restriction of  the fixed embedding $K^{\mathrm{ab}} \hookrightarrow \C$. 
 For an $\CO_K$-algebra $M$ containing $\psi(\widehat{K}^{\times})$,
define an $M$-module
\[
V_M(\psi)=\rH^1_{}(A(\C), \Z)^{\otimes_{}i}  \otimes_{}  \overline{\rH^1_{}(A(\C), \Z)}^{\otimes_{}j} \otimes_{}M, 
\]
where the tensor products are taken over $\CO_K$,
$(A, \alpha)$  is the canonical CM pair over $K(\fg)$ of conductor $\fg$ in the sense of \cite[\S 15.3]{K}, $\fg\neq 0$ is an integral ideal of $K$ contained in the conductor of $\psi$, 
$\overline{\rH^1_{}(A(\C), \Z)}:=\rH^1_{}(A(\C), \Z)\otimes_{\CO_K, c}\CO_K$ for 
$c$ the non-trivial element of $\Gal(K/\Q)$. 

If $M$ contains the subextension $L_{\psi} \subseteq \C$ generated by $K$ and $\psi(\hat{K}^{\times})$, then $V_M(\psi)$ is independent of the choice of $\fg$ (cf.~\cite[p.~257]{K}).  
If $M$ is a finite extension of $L_{\psi, \lambda}$ for $\lambda$ the restriction of the fixed place of $\overline{\Q}$ above $p$,
then 
$V_M(\psi)$ is naturally  endowed with a $G_K$-action characterised as follows: the arithmetic Frobenius $\Fr_{\fq}$
acts by $\psi(\fq)^{-1}$ for $\fq$ relatively prime to $p$ and to the conductor of $\psi$. 

\subsubsection{de Rham realisation: the infinity type $(k,0)$ case}
We suppose that $\psi$ is of infinity type $(k,0)$ for $k\geq 1$. Following \cite[\S 15]{K},
we introduce some notation concerning its de Rham realisation.

Let $S(\psi)$  be as in \cite[\S 15]{K}: 
a one-dimensional $L_{\psi}$-subspace of $\coLie(A)^{\otimes_{K(\fg)} k}\otimes_K L_{\psi}$
such that $$S(\psi)\otimes_K K(\fg)=\coLie(A)^{\otimes_{K(\fg)} k}\otimes_K L_{\psi}.$$ 
More concretely,
$S(\psi)$ is defined as $\rH^0(K(\fg)/K, \coLie(A)^{\otimes_{K(\fg)} k}\otimes_K L_{\psi})$,
where $\sigma=(\fb, K(\fg)/K) \in\Gal(K(\fg)/K)$ acts on $\coLie(A)^{\otimes_{K(\fg)} k}\otimes_K L_{\psi}$ by
\[
\begin{split}
\coLie(A)^{\otimes_{K(\fg)} k}\otimes_K L_{\psi} \xrightarrow{\sigma^*\otimes 1}\coLie(A^{\sigma})^{\otimes_{K(\fg)} k}\otimes_K L_{\psi}\xrightarrow{(\eta_{\fb}^*)^{\otimes k}\otimes \psi(\fb)^{-1}} & \coLie(A/A[\fb])^{\otimes_{K(\fg)} k}\otimes_{K}L_{\psi}\\ 
=&\coLie(A)^{\otimes_{K(\fg)} k}\otimes_{K}L_{\psi}. 
\end{split}
\]
Here $\sigma: A^{\sigma}\cong A$, and $\eta_{\fb}: (A/A[\fb], \alpha+A[\fb]) \to (A^{\sigma}, \alpha^{\sigma})$ is  the unique isomorphism of the canonical CM pairs\footnote{The former is not a $K(\fg)$-morphism.}.  
For a finite extension $M$ of $L_{\psi}$, put $S_M(\psi)=S(\psi)\otimes_{L_{\psi}}M.$
Note that $\sigma(ax)=\sigma(a)\sigma(x)$ for $a\in K(\fg),\ x\in \coLie(A)^{\otimes_{K(\fg)} k}\otimes_K L_{\psi}$, and that 
 $S(\psi)$ is independent of the choice of $\fg$ up to canonical isomorphism.
 
In view of the above construction, we have the following.
\begin{lem}\label{properties-of-S}\leavevmode 
\begin{enumerate}
\item For a finite order character $\chi$ of $G_K$ of conductor dividing $\fg$ and a finite extension $M$ of $L_{\psi}L_{\chi}$ inside $\C$,
$S_M(\psi\chi)$ is identified with $\rH^0(G_K, (S(\psi)\otimes_{L_{\psi}}M)\otimes_KK(\fg)),$
where $G_K$ acts on $(S(\psi)\otimes_{L_{\psi}}M)\otimes_KK(\fg)$ by
\[
\sigma (a\otimes b) =\chi^{-1}(\sigma)a\otimes b^{\sigma} \quad (\sigma\in \Gal(K(\fg^{} )/K),\ a\in  S(\psi)\otimes_{L_{\psi}}M,\ b \in K(\fg^{})). 
\]
\item
If $\psi^{\prime}$ is of infinity type $(k^{\prime},0)$ with $k^{\prime}\ge 1$ and conductor dividing $\fg$,
then for a finite extension $M$ of $L_{\psi}L_{\psi^{\prime}}$, 
we have 
\begin{equation}\label{product-S}
S(\psi\psi^{\prime})\otimes_{L_{\psi}L_{\psi^{\prime}}} M=S_M(\psi)\otimes_{M} S_M(\psi^{\prime}).
\end{equation}
\end{enumerate}
\end{lem}
Note that  $S(\psi)$  gives an $L_{\psi}$-structure to $D_{\dR}^j(K\otimes\Q_p, V_{L_{\psi,\lambda}}(\psi))$ for $1\le j \le k$:
the comparison isomorphism 
\[
D_{\dR}\left(K(\fg)\otimes\Q_p, \rH^1_{\et}(A\otimes_{K(\fg)} \overline{K(\fg)}_w, \Q_p )\right)\cong \rH^1_{\dR}(A/K(\fg))\otimes_{\Q}\Q_p,
\]
where $w$ is a place of the algebraic closure  $\overline{K(\fg)}\subseteq \C$ above $p$, 
induces
\[
D^j_{\dR}(K(\fg)\otimes\Q_p, V_{L_{\psi,\lambda}}(\psi)) \cong \coLie(A)^{\otimes_{K(\fg)}k }\otimes_KL_{\psi,\lambda} \]
and
\begin{equation}\label{p-adicHodge}
D_{\dR}^j(K\otimes\Q_p, V_{L_{\psi,\lambda}}(\psi)) \cong S(\psi)\otimes_{L_{\psi}} L_{\psi, \lambda}. 
\end{equation}

The period map
\[
\per_A: \coLie(A) \to \rH^1(A(\C), \C)=\mathrm{Hom}(\rH_1(A(\C), \Z), \C), \quad \omega \mapsto \left(a \mapsto \int_a\omega\right )
\]
induces $\coLie(A)^{\otimes_{K(\fg)}k} \to \rH^1(A(\C), \Q)^{\otimes_K k}\otimes_K \C$ and
\[
\per_{\psi}: S(\psi) \to V_{\C}(\psi).
\]

\subsubsection{de Rham realisation: the infinity type $(k+1,-k)$ case}
Suppose that $\psi$ is of infinity type $(k+1,-k)$ and of the form $\varphi^{k+1}\bar{\varphi}^{-k}\chi=\varphi^{2k+1}\chi/\chi_{\cyc}^{k}$ with $k\ge 0$,
where $\varphi$ is of infinity type $(1,0)$ introduced before,
 $\chi$ is a finite character, and $\chi_{\cyc}$ denotes the cyclotomic character.

For a finite extension $M$ of $L_{\psi}$, put
\[
S_M(\psi)=S_M(\varphi^{2k+1}\chi).
\]
We note  the canonical identification
\begin{equation}\label{identification-betti}
\begin{split}
V_{L_{\psi}}(\psi)&=\rH^1(A(\C), \Q)^{\otimes_K k+1}\otimes_{K}\overline{\rH^1(A(\C), \Q)}^{\otimes_K -k}\otimes_K L_{\psi}\\
&=\rH^1(A(\C), \Q)^{\otimes_K 2k+1}\otimes_{K}\left(\rH^1(A(\C), \Q)\otimes_K\overline{\rH^1(A(\C), \Q)}\right)^{\otimes_K -k} \otimes_KL_{\psi}\\
&\cong\rH^1(A(\C), \Q)^{\otimes_K 2k+1}\otimes_{K}L_{\psi}=V_{L_{\psi}}(\varphi^{2k+1}\chi),
\end{split}
\end{equation}
where the last isomorphism is due to the isomorphism 
\begin{equation}\label{K-linear}
\rH^1(A(\C), \Q)\otimes_K\overline{\rH^1(A(\C), \Q)}\cong K
\end{equation} induced by the base change 
\begin{equation}\label{base-change}
( \ , \ )_{\rH^1, K}:  (\rH^1(A(\C), \Q)\otimes_\Q K) \times  (\rH^1(A(\C), \Q)\otimes_{\Q}K) \to  K
\end{equation}
of the canonical one 
and the decomposition
\begin{equation}\label{KxKdecomposition}
\rH^1(A(\C), \Q)\otimes_\Q K=\rH^1(A(\C), \Q)\oplus  \overline{\rH^1(A(\C), \Q)}
\end{equation}which are compatible with $K$-action (the $K$-action on the left hand side is given by the one on the second component in the tensor product). 

Then via  the de Rham version of the identification (\ref{identification-betti}),
 the comparison (\ref{p-adicHodge})  implies that 
$S_{L_{\psi}}(\psi)$ is a one-dimensional $L_{\psi}$-space which canonically gives an $L_{\psi}$-structure
on $D^j_{\mathrm{dR}}(K\otimes \Q_p, V_{L_{\psi, \lambda} }(\psi))$  for $1-k\le j \le 1+k$. 
Moreover, the \'etale version of (\ref{identification-betti}) induces $$V_{L_{\psi, \lambda} }(\psi)\cong V_{L_{\psi, \lambda} }(\varphi^{2k+1}\chi)(k).$$ 

\subsection{Rubin-type $p$-adic $L$-function and Hecke $L$-values}\label{interpolation-section}
\subsubsection{Explicit reciprocity law}
Let $\fg=f\CO_K$ where $f$ is a positive integer.
Then, $(A,\alpha)$ is canonically isomorphic to $(\C/\CO_K, 1/f)$, and $\rH_1^{}(A(\C), \Z)=\Hom_{\Z}(\rH^1(A(\C),\Z),\Z )$ is 
an $\CO_K$-module 
free of rank one.

Let $\delta \in \rH_1^{}(A(\C), \Z)$ be the $\CO_K$-basis which corresponds to $1\in \CO_K$ under  the isomorphism $\rH_1^{}(A(\C), \Z)=\rH_1^{}(\C/\CO_K, \Z)\cong \CO_K$. 
 Let $\delta^{\otimes -1} \in \rH_1^{}(A(\C), \Z)^{\otimes_{\CO_K} -1}$ be its dual basis. 
Let $t \in \rH^1_{}(A(\C), \Q)^{}$ be the image of $2^{-1}\delta^{\otimes -1}$
under
\[
\rH_1^{}(A(\C), \Q)^{\otimes_{K} -1} \xrightarrow{\mathrm{Tr}_{K/\Q}} \Hom_{\Q}(\rH_1^{}(A(\C), \Q), \Q) = \rH^1(A(\C),\Q), 
\] 
where the arrow sends $a$ to $\mathrm{Tr}_{K/\Q}\circ a.$
Denote by $\bar{t} \in \overline{\rH^1_{}(A(\C), \Q)}^{}:=\rH^1_{}(A(\C), \Q)\otimes_{K, c}K$ the element $t \otimes 1.$
\begin{lem}\label{comparison-periods}
Under the map (\ref{K-linear}), 
the image of 
$t \otimes \overline{t}$  is given by $ -\sqrt{d_K},$
where $d_K<0$ denotes the fundamental discriminant of $K$.
\end{lem}
\begin{proof}
Let $\theta \in \CO_K \subseteq \C$ be an element such that $\CO_K=\Z \oplus \Z\theta$ and $\mathrm{Im}(\theta)>0$.  
The alternating pairing $( \ , \  )_{\rH^1, K}$ is induced by duality and (the base change of ) the alternating pairing
\[
\langle \ , \ \rangle_{\rH_1}:\wedge^2_{\Q}\rH_1(A(\C), \Q) \to \Q,\quad z \wedge w \mapsto \frac{z\bar{w}-\bar{z}w}{\theta-\bar{\theta}}=\frac{1}{\sqrt{d_K}}(z\bar{w}-\bar{z}w)
\]where $\rH_1(A(\C), \Q)$ is identified with $K$ as above. 
Since $\langle \delta , \theta\delta \rangle_{\rH_1}=-\langle \delta , \bar{\theta}\delta \rangle_{\rH_1}=-1,$
 we have $\langle \delta, \sqrt{d_K}\delta \rangle_{\rH_1}=-2.$
So $( \ , \  )_{\rH^1, K}$ is given by the base change of the inverse map of
\[
\Q \to \wedge_{\Q}^2\rH^1(A(\C),\Q),\quad 1 \mapsto \frac{-1}{2}(\sqrt{d_K}\delta)^{\vee}\wedge \delta^{\vee} 
\]
where $\delta^{\vee}, (\sqrt{d_K}\delta)^{\vee} \in \rH^1(A(\C),\Q)=\Hom_{\Q}(\rH_1(A(\C),\Q),\Q)$ is the dual basis of $\delta, \sqrt{d_K}\delta \in \rH_1(A(\C),\Q).$

Under (\ref{KxKdecomposition}),
the element $\bar{t}$ corresponds to the element 
\[
\frac{t\otimes\sqrt{d_K}-(\sqrt{d_K}^*t)\otimes 1}{2\sqrt{d_K}}\in \rH^1(A(\C),\Q)\otimes_{\Q}K,
\]where $\sqrt{d_K}^*$ denotes the endomorphism of $\rH^1(A(\C),\Q)$ induced by $\sqrt{d_K} \in \CO_K$.
Since 
\[t(\delta)=1,\quad t(\sqrt{d_K} \delta)=0,\quad
\sqrt{d_K}^*t(\delta)=0,\quad \sqrt{d_K}^*t(\sqrt{d_K} \delta)={d_K},
\]
we have
\[
\begin{split}
(t\otimes_{\Q} 1) \wedge_{\Q} \left(\frac{t\otimes\sqrt{d_K}-(\sqrt{d_K}^*t)\otimes 1}{2\sqrt{d_K}}\right)
=&-(t\wedge_{\Q} \sqrt{d_K}^*t ) \otimes (2\sqrt{d_K})^{-1}\\
=&-(\delta^{\vee}\wedge_{\Q} (\sqrt{d_K}\delta)^{\vee}) \otimes 2^{-1}\sqrt{d_K}.\\
\end{split}
\]
Hence,
\begin{equation}\label{image-of-cup}
(t, \overline{t})_{\rH^1, K}=-\sqrt{d_K}.
\end{equation}
\end{proof}

Let $\eta: \Gal(K(\fg)/K) \to L_{\lambda}^{\times}$ be the character such that $\varphi\eta^{-1}: G_K \to L_{\lambda}^{\times}$ factors through $\Gamma$. 
Put $$\varphi_{\ac}=\varphi/\bar{\varphi}, \quad \eta_{\ac}=\eta/\overline{\eta}.$$

For $k\ge 0$,
regarding $t$ as an element  in $V_{L_{\lambda}}(\varphi\eta^{-1}),$
consider the map
\[
\bfH^1_{f\CO_K}(V_{L_{\lambda}}(\varphi)(1)) \xrightarrow{\otimes t^{\otimes k} \bar{t}^{\otimes -k} } 
\bfH^1_{f\CO_K}(V_{L_{\lambda}}(\varphi^{k+1}\bar{\varphi}^{-k} \eta_{\ac}^{-k})(1))
\to \rH^1(K, V_{L_{\lambda}}(\varphi^{k+1}\bar{\varphi}^{-k}\eta_{\ac}^{-k} )(1)\otimes_{\CO} \Lambda_{\ac} ), 
\]
where the last map is induced by the restriction map and Shapiro's lemma. 
It induces  
\begin{equation}\label{infinite-twists}
\Tw_{t}^k: \rH^1(K, T^{\otimes -1}(1)\otimes_{\CO} \Lambda_{\ac} )\to \rH^1(K, V_{L_{\lambda}}(\varphi^{k+1}\bar{\varphi}^{-k}\eta_{\ac}^{-k} )(1)\otimes_{\CO} \Lambda_{\ac} ).
\end{equation}
For a finite character $\chi$ of $\Gamma$, note that 
 $$V^{\otimes -1}(1)\otimes \Lambda\otimes_{\Lambda, \varphi_{\ac}^k\eta_{\ac}^{-k}\chi} L_{\lambda}(\mathrm{Im}(\chi))\cong V_{L_{\lambda}(\mathrm{Im}(\chi))}(\varphi\varphi_{\ac}^{k}\eta_{\ac}^{-k}\chi )(1).$$
 
 We have the following explicit reciprocity law.
\begin{prop}\label{higher-explicit-reciprocity-law}
Let $f$ and $t$ be as above.
Let $\omega $ be a non-zero element in $S(\varphi)$ and $\Omega\in \C^{\times}$ such that $\per_{\varphi}(\omega)=\Omega t$.
 For any finite character $\chi$ of $\Gamma$,
the image of $z_{p^{\infty}f}^{t, \ac} \in \rH^1(K, T^{\otimes -1}(1)\otimes_{\CO} \Lambda_{\ac} )$
under the composite 
\begin{equation}\label{twist-and-exp}
\begin{split}
\rH^1(K_p, T^{\otimes -1}(1)\otimes_{\CO} \Lambda_{\ac} )\xrightarrow{\Tw_{t}^k} &\rH^1(K_p, V_{L_{\lambda}}(\varphi^{k+1}\bar{\varphi}^{-k}\eta_{\ac}^{-k} )(1)\otimes_{\CO} \Lambda_{\ac} )\\
 \to  & \rH^1(K_p, V_{L_{\lambda}(\mathrm{Im}(\chi))}(\varphi^{k+1}\bar{\varphi}^{-k}\eta_{\ac}^{-k}\chi)(1))\\
 \xrightarrow{\exp^*} 
 & D_{\dR}^0(V_{L_{\lambda}(\mathrm{Im}(\chi))} (\varphi^{k+1}\bar{\varphi}^{-k}\eta_{\ac}^{-k}\chi^{})(1))\\
=& D_{\dR}^1(V_{L_{\lambda}(\mathrm{Im}(\chi))} (\varphi^{k+1}\bar{\varphi}^{-k}\eta_{\ac}^{-k}\chi^{}))
\end{split}
\end{equation}
is given by 
\begin{equation}\label{reciprocity-law}
\left(-\frac{2\pi}{\sqrt{|d_K|}}\right)^{k}\frac{L_{pf}(\varphi^{2k+1}\eta_{\ac}^{-k}\chi, k+1 )}{\Omega^{2k+1}} \cdot \omega^{\otimes 2k+1},
\end{equation}
where
the second arrow in (\ref{twist-and-exp}) is given by the base change with respect to $\Lambda_{\ac}\xrightarrow{\chi}L_{\lambda}(\mathrm{Im}(\chi))$.
\end{prop}
\begin{remark}\label{omega-in-S}
The element  $\omega^{\otimes 2k+1} \in S_{L}(\varphi^{k+1}\bar{\varphi}^{-k})$ may not lie in $D_{\dR}^0(V_{L_{\lambda}(\mathrm{Im}(\chi))} (\varphi^{k+1}\bar{\varphi}^{-k}\eta_{\ac}^{-k}\chi^{})(1))$,
but the term \eqref{reciprocity-law} does, under the inclusion
\[
D_{\dR}^0(V_{L_{\lambda}(\mathrm{Im}(\chi))} (\varphi^{k+1}\bar{\varphi}^{-k}\eta_{\ac}^{-k}\chi^{})(1)) \subseteq 
S_{L}(\varphi^{k+1}\bar{\varphi}^{-k})\otimes_L \overline{\Q}_p
\]
(cf.\ Lemma \ref{properties-of-S}). 

\end{remark}
\begin{proof}
For $\fa$ relatively prime to $6pf$,
denote by $\!_{\fa}z^{t^{\otimes 2k+1}}_{\cyc} $ the image of $(\!_{\fa}z_{p^{n}\ff})_n \in \bfH^1_{\ff}(\Z_p(1))$ of \cite[\S 15.5]{K}
under the composite 
\[
\begin{split}
\bfH^1_{f\CO_K}(L_{\lambda}(1)) \xrightarrow{\otimes t^{\otimes 2k+1}}  \bfH^1_{f\CO_K}(V_{L_{\lambda}}(\varphi^{2k+1}\eta_{\ac}^{-k}\chi)(1))\to &\rH^1(K, V_{L_{\lambda}}(\varphi^{2k+1}\eta_{\ac}^{-k}\chi)(1)\otimes_{\CO}\Lambda^{\cyc})\\
=&\rH^1(\Q, \tilde{V}_{L_{\lambda}}(\varphi^{2k+1}\eta_{\ac}^{-k}\chi)(1)\otimes_{\CO}\Lambda^{\cyc}),
\end{split}
\]
where $\tilde{V}_{L_{\lambda}}(\varphi^{2k+1}\eta_{\ac}^{-k}\chi)=\mathrm{Ind}_{K/\Q}({V}_{L_{\lambda}}(\varphi^{2k+1}\eta_{\ac}^{-k}\chi))$,
and $\Lambda^{\cyc}=\varprojlim_{n}\CO[\Gal(\Q(\zeta_{p^n} )/\Q)]$. 
Let $\zeta  \in \Z_p(1)$ be the generator such that 
the image of $t \otimes \overline{t}$ under the \'etale version
\[
\rH^1_{\et}(A\otimes_{K(\fg)}\overline{K(\fg)}_w,\Z_p ) \otimes_{\CO_K\otimes \Z_p} \overline{\rH^1_{\et}(A\otimes_{K(\fg)}\overline{K(\fg)}_w,\Z_p )} \to \CO_{K}\otimes\Z_p(-1)
\]
of  (\ref{K-linear})
 is given\footnote{The existence is due to (\ref{image-of-cup}).} by $-\sqrt{d_K}\zeta^{\otimes -1}$ (this choice is for the compatibility with the Betti-version).

By
 \cite[Thm.\ 12.5, Lem.\ 15.11,  (15.6.1)]{K},
 the image of $\!_{\fa}z^{t^{\otimes 2k+1}}_{\cyc} $ under the composite 
  \[
  \begin{split}
  \rH^1(\Q, \tilde{V}_{L_{\lambda}}(\varphi^{2k+1}\eta_{\ac}^{-k}\chi)(1)\otimes_{\CO}\Lambda^{\cyc}) 
  \xrightarrow{\otimes \zeta_{}^{\otimes k}}  & \rH^1(\Q, \tilde{V}_{L_{\lambda}}(\varphi^{2k+1}\eta_{\ac}^{-k}\chi)(k+1)\otimes_{\CO}\Lambda^{\cyc})  \\
  \to &  \rH^1(\Q, \tilde{V}_{L_{\lambda}}(\varphi^{2k+1}\eta_{\ac}^{-k}\chi)(k+1)) \\
  \xrightarrow{\exp^*} &  D^0_{\dR}(\Q_p, \tilde{V}_{L_{\lambda}}(\varphi^{2k+1}\eta_{\ac}^{-k}\chi)(k+1)) \\
  = &  D^{k+1}_{\dR}(\Q_p, \tilde{V}_{L_{\lambda}}(\varphi^{2k+1}\eta_{\ac}^{-k}\chi))\\
  = &  D^{k+1}_{\dR}(K\otimes\Q_p, {V}_{L_{\lambda}}(\varphi^{2k+1}\eta_{\ac}^{-k}\chi))
  \end{split}
  \]
lies in $ S(\varphi^{2k+1}\eta_{\ac}^{-k}\chi)$, 
and
its image under $ \per_{\varphi^{2k+1}\eta_{\ac}^{-k}\chi}$
 is given by 
 \begin{equation}\label{eq, Kato-ERL}
 (N(\fa)-\varphi^{2k+1}\eta_{\ac}^{-k}\chi(\fa))(2\pi\sqrt{-1})^kL_{pf}(\overline{\varphi}^{2k+1}\bar{\eta}_{\ac}^{-k}\bar{\chi}, k+1) \cdot t^{\otimes 2k+1} \in V_{\C}(\varphi^{2k+1}\eta_{\ac}^{-k}\chi).
 \end{equation}
 
Denote by  $\mathbf{1}(\!_{\fa}z^{t^{\otimes 2k+1}}_{\cyc}\otimes \zeta^{\otimes k})$ and 
$ \chi(\Tw_{t}^k(z_{p^{\infty}f}^{t,\ac}))$ the images of the element in $
  \rH^1(\Q, \tilde{V}_{L_{\lambda}}(\varphi^{2k+1}\eta_{\ac}^{-k}\chi)(k+1)\otimes_{\CO}\Lambda^{\cyc}) $ given  by $\!_{\fa}z^{t^{\otimes 2k+1}}_{\cyc}\otimes \zeta^{\otimes k}$
  and $\Tw_{t}^k(z_{p^{\infty}f}^{t,\ac}) \in \rH^1(K, V_{L_{\lambda}}(\varphi^{k+1}\bar{\varphi}^{-k} \eta_{\ac}^{-k})(1)\otimes_{\CO}\Lambda_{\ac})$ in
  \[
 \rH^1(\Q, \tilde{V}_{L_{\lambda}}(\varphi^{2k+1}\eta_{\ac}^{-k}\chi)(k+1))=\rH^1(K, {V}_{L_{\lambda}}(\varphi^{k+1}\bar
{\varphi}^{-k}\eta_{\ac}^{-k}\chi))
\]
respectively, 
where the identification is due to Shapiro's lemma and (\ref{identification-betti}). 
By definition, we have
  \[
 (N(\fa)-\varphi^{2k+1}\eta_{\ac}^{-k}\chi(\fa))^{-1}(-\sqrt{d_K} )^{-k} \mathbf{1}(\!_{\fa}z^{t^{\otimes 2k+1}}_{\cyc}\otimes \zeta^{\otimes k})= \chi(\Tw_{t}^k(z_{p^{\infty}f}^{t,\ac})) \in \rH^1(K, {V}_{L_{\lambda}}(\varphi^{k+1}\bar
{\varphi}^{-k}\eta_{\ac}^{-k}\chi)).
\]
Hence,
the proposition follows from \eqref{eq, Kato-ERL} by noting that  
$$\mathrm{per}_{\varphi^{2k+1}}(\omega^{\otimes 2k+1})=\Omega^{2k+1}t^{\otimes 2k+1}, \quad 
L_{pf}(\overline{\varphi}^{2k+1}\bar{\eta}_{\ac}^{-k}\bar{\chi}, k+1)=L_{pf}({\varphi}^{2k+1}{\eta}_{\ac}^{-k}{\chi}, k+1).$$ 
\end{proof}
\subsubsection{Interpolation property of the $p$-adic $L$-function}
Define 
\[
\begin{split}
\exp^*_{\varphi_{\ac}^{k}\eta_{\ac}^{-k}\chi, \omega}: \rH^1(K_p, T^{\otimes -1}(1)\otimes \Lambda_{\ac}) \xrightarrow{(\ref{twist-and-exp})} & D_{\dR}^0(V_{L_{\lambda}(\mathrm{Im}(\chi))} (\varphi^{k+1}\bar{\varphi}^{-k}\eta_{\ac}^{-k}\chi^{})(1))\\
 \xrightarrow{\subseteq }& S_{L}(\varphi^{k+1}\bar{\varphi}^{-k})\otimes_L \overline{\Q}_p \cong  \overline{\Q}_p, 
\end{split}
\]
where the last isomorphism is defined  by $ \omega^{\otimes 2k+1} \otimes a\mapsto a$. 
(Note that $\exp^*_{\varphi_{\ac}^{k}\eta_{\ac}^{-k}\chi, \omega}$ depends on both $t$ and $\omega$.)
Adapting the notation of \cite[Def.\ 2.10]{BKNO},
we define 
\[
\hat{\varepsilon}_p(\Ind_{K/\Q}(\varphi^{k+1}\overline{\varphi}^{-k}\eta_{\ac}^{-k}\chi))=\hat{\varepsilon}_p(\Ind_{K/\Q}(V_{L_{\lambda}}(\varphi^{k+1}\overline{\varphi}^{-k}\eta_{\ac}^{-k}\chi))) \in \{\pm 1\}.
\]
\begin{thm}\label{interpolation-rubin}
Let $f, t, \omega, \Omega $ be as in Proposition \ref{higher-explicit-reciprocity-law},
and $\varepsilon$ as in Definition \ref{def, epsilon}.
Suppose that  $T \subseteq V$ is a lattice such that $T^{\otimes -1}$ is generated by $t$. 
Let $k\ge 0$ be an integer and $\chi$ a finite order character of $\Gamma$ such that  $$\hat{\varepsilon}_p(\Ind_{K/\Q}(\varphi^{k+1}\overline{\varphi}^{-k}\eta_{\ac}^{-k}\chi))=\frac{\varepsilon(\varphi)}{\varepsilon_p(\Ind_{K/\Q}(\varphi))}.$$  
Then,
$\exp^*_{\varphi_{\ac}^{k}\eta_{\ac}^{-k}\chi, \omega}(v_{\varepsilon})\neq 0$, 
and the evaluation of the $p$-adic $L$-function $\mathscr{L}_{p,v_{\varepsilon},t}(\varphi)$ at the character $\varphi^{k}\bar{\varphi}^{-k}\eta_{\ac}^{-k}\chi$ of $\Gamma$,
which is defined as its image under the homomorphism  $\Lambda_{\ac}\to \overline{\Q}_p$ induced by $\varphi^{k}\bar{\varphi}^{-k}\eta_{\ac}^{-k}\chi$ regarded as a character on $\Gamma$,
  is given by
\[
\mathscr{L}_{p,v_{\varepsilon},t}(\varphi)(\varphi^{k}\bar{\varphi}^{-k}\eta_{\ac}^{-k}\chi)= 
\left(-\frac{2\pi}{\sqrt{|d_K|}}\right)^{k}\frac{L_{pf}(\varphi^{2k+1}\eta_{\ac}^{-k}\chi, k+1 )}{\exp^*_{\varphi_{\ac}^{k}\eta_{\ac}^{-k}\chi, \omega}(v_{\varepsilon})\Omega^{2k+1}}. 
\]
Moreover, $\mathscr{L}_{p,v_{\varepsilon},t}(\varphi)$ is non-zero. 
\end{thm}
\begin{proof} 
The non-vanishing of $\exp^*_{\varphi_{\ac}^{k}\eta_{\ac}^{-k}\chi, \omega}(v_{\varepsilon})$ follows from \cite[Cor.\ 3.2]{BKNO}. As for the interpolation property,
by definition
we have
\[
\mathscr{L}_{p,v_{\varepsilon},t}(\varphi)(\varphi^{k}\bar{\varphi}^{-k}\eta_{\ac}^{-k}\chi)\exp^*_{\varphi_{\ac}^{k}\eta_{\ac}^{-k}\chi, \omega}(v_{\varepsilon})=\exp^*_{\varphi_{\ac}^{k}\eta_{\ac}^{-k}\chi, \omega}(\Tw_{t}^k(z_{p^\infty f}^{t,\ac})),
\]
by which Proposition~\ref{higher-explicit-reciprocity-law} implies the desired interpolation property. 

In turn, the non-vanishing of $\mathscr{L}_{p,v_{\varepsilon},t}(\varphi)$ follows by taking $k=0$, and applying the generic  anticyclotomic non-vanishing of Hecke $L$-values in \cite{Ro}.
\end{proof}

\section{Signed main conjecture and the $p$-adic $L$-function}\label{s, rMC}
Based on the results of \S\ref{section-local-pt}, we introduce $\pm$-anticyclotomic Selmer groups and study their structures. 
The main result of this section establishes an Iwasawa main conjecture relating the characteristic ideal of one of the signed Selmer groups to the $p$-adic $L$-function introduced in \S\ref{general-interpolation} (see Theorems~\ref{rank} and~\ref{main-result}).

\subsection{$\pm$-Selmer conditions}\label{subsection-pm-condition}

\begin{defn}\label{pm-structure}
Applying Definition~\ref{def, pm} 
for $\psi=\varphi|_{K_p^{\times}}$ and identifying $\Gal(\Psi_{\infty}/\Psi)$ with $\Gal(K_{\infty}^{\ac}/K)$,
 we define $\rH^1_{\pm}(K_p, T^{\otimes -1}(1)\otimes_{\CO} \Lambda_{\ac})$ and
 $\rH^1_{\pm}(K_{n,p}^{\ac}, T^{\otimes -1}(1))$.
 Also define $\rH^1_{\pm}(K^{\ac}_{n,p}, W)$ as the orthogonal complement of $\rH^1_{\pm}(K^{\ac}_{n, p}, T^{\otimes -1}(1))$ with respect to the perfect pairing $$\rH^1_{}(K^{\ac}_{n,p}, W)\times \rH^1_{}(K^{\ac}_{n,p}, T^{\otimes -1}(1)) \to L_{\lambda}/\CO$$
induced by the local duality. 
\end{defn}

\begin{lem}\label{universal-norm}
We have $\varprojlim_n \rH^1_{\rf}(K_{n,p}^{\ac}, T^{\otimes -1}(1))=\{0\}.$
\end{lem}
\begin{proof}
By Lemma \ref{density-of-xi},
both $\Xi^{+}$ and $\Xi^-$ contain infinitely many elements. 
Then, by Theorem \ref{higher-weight-decomposition} and (\ref{saturated}),
for infinitely many $\chi_{} \in \Xi^{\mp}$, we have 
\[
\begin{split}
(\rH^1_{\pm}(K_p, T^{\otimes -1}(1)\otimes \Lambda_{\ac})\otimes \Q_{p}) \otimes_{\Lambda_{\ac}, \chi} \CO_{\chi}\
=&\rH^1_{\pm}(\Q_p, \Ind_{K_p/\Q_p}(V_{\varphi\chi}^{\otimes -1}(1)))\\
\cong & \rH^1_{\pm}(\Q_p, \Ind_{K_p/\Q_p}(\varphi\chi^{}))= \rH^1_{\rf}(\Q_p, \Ind_{K_p/\Q_p}(\varphi\chi^{})),
\end{split}
\]
where 
the isomorphism holds since $\Ind_{K_p/\Q_p}(\varphi\chi^{})$ is symplectic self-dual.
Hence, Theorem \ref{higher-weight-decomposition} implies that  
$$\varprojlim_n \rH^1_{\rf}(K_{n,p}^{\ac}, T^{\otimes -1}(1))\otimes \Q_p=0.$$
Since $\varprojlim_n \rH^1_{\rf}(K_{n,p}^{\ac}, T^{\otimes -1}(1))$ is $\Lambda_{\ac}$-torsion-free, we deduce the lemma. 
\end{proof}

\begin{cor}\label{bc-to-ac-sel}\leavevmode 
\begin{enumerate}
\item 
The natural map $X\otimes_{\Lambda}\Lambda_{\ac}\otimes \Q_{p} \to X^{\ac}\otimes \Q_p$ is an isomorphism. 
\item Under Assumption \ref{for-integrality},
the natural $\Lambda_{\ac}$-homomorphism $X\otimes_{\Lambda}\Lambda_{\ac}\to X^{\ac}$ is an isomorphism.
\end{enumerate}
\end{cor}
\begin{proof}
Put $I=\ker(\Lambda\to \Lambda_{\ac})$.

We have the following commutative diagram of $\Lambda_\ac$-modules  with exact rows
\[
\xymatrix{
 & 0 \ar[d]& 0 \ar[d] &  \\
0 \ar[r]& \Sel_{}(K_{\infty}^{\ac}, W) \ar[r]\ar[d]& \rH^1(K_S/K_{\infty}^{\ac}, W) \ar[r]\ar[d]^{(a)} &\bigoplus_{v\mid p\ff}\frac{\rH^1(K_{\infty,v}^{\ac}, W)}{\rH^1_{\rf}(K_{\infty,v}^{\ac}, W)}\ar[d]^{(b)} \\
0 \ar[r]& \Sel_{}(K_{\infty}^{}, W)[I] \ar[r]& \rH^1(K_S/K_{\infty}^{}, W)[I] \ar[r] & \bigoplus_{v\mid p\ff}\frac{\rH^1(K_{\infty}\otimes_{K_{\infty}^{\ac}}K_{\infty,v}^{\ac}, W)}{\rH^1_{\rf}(K_{\infty}\otimes_{K_{\infty}^{\ac}}K_{\infty,v}^{\ac}, W)}[I]. 
}
\]
Note that the kernel and cokernel of (a) are killed by some power of $p$ and that, by (\ref{vanishing-ramified}) and Lemmas \ref{inj-ram}, \ref{finiteness-at-p}, and \ref{universal-norm},
 the kernel of the  map (b) is finite. 
 Hence,  the snake lemma implies part 1). 
 As for part 2), note in addition that 
 the map $(a)$ is an isomorphism by (\ref{vanishing-p}) and (b) is injective. 
\end{proof}

\subsection{$\pm$-Selmer groups}\label{subsection-pm-selmer}
\begin{defn}
 For $n\ge 1$ and $M\in \{T^{\otimes -1}(1), W\}$,
define
\[
\Sel_{\pm}(K_{n}^{\ac},M)=\ker\left(\rH^1(K_{n}^{\ac}, M) \to \frac{\rH^1(K_{n,p}^{\ac}, M)}{\rH^1_{\pm}(K_{n,p}^{\ac}, M)} \times \prod_{v\nmid p}\frac{\rH^1(K_{n,v}^{\ac}, M)}{\rH^1_{\rf}(K_{n,v}^{\ac}, M)}\right),
\]
and 
\[
\scrS_{\pm}=\varprojlim_{K\subseteq F \subseteq K_{\infty}^{\ac}}\Sel_{\pm}(F, T^{\otimes -1}(1)),\quad 
 X_{\pm}=
\left(\varinjlim_{K\subseteq F \subseteq K_{\infty}^{\ac}}
 \Sel_{\pm}(F, W)\right)^{\vee}.
\]
\end{defn}
We have the following result about the structure of these Selmer groups. 
\begin{thm}\label{rank}\leavevmode 
\begin{enumerate}
\item We have $\rank_{\Lambda_{\ac}}(\scrS_{\pm})=\rank_{\Lambda_{\ac}}(X_{\pm}^{}).$
\item Let $\varepsilon$ be as in Definition \ref{def, epsilon}. 
Then 
 $\rank_{\Lambda_\ac} (X_{\varepsilon})=1$, 
$X_{-\varepsilon}$ is $\Lambda_{\ac}$-torsion,
$\scrS_{-\varepsilon}=\{0\},$
 and $\scrS_{\rel}^{\ac}=\scrS_{\varepsilon}.$
\end{enumerate}
\end{thm}
\begin{proof}  
In view of the Poitou--Tate global duality, we have the following short exact sequence of $\Lambda_\ac$-modules: 
\begin{equation}\label{duality-pm-rel-str}
0\to \scrS_{\pm} \to \scrS_{\rel}^{\ac} \to \frac{\rH^1(K_{p}, T^{\otimes -1}(1)\otimes_{\CO} \Lambda_{\ac})}{\rH^1_{\pm}(K_{p}, T^{\otimes -1}(1)\otimes_{\CO} \Lambda_{\ac})} \to X_{\pm} \to X_{\str}^{\ac} \to 0.
\end{equation}
 Then Proposition \ref{half-IMC}, Theorem \ref{higher-weight-decomposition} and Remark \ref{signed-condition-p-invert} imply part 1).

By Proposition \ref{prop, zeta-in-epsilon}, 
we have 
\begin{equation}\label{zeta-in-epsilon}
\loc_{p}(z_{p^{\infty}\ff }^{t,\ac}) \in \rH^1_{\varepsilon}(K_p, T^{\otimes -1}(1)\otimes \Lambda_{\ac}),
\end{equation}
which implies that $z_{p^{\infty}\ff }^{t,\ac}$ lies in $\scrS_{\varepsilon}.$
Then, by  Proposition \ref{half-IMC}, we have $\rank_{\Lambda_{\ac}}(\scrS_{\varepsilon})=\rank_{\Lambda_{\ac}}(\scrS_{\rel}^{\ac})=1.$
Hence, part 1) implies that $\rank_{\Lambda_\ac} (X_{\varepsilon})=1.$
 
 We next show that $\scrS_{-\varepsilon}=0$, which also implies that $\rank_{\Lambda_{\ac}}(X_{-\varepsilon})=0.$
Since the $\Lambda_{\ac}$-module $\scrS_{\rel}^{\ac}$ is torsion-free,
it suffices to show that $\rank_{\Lambda_{\ac}}(\scrS_{-\varepsilon})=0.$
Assume that $\rank_{\Lambda_{\ac}}(\scrS_{-\varepsilon})>0$.
Then,
since $\rank_{\Lambda_{\ac}}(\scrS_{\rel}^{\ac})=1$,
we have $\rank_{\Lambda_{\ac}}(\scrS_{-\varepsilon})=\rank_{\Lambda_{\ac}}(\scrS_{\varepsilon})=1$
and $\rank_{\Lambda_{\ac}}(\scrS_{+}\cap \scrS_{-})=1$.
By Theorem \ref{higher-weight-decomposition} and Remark \ref{signed-condition-p-invert}, 
we have $(\scrS_{+}\cap \scrS_{-})\otimes_{\Z_p} \Q_p=\scrS_{\str}^{\ac}\otimes_{\Z_p}\Q_p.$
However, this implies that $\rank_{\Lambda_{\ac}}(\scrS_{\str}^{\ac})\ge 1,$
which contradicts Proposition \ref{half-IMC}.

It remains to show that $\scrS_{\rel}^{\ac}=\scrS_{\varepsilon},$
which follows from the exact sequence
\[
0\to \scrS_{\varepsilon}\to \scrS_{\rel}^{\ac} \to \frac{\rH^1(K_p, T^{\otimes -1}(1)\otimes_{\CO} \Lambda_{\ac} )}{\rH^1_{\varepsilon}(K_p, T^{\otimes -1}(1) \otimes_{\CO} \Lambda_{\ac})},
\]
and the facts: $\rH^1(K_p, T^{\otimes -1}(1) \otimes_{\CO} \Lambda_{\ac})\otimes \Q_p/\rH^1_{\varepsilon}(K_p, T^{\otimes -1}(1)\otimes_{\CO} \Lambda_{\ac})\otimes \Q_{p} = \rH^1_{-\varepsilon}(K_p, T^{\otimes -1}(1)\otimes_{\CO}\Lambda_{\ac})\otimes \Q_{p}$ is $\Lambda_{\ac}\otimes \Q_{p}$-free (cf.\ Theorem \ref{higher-weight-decomposition} and Remark \ref{signed-condition-p-invert}),
$\rH^1(K_p, T^{\otimes -1}(1) \otimes_{\CO} \Lambda_{\ac})/\rH^1_{\varepsilon}(K_p, T^{\otimes -1}(1)\otimes_{\CO} \Lambda_{\ac})$ is $p$-torsion-free,
and $\rank_{\Lambda_{\ac}}(\scrS_{\rel}^{\ac})=\rank_{\Lambda_{\ac}}(\scrS_{\varepsilon})=1$. (The latter implies that $\scrS_{\rel}^{\ac}/\scrS_{\varepsilon}$ is $\Lambda_{\ac}$-torsion.)
\end{proof}

\subsection{Signed anticyclotomic main conjecture}\label{subsection-IMC}
The main result of this section is the following.
\begin{thm}\label{main-result}\leavevmode 
\begin{enumerate}
\item We have 
$$\Ch_{\Lambda_{\ac}}(X_{-\varepsilon})\otimes\Q_p=(\mathscr{L}_{p,v_{\varepsilon},t}(\varphi)) \subset 
\Lambda_{\ac}\otimes \Q_{p}.$$
\item Suppose that Assumption \ref{for-integrality} holds
and $p$ does not divide the number of roots of unity in the Hilbert class field of $K$. Then 
$$\Ch_{\Lambda_{\ac}}(X_{-\varepsilon})=(\mathscr{L}_{p,v_{\varepsilon},t}(\varphi)) \subset 
\Lambda_{\ac}.$$

\end{enumerate}
\end{thm}

\begin{proof}
We first consider part 2).

Since $\rH^1_{}(K_p,T^{\otimes-1}(1)\otimes_{\CO} \Lambda_{\ac})/\rH^1_{-\varepsilon}(K_p,T^{\otimes-1}(1)\otimes_{\CO} \Lambda_{\ac})=\rH^1_{\varepsilon}(K_p,T^{\otimes-1}(1)\otimes_{\CO} \Lambda_{\ac})$,
by Theorem \ref{rank} and (\ref{duality-pm-rel-str})
we have an exact sequence of $\Lambda_{\ac}$-modules: 
\begin{equation}\label{rel-epsilon}
0\to \scrS_{\rel}^{\ac} \to \rH^1_{\varepsilon}(K_p,T^{\otimes-1}(1)\otimes_{\CO} \Lambda_{\ac}) \to  X_{-\varepsilon} \to X_{\str}^{\ac}\to 0.
\end{equation}
Note that by Proposition \ref{half-IMC} (1) and (\ref{zeta-in-epsilon}),
 the quotient $\rH^1_{\varepsilon}(K_p,T^{\otimes-1}(1)\otimes_{\CO} \Lambda_{\ac})/\scrS_{\rel}^{\ac}$ --- via the localisation map at $p$ --- 
is $\Lambda_{\ac}$-torsion. Moreover, (\ref{rel-epsilon}) implies that 
\begin{equation}\label{char--epsilon-str}
\begin{split}\Ch_{\Lambda_{\ac}}(X_{-\varepsilon})&=\Ch_{\Lambda_{\ac}}(X^{\ac}_{\str})\Ch_{\Lambda_{\ac}}\left(\rH^1_{\varepsilon}(K_p,T^{\otimes-1}(1)\otimes_{\CO} \Lambda_{\ac})/\scrS_{\rel}^{\ac}\right)\\
&=\Ch_{\Lambda_{\ac}}\left(\scrS_{\varepsilon}^{}/ \Lambda_{\ac} z_{p^{\infty}\ff }^{t,\ac}\right)\Ch_{\Lambda_{\ac}}\left(\rH^1_{\varepsilon}(K_p,T^{\otimes-1}(1)\otimes \Lambda_{\ac})/\scrS_{\varepsilon}^{}\right),
\end{split}
\end{equation}
where the second equality follows from Theorem \ref{anticyc-zetaIMC} and Theorem \ref{rank}.

Hence, part (2) is a consequence of (\ref{char--epsilon-str}), the exact sequence\footnote{Note that  the three terms are all $\Lambda_{\ac}$-torsion.}
\[
0\to \scrS_{\varepsilon}^{}/ \Lambda_{\ac} z_{p^{\infty}\ff}^{t,\ac} \to \rH^1_{\varepsilon}(K_p,T^{\otimes-1}(1)\otimes_{\CO} \Lambda_{\ac}) / \Lambda_{\ac} \loc_p(z_{p^{\infty}\ff}^{t,\ac}) \to \rH^1_{\varepsilon}(K_p,T^{\otimes-1}(1)\otimes_{\CO} \Lambda_{\ac})/\scrS_{\varepsilon}^{} \to 0
\]
and the fact 
\[
\Ch_{\Lambda_{\ac}}\left(\rH^1_{\varepsilon}(K_p,T^{\otimes-1}(1)\otimes \Lambda_{\ac})/ \Lambda_{\ac} \loc_p(z_{p^{\infty}\ff}^{t,\ac})\right)=(\mathscr{L}_{p,v_{\varepsilon},t}(\varphi)), 
\] the latter being immediate from Definition~\ref{def, RpL}.

Part 1) follows from the same argument as above by inverting $p$ and using 
 $$\Ch_{\Lambda}(\scrS_{\rel}/\fz_{\fa})\otimes \Q_p=\Ch_{\Lambda}(X_{\str})\otimes \Q_p$$
  (cf.~\cite[Thm.\ 5.2]{JL-K}). 
\end{proof}

\section{Asymptotic  Selmer and Mordell--Weil ranks}\label{s, SMW}
This section presents a control theorem for the signed Selmer groups and its consequences for asymptotic Selmer and Mordell--Weil ranks (see Theorems~\ref{controlling-thm} and ~\ref{vGZK}).
\subsection{A control theorem}
\begin{thm}\label{controlling-thm}
Suppose that Assumption \ref{for-integrality} (1) holds.
Then for $n\ge 1$,
the natural map 
\begin{equation}\label{control-restriction-pm}
\Sel_{\pm}(K_n^{\ac},W) \to \Sel_{\pm}(K_{\infty}^{\ac},W)[\omega_n^{}]
\end{equation}
is injective, and its cokernel is finite and of bounded order as $n$ varies. 
Here $\omega_n$ is a generator of the kernel of $\Lambda_{\ac}\to \CO[\Gal(K_n^{\ac}/K)]$.
\end{thm}

\begin{proof}
 In view of (\ref{local-vanishing}),
we have the following commutative diagram  with exact rows: 
{\footnotesize
\[
\xymatrix{
0 \ar[r]& \Sel_{\pm}(K_{n}^{\ac}, W) \ar[r]\ar[d]& \rH^1(K_S/K_{n}^{\ac}, W) \ar[r]\ar[d]^{(a)} &\dfrac{\rH^1(K_{n,p}^{\ac}, W)}{\rH^1_{\pm}(K_{n,p}^{\ac}, W)}\oplus \bigoplus_{v\mid \ff^{(p)}}\rH^1(K_{n}^{\ac}\otimes_K K_v, W)\ar[d]^{(b)} \\
0 \ar[r]& \Sel_{\pm}(K_{\infty}^{\ac}, W)[\omega_{n}^{}] \ar[r]& \rH^1(K_S/K_{\infty}^{\ac}, W)[\omega_n^{}] \ar[r] & \dfrac{\rH^1(K_{\infty,p}^{\ac}, W)}{\rH^1_{\pm}(K_{\infty,p}^{\ac}, W)}[\omega_n^{}]\oplus \bigoplus_{v\mid \ff^{(p)}}\rH^1(K_{\infty}^{\ac}\otimes_K K_v, W)[\omega_n^{}]. 
}
\]}
Here $K_S$ denotes the maximal algebraic extension of $K$ unramified outside $p\ff$,
 $\ff^{(p)}$ is the prime-to-$p$ part of $\ff$,
the vertical arrows are induced by the restriction map with respect to the extension $K_{\infty}^{\ac}/K_n^{\ac}$
  and $v$ ranges over all primes of $K$ dividing $\ff^{(p)}.$
In view of (\ref{vanishing-p}),
the above map (a) is bijective,
and hence the snake lemma reduces the assertion to bounding the kernel of (b).

Note that the map
 $$\frac{\rH^1(K_{n,p}^{\ac}, W)}{\rH^1_{\pm}(K_{n,p}^{\ac}, W)} \to 
\frac{\rH^1(K_{\infty,p}^{\ac}, W)}{\rH^1_{\pm}(K_{\infty,p}^{\ac}, W)}[\omega_n^{}]
$$ may be identified with the dual of 
$\rH^1_{\pm}(K_p, T^{\otimes-1}(1)\otimes \Lambda_{\ac})/\omega_n \to 
\rH^1_{\pm}(K_{n,p}^{\ac}, T^{\otimes-1}(1))$,
which is an isomorphism (cf.\ Theorem \ref{higher-weight-decomposition}).

It  remains to bound the kernel of $\rH^1(K_{n}^{\ac}\otimes_K K_v, W)\to \rH^1(K_{\infty}^{\ac}\otimes_K K_v, W)$ for each $v\mid \ff^{(p)}.$
Let $D_{v}$ be  the decomposition group  of $\Gal(K_{\infty}^{\ac}/K)$ at $v$.
Since $D_v$ is a closed subgroup of $\Gal(K_{\infty}^{\ac}/K)\cong \Z_p$,
either $D_v=0$ or $[\Gal(K_{\infty}^{\ac}/K): D_{v}]=p^{n(v)}$  for some $n(v) \ge 0$. 
 
 If $D_v=0$, i.e.\ $K_{\infty}^{\ac}\otimes_K K_v$ is an infinite product of $K_v$,
then  the injectivity of the $v$-component of (b) is clear.
Hence, we may assume that $[\Gal(K_{\infty}^{\ac}/K): D_{v}]=p^{n(v)}$.
Then $v$ splits completely in $K_{n(v)}^{\ac}/K$, 
and for $n\ge n(v)$ a  prime of $K^{\ac}_{n(v)}$ above $v$ is inert in $K^{\ac}_n/K^{\ac}_{n(v)}$.
Since
for $n\ge n(v)$, we have 
\[
K_n^{\ac}\otimes_KK_v=K_n^{\ac}\otimes_{K^{\ac}_{n(v)}}K^{\ac}_{n(v)}\otimes_KK_v 
=K_n^{\ac}\otimes_{K^{\ac}_{n(v)} }\bigoplus_{w\mid v}K^{\ac}_{n(v),w}
=\bigoplus_{w\mid v} K_n^{\ac}\otimes_{K^{\ac}_{n(v)} }K^{\ac}_{n(v),w},
\]
the $v$-component of (b) may be identified with
\[
\bigoplus_{w\mid v}\rH^1(K_{n,w}^{\ac}, W)\to \bigoplus_{w\mid v}\rH^1(K_{\infty,w}^{\ac}, W),
\]
where $w $ ranges over all primes of  $K^{\ac}_{n(v)}$ above $v$,
$K^{\ac}_{m,w}:=K^{\ac}_m\otimes_{K^{\ac}_{n(v)}}K^{\ac}_{n(v),w}$ for $m\ge n(v)$ and $m=\infty$,
which is a field.
Hence,
 it suffices to show that for $n\ge n(v)$ and a prime $w$ of $K_{n(v)}^{\ac}$ above $v$,
the kernel of 
$\rH^1(K_{n,w}^{\ac}, W)\to \rH^1(K_{\infty,w}^{\ac}, W)$ is finite, and 
 its order  is bounded independently of $n$. 
Since this  kernel may be identified with a quotient of
$\rH^0(K_{\infty,w}^{\ac}, W)$,
Lemma \ref{finiteness-at-ramified} concludes the proof. 
\end{proof}
\begin{remark} If the hypothesis of Theorem~\ref{controlling-thm} is not satisfied, i.e.\ 
$K\otimes \Q_p\cong \Q_3(\sqrt{-3})$ and 
$\rH^0(K_{\fp}, T^{\otimes -1}(1)\otimes \CO/\fm)\neq 0$, then 
the kernel and cokernel of (\ref{control-restriction-pm}) are finite.
\end{remark}
\subsection{Anticyclotomic variation of Selmer ranks}
In this subsection we assume that $p$ is \textit{ramified} in $K$. 
Moreover, without loss of generality\footnote{This is a consequence of \cite[Lem.\ 7.5]{BKNO}: we may twist it by a finite character on $\Gamma$ to obtain  a Hecke character satisfying the underlying hypothesis (cf.~the proof of Lemma \ref{density-of-xi}). }, we suppose that
\begin{equation}\label{minimal-conductor}
\fp \parallel \ff,
\end{equation}
i.e.\ $\fp$ exactly divides $\ff$.
Until the end of this subsection, 
replacing  $L\subseteq \C$ by its quadratic extension (if necessary),
we assume that $L$ contains the quadratic extension of $\Q$ inside $\Q(e^{2\pi i/p})$.
\subsubsection{A variant of the control theorem}
Let $\delta \in K_p=K\otimes \Q_p$ be a uniformiser such that $\delta^2\in \Q_p$ 
so that $-\delta^2=N_{K_p/\Q_p}(\delta).$
By using \cite[Cor.\ 7.8 ii)]{BKNO},
 we fix a system of finite order characters $(\chi_n)_{n\ge 1}$ of\footnote{Recall that $p\nmid h_K$.} $\Gal(K_{\infty}^{\ac}/K)=\Gal(K_{\infty,p}^{\ac}/K_p)$ 
such that for $n\ge 1$ we have $\chi_n\in \Xi_n^+$  
and $$\chi_{n+1}^{-\delta^2}=\chi_{n}.$$

As in (\ref{def-omega}),
for $n\ge 1$,
we define elements $$\Phi_k^{\pm}(\gamma), \quad 
\omega_n^+=\prod_{0\le k\le n}\Phi_{k}^+(\gamma), \quad
\omega_n^- =\prod_{0\le k\le n}\Phi_{k}^-(\gamma) \in \CO[\Gal(K_{\infty}^{\ac}/K)],$$
where  $\Phi_0^{\varepsilon_p(\mathrm{Ind}_{K/\Q}(\varphi) )}=\gamma-1$ and $\Phi_0^{-\varepsilon_p(\mathrm{Ind}_{K/\Q}(\varphi) )}=1.$

\begin{cor}\label{variant-control}
For $n\ge 1$, the natural map 
$$\Sel_{\pm}(K_n^{\ac},W) [\iota\omega_n^{\mp}]\to \Sel_{\pm}(K_{\infty}^{\ac},W)[\omega_n^{\mp}]$$
is injective, and its cokernel is finite and of bounded order as $n$ varies. 
Here $\iota: \Lambda_{\ac}\to \Lambda_{\ac}$ denotes the involution induced by $g \mapsto g^{-1}$ for $g\in \Gamma$.  
\end{cor}
\begin{proof}
This immediately follows from Theorem \ref{controlling-thm}: 
just take the part killed by $\omega_n^{\pm}$.
\end{proof}

\subsubsection{Asymptotic ranks of Selmer groups}

\begin{lem}\label{local-pm-to-f}
We have $$(\iota \omega_{n}^{\pm})\rH^1_{\pm}(K_{n,p}^{\ac},W) \subseteq \rH^1_{\rf}(K_{n,p}^{\ac},W), \quad
(\iota \omega_{n}^{\pm})\rH^1_{\rf}(K_{n,p}^{\ac},W) \subseteq \rH^1_{\pm}(K_{n,p}^{\ac},W).$$
\end{lem}
\begin{proof}
Note that  $\rH^1_{\pm}(K^{\ac}_{n,p}, W)$  coincides with the image of 
$\rH^1_{\pm}(K^{\ac}_{n,p}, V):=\left(\rH^1_{\pm}(K^{\ac}_{n,p}, T^{\otimes -1}(1))\otimes \Q_p\right)^{\perp}$ under the natural map $\rH^1_{}(K^{\ac}_{n,p}, V)\to \rH^1_{}(K^{\ac}_{n,p}, W)$,  where the orthogonal complement is with respect to the local Tate duality.
The pairing 
\[
\rH^1_{}(K^{\ac}_{n,p}, V)\otimes_{\CO[\Gal(K_n^{\ac} /K)], \iota}\CO[\Gal(K_n^{\ac} /K)] \ \times \ \rH^1_{}(K^{\ac}_{n,p}, V^{\otimes -1}(1)) \to \CO[1/p]
\] induced by the local Tate duality is $\CO[\Gal(K_n^{\ac} /K)]$-bilinear. 
So the subspace $\rH^1_{\pm}(K_{n,p}^{\ac},V)$ 
consists of the elements $x \in \rH^1_{}(K_{n,p}^{\ac},V)$ such that 
for each finite character $\chi $ on $\Gal(K_n^{\ac}/K)$,
the image $\chi^{-1}(x)$ of $x$ in $\rH^1_{}(K^{}_{p}, V_{\varphi\chi})$
under the map
\[
\rH^1_{}(K^{\ac}_{n,p}, V )\otimes_{\CO[\Gal(K_n^{\ac}/K )],\chi^{-1} }\CO_{\chi}=\rH^1_{}(K^{}_{p}, V_{\varphi\chi})
\]
lies in $\rH^1_{\pm}(K^{}_{p}, V_{\varphi\chi^{}}):=\rH^1_{\pm}(\Q_{p}, \Ind_{K_p/\Q_p}(V_{\varphi\chi^{} }))$.

For $\chi \in \Xi^{\mp}_{\le n}$ and $x \in \rH^1_{\pm}(K^{\ac}_{n,p}, V)$, note that 
\[
\chi^{-1}(x) \in  \rH^1_{\rf}(K^{\ac}_{n,p}, V_{})\otimes_{\CO[\Gal(K_n^{\ac}/K )],\chi^{-1} }\CO_{\chi} = \rH^1_{\rf}(K^{}_{p}, V_{\varphi\chi^{}}). 
\]
Recall that for $\chi \in \Xi^{\pm}$, $\chi^{-1}(\iota\omega_{n}^{\pm})=\chi(\omega_n^{\pm})=0$. 
Hence,
the first inclusion of the lemma follows.

Let $x \in \rH^1_{\rf}(K_{n,p}^{\ac},V )$.
For $\chi \in \Xi^{\mp}_{\le n}$,
by \cite[Thm.\ 1.3]{BKNO},
$$\chi^{-1}(x) \in \rH^1_{\rf}(K_{p}^{},V_{\varphi\chi^{}} )=\rH^1_{\pm}(K_{p}^{},V_{\varphi\chi^{}} ).$$
Hence, the second inclusion of the lemma follows.
\end{proof}

\begin{lem}\label{sel-to-pm}
We have
\[
\corank_{\CO}(\Sel_{+}(K_n^{\ac},W )[\iota \omega_n^{-}])+\corank_{\CO}(\Sel_{-}(K_n^{\ac},W )[\iota \omega_n^{+}])=\corank_{\CO}(\Sel_{}(K_n^{\ac},W)).
\]
\end{lem}
\begin{proof}
The two inclusion maps 
$\iota\omega_n^{\pm}\Sel_{\pm}(K_n^{\ac}, W) \subseteq \Sel(K_n^{\ac},W)$ of Lemma \ref{local-pm-to-f} 
 give rise to the natural map
\begin{equation}\label{sum-of-inclusion}
(\iota\omega_n^{+})\Sel_{+}(K_n^{\ac}, W)\oplus (\iota\omega_n^{-})\Sel_{-}(K_n^{\ac}, W) \to \Sel_{}(K_n^{\ac},W).
\end{equation}
Note that the kernel is finite: 
$$(\iota\omega_n^{+})\Sel_{+}(K_n^{\ac}, W) \cap (\iota\omega_n^{-})\Sel_{-}(K_n^{\ac}, W) \subseteq  (\iota\omega_n^{+})\rH^1(K_n^{\ac}, W) \cap (\iota\omega_n^{-})\rH^1(K_n^{\ac}, W),$$
whose dual is isomorphic to the finite\footnote{as $\CO[\Gal(K_n^{\ac}/K) ]/(\iota\omega_n^+, \iota\omega_n^- )$ is a finite ring} group $\rH^1(K_n^{\ac},T^{\otimes -1}(1))/(\iota\omega_n^+, \iota\omega_n^- )$.
Hence, 
\begin{equation}\label{corank-pm-to-f}
\corank_{\CO} \left((\iota\omega_n^{+})\Sel_{+}(K_n^{\ac}, W)\right)+ \corank_{\CO}\left((\iota\omega_n^{-})\Sel_{-}(K_n^{\ac}, W)\right)\le \corank_{\CO}(\Sel_{}(K_n^{\ac},W)).
\end{equation}

By Lemma \ref{local-pm-to-f}, we have $(\iota\omega_n^{\pm})\Sel(K_n^{\ac},W) \subseteq \Sel_{\pm}(K_n^{\ac}, W)[\iota\omega_n^{\mp}]$. 
 Since $(\iota\omega_n^{+})\Sel(K_n^{\ac},W)\cap (\iota\omega_n^{-})\Sel(K_n^{\ac},W)$ is finite, it follows that 

\begin{equation}\label{corank-with-omega}
\begin{split}
\corank_{\CO}\Sel(K_n^{\ac},W)=&\corank_{\CO} \left((\iota\omega_n^{+})\Sel_{}(K_n^{\ac}, W)\right)+ \corank_{\CO}\left((\iota\omega_n^{-})\Sel_{}(K_n^{\ac}, W)\right)\\
\le & \corank_{\CO} \left(\Sel_{+}(K_n^{\ac}, W)[\iota\omega_n^{-}]\right)+ \corank_{\CO}\left(\Sel_{-}(K_n^{\ac}, W)[\iota\omega_n^{+}]\right).
\end{split}
\end{equation}
For a finitely generated  $\CO[\Gal(K_n^{\ac}/K)]$-module $M$, recall that 
$\corank_{\CO}(\iota\omega_{n}^{\pm}M)=\corank_{\CO}(M[\iota\omega_{n}^{\mp}])$. 
Hence, (\ref{corank-pm-to-f}) and (\ref{corank-with-omega}) imply the lemma.
 \end{proof}

\begin{thm}\label{selmer-rank-behaviour}
For $n \ge 1$, we have 
\[
\corank_{\CO}(\Sel(K_n^{\ac},W)) = \rank_{\CO}(\Lambda_{\ac}/\iota\omega_{n}^{-\varepsilon})+ \rank_{\CO}(X_{\varepsilon, \Lambda_{\ac}-\mathrm{tor}}/\iota \omega_n^{-\varepsilon}) +  \rank_{\CO}(X_{-\varepsilon}/\iota\omega_n^{\varepsilon}),
\]
where $X_{\varepsilon, \Lambda_{\ac}-\mathrm{tor}}$ denotes the $\Lambda_{\ac}$-torsion part of $X_{\varepsilon}$.
In particular, there exists an integer $c\ge 0$ such that for $n\gg 1$
we have
\[
\corank_{\CO}(\Sel(K_n^{\ac},W)) = \frac{p^n-1}{2}+c.
\]
\end{thm}
\begin{proof} 
This is a consequence of Theorem \ref{rank}, Corollary \ref{variant-control}, and Lemma  \ref{sel-to-pm}.
For the `In particular' part, note that, when $\omega_n^{\pm}$ is regarded as a
polynomial in $\gamma$, the degree of $\omega_n^{\varepsilon_p(\Ind_{K/\Q}(\varphi))}$
equals $1+(p^n-1)/2$ and that of $\omega_n^{-\varepsilon_p(\Ind_{K/\Q}(\varphi))}$
equals $(p^{n}-1)/2$.
\end{proof}

\subsection{Asymptotic behaviour of Mordell--Weil rank}
This subsection establishes a formula for the anticyclotomic behaviour of the Mordell--Weil rank of a CM abelian variety  associated to $\varphi$. 

Let $K_{\varphi}$ be a number field which contains  $K(\mathrm{Im}(\varphi))$, 
where $\varphi$ is regarded as a character on the group of fractional ideals relatively prime to the conductor of $\varphi$,
and let $A_{\varphi}$ be a CM abelian variety over $K$ associated to $\varphi$ of dimension $[K_{\varphi}:\Q]/2$,
which is equipped with
$i:K_{\varphi}\hookrightarrow \End^0(A_{\varphi}):=\End(A_{\varphi})\otimes \Q$,
which gives a CM type of $K_{\varphi}$ due to the induced  action 
$K_{\varphi}$ on the $[K_{\varphi}:\Q]/2$-dimensional  $\C$-vector  space $\Omega^1_{A_{\varphi}/\C}$.
We note that its reflex CM type is given by $\sigma_1 \in \Z[\{K\hookrightarrow \C \}]$,
where $\sigma_1:K\hookrightarrow \C$ is the inclusion $K\subseteq\C,$
and we also have 
\[
L(A_{\varphi}/K,s)=\prod_{\sigma: K_{\varphi}\hookrightarrow \C}L(\varphi^{\sigma},s)=\prod_{\sigma: K_{\varphi}\hookrightarrow \C}L(\theta_{\varphi^{\sigma}},s),
\]
where
$\theta_{\varphi^{\sigma}} \in S^{\mathrm{new}}_{2}(\Gamma_0(D_KN_{K/\Q}(\ff)))$ denotes the theta series
attached to $\varphi^{\sigma}$ (cf.\ \cite[Thm.\ 3.4]{Ri}).  
We note that the above construction may be adapted to the case of $\varphi\chi$ as well if $\chi$ is a finite character of $\Gal(K_{\infty}^{\ac}/K).$
The abelian variety $A_{\varphi\chi}$ is not simple unless $K_{\varphi\chi}$ is generated by the image  of $\varphi\chi$.  
In the following,
we sometimes make $K_{\varphi\chi}$ vary, which does not change $\dim_{K_{\varphi\chi}}(A_{\varphi\chi}\otimes_{\Z}\Q).$

 The main result of this subsection is the following. 
 \begin{thm}\label{vGZK}
There exists a constant $c\in\Z$ such that  for any sufficiently large integer $n$, we have 
 $$\dim_{K_{\varphi}}\,(A_{\varphi}(K_n^{\rm ac})\otimes_{\Z}\Q)=\frac{p^n-1}{2}+c. $$ 
\end{thm} 

 Noting that  $A_{\varphi\chi}$ is isogenous to the base change to $K$ of the product of copies of the abelian variety over $\Q$ of $\GL_2$-type associated to $\theta_{\varphi\chi},$
 we have the following fundamental result towards the Birch and Swinnerton-Dyer conjecture due to Gross--Zagier and Kolyvagin (cf.~\cite[Prop.~3.5]{BKO24}).
 
 \begin{prop}\label{GZK}
 Suppose that $\mathrm{ord}_{s=1} L({ \varphi}\chi^{},s)\in\{0,1\}$. 
 Then $$\mathrm{dim}_{K_{\varphi\chi}} (A_{\varphi\chi}(K)\otimes_{\Z} \Q)=\mathrm{ord}_{s=1} L({ \varphi}\chi^{},s).$$ 
 \end{prop} 

\begin{proof}[Proof of Theorem \ref{vGZK}]
The following is similar to the proof of \cite[Thm.\ 8.4]{BKNO}.

 For simplicity of notation, we assume that $\fp$ exactly divides the conductor of $\varphi$: by the same argument as in the proof of Lemma \ref{density-of-xi}, we can reduce to this case by replacing $\varphi$ with $\varphi\chi_0$ with $\chi_0$ being some finite character of $\Gal(K^{\ac}_{\infty}/K)$. 
 
Let $A=A_{\varphi}$, and let $B_n$ denote 
the Weil restriction $\Res_{K_n^{\rm ac}/K}(A_{/K_n^{\rm ac}})$ of $A_{/K_{n}^{\rm ac}}$. 
Then,
$B_n$ is equipped with
\[
K_{\varphi}[\Gal(K_n^{\ac}/K)]\hookrightarrow \End(B_{n/K})\otimes\Q
\]
and $\dim(B_n)=2^{-1}\dim_{\Q}(K_{\varphi}[\Gal(K_n^{\ac}/K)]).$
For $n\ge 1$,
enlarging $K_{\varphi}$ so that $K_{\varphi}$ contains $\Q(\zeta_{p^n}),$
  we have
\[
\begin{split}
K_{\varphi}[\Gal(K_{n}^{\rm ac}/K)] &= (K_{\varphi}\otimes_{\Q} \Q[\gamma_{n}]/\Phi_n(\gamma_{n})) \times K_{\varphi}[\Gal(K_{n-1}^{\rm ac}/K)]\\
&\cong \prod_{\chi: \mathrm{ord}(\chi)=p^n} K_{\varphi}(\chi) \quad  \times  \quad 
K_{\varphi}[\Gal(K_{n-1}^{\rm ac}/K)],
\end{split}
\]
where $\gamma_n$ is a generator of $\Gal(K_{n}^{\rm ac}/K)$, and $\Phi_n(X)=(X^{p^{n}}-1)/(X^{p^{n-1}}-1),$
$\chi$ ranges over all characters of $\Gal(K_n^{\ac}/K)$ of order $p^n$,
the isomorphism is induced by the characters $\chi:\Gal(K_{n}^{\rm ac}/K) \to (K_{\varphi})^{\times}$ 
and $K_{\varphi}(\chi)$ denotes the $K_{\varphi}[\Gal(K_{n}^{\rm ac}/K)]$-algebra $K_{\varphi}$ whose structure morphism is given by   
$\chi$.
Correspondingly,
taking $K_{\varphi\chi}$ to be $K_{\varphi}$
(as $K(\mathrm{Im}(\chi))=K(\zeta_{p^n})\subseteq K_{\varphi}$),
 we have
an isogeny
$$
B_n \quad \sim \quad  \prod_{\chi:\mathrm{ord}(\chi)=p^n} A_{\varphi\chi}  \  \times \ B_{n-1}
$$
 of abelian varieties over $K$ which is compatible with the action of $K_{\varphi}=K_{\varphi\chi}$.
 We then have 
\begin{equation*}
A(K_n^{\rm ac})\otimes_{\Z} \Q\cong \bigoplus_{\chi:\mathrm{ord}(\chi)=p^n} \left(A_{\varphi\chi}(K)\otimes_{\Z} \Q\right)  \ \oplus \ \left(A(K_{n-1}^{\rm ac})\otimes_{\Z} \Q\right).
\end{equation*}
 Hence, since the number of $\chi$ of order $p^n$ is $p^{n}-p^{n-1}$, 
  it suffices to show that for any sufficiently large integer $n$ 
 and a character $\chi$ of order $p^n$, we have
\begin{equation}\label{eq, generic rank}
\dim_{K_{\varphi\chi}} \left(A_{\varphi\chi}(K) \otimes_{\Z}\Q\right)+
\dim_{K_{\varphi\chi^{b}}} \left(A_{\varphi\chi^b}(K) \otimes_{\Z}\Q\right)
=1,
\end{equation}
where $b\in \Z_p^{\times}$ is an element such that $\left(\frac{b}{p}\right)=-1$.

By Rohrlich \cite{Ro} and Lemma \ref{density-of-gxi},
for sufficiently large $n$ and for a character $\chi$ on $\Gal(K_{\infty}^{\ac}/K)$ of order $p^n$,
putting $\varepsilon_{\chi}:=\varepsilon({\varphi}\chi)\in\{\pm 1\}$,
we have 
$$
\ord_{s=1}L({\varphi}\chi,s)=\frac{1-\varepsilon_{\chi}}{2},\quad \ord_{s=1}L({\varphi}\chi^{b},s)=\frac{1+\varepsilon_{\chi}}{2} \in \{0,1\},
$$
which implies that  $\ord_{s=1}L({\varphi}\chi,s)+\ord_{s=1}L({\varphi}\chi^{b},s)=1.$
Hence, by Proposition~\ref{GZK}
we obtain (\ref{eq, generic rank}).
 \end{proof}

\section{The $p$-adic $L$-function and Selmer elements}\label{s, pLo}
 The main result of this section is Theorem \ref{p-adic-L-with-global-class}: a link between the Rubin-type $p$-adic $L$-function at characters with $\varepsilon$-constant being $-1$, and elements in the associated Selmer group. 
\subsection{Selmer element} 
\subsubsection{Set-up} Let the setting be as in Section \ref{section-Selmer-group}, and 
suppose that $p$ is ramified in $K$.

For simplicity, we suppose that Assumption \ref{for-integrality} (1) holds.
Without assuming it, the following arguments work by tensoring with $\Q_p$. 

Put
$
\tilde{\bT}=\Ind_{K/\Q}(T^{\otimes -1}(1)\otimes_{\CO}\Lambda_{\ac}).
$
Since $\det_{\Lambda_{\ac}}(\tilde{\bT})$ is a free $\Lambda_{\ac}$-module of rank one on which $G_{\Q}$ acts  by the $p$-adic cyclotomic character,
there exists an alternating, perfect $G_{\Q}$-equivariant pairing
\begin{equation}\label{iwasawa-etale-pairing}
 \tilde{\bT}\times \tilde{\bT} \to \Lambda_{\ac}(1),
 \end{equation}
which we fix. 
It induces 
a symmetric, perfect $\Lambda_{\ac}$-linear pairing
\[
(\ , \ ): \rH^1(\Q_p, \tilde{\bT}) \times \rH^1(\Q_p, \tilde{\bT}) \to \Lambda_{\ac},
\]
with respect to which $\rH^1_{\pm}(\Q_p, \tilde{\bT})$ is a Lagrangian (cf.\ \cite[Thm.\ 1.3]{BKNO}).
Let $v_{\pm}$ be $\Lambda_{\ac}$-bases of $\rH^1_{\pm}(\Q_p, \tilde{\bT})$ such that 
\begin{equation}
(v_{+}, v_{-})=1 \in \Lambda_{\ac}.
\end{equation}

\subsubsection{Selmer element}

Adapting the notation of \S \ref{interpolation-section},
for a character $\chi=\varphi_{\ac}^{k}\eta_{\ac}^{-k}\chi^{\prime}$ of $\Gamma$, where $\chi^{\prime}$ is a finite order character of $\Gamma$ and $k\ge 0$,  put $\CO_{\chi}=\CO[\mathrm{Im}(\chi)]$, and denote by 
\[
\chi(z^{t, \ac}_{p^{\infty} \ff}) \in \rH^1(\Q, \tilde{\bT}\otimes_{\Lambda_{\ac},\chi } \CO_{\chi})=\rH^1(K, T_{\varphi\chi}^{\otimes -1}(1))
\]
the image of $z^{t,\ac}_{p^{\infty} \ff}$.
We similarly define $\chi(v_{\pm}) \in \rH^1(\Q_p, \tilde{\bT}\otimes_{\Lambda_{\ac},\chi } \CO_{\chi})=\rH^1(K_p, T_{\varphi\chi}^{\otimes -1}(1))$. 
\begin{lem}\label{lie-in-f}
If $\hat{\varepsilon}_p(\Ind_{K/\Q}(\varphi\chi))=-\varepsilon$,
then 
$$
\chi(z^{t, \ac}_{p^{\infty} \ff}) \in  \rH^1_{\rf}(\Q, \Ind_{K/\Q}(T_{\varphi\chi}^{\otimes -1}(1)))= \Sel(K, T_{\varphi\chi}^{\otimes -1}(1)). 
$$
\end{lem}
\begin{proof}
By Proposition \ref{prop, zeta-in-epsilon} and Theorem \ref{lsd-ind}, 
we have $\loc_p(\chi(z^{t, \ac}_{p^{\infty} \ff})) \in  \rH^1_{\varepsilon}(\Q_p, \Ind_{K/\Q}(T_{\varphi\chi}^{\otimes -1}(1)))$.
By \cite[Thm.\ 1.3]{BKNO},
\[
\rH^1_{\varepsilon}(\Q_p, \Ind_{K/\Q}(T_{\varphi\chi}^{\otimes -1}(1)))=\rH^1_{-\hat{\varepsilon}_p(\Ind_{K/\Q}(\varphi\chi))}(\Q_p, \Ind_{K/\Q}(T_{\varphi\chi}^{\otimes -1}(1)))=\rH^1_{\rf}(\Q_p, \Ind_{K/\Q}(T_{\varphi\chi}^{\otimes -1}(1))).
\]
Hence, (\ref{lie-in-rel}) implies the lemma. 
\end{proof}
\subsubsection{Bloch--Kato logarithm and exponential}
Let 
\[
\log: \rH^1_{\rf}(\Q_p, \Ind_{K/\Q}(V_{\varphi\chi}^{\otimes -1}(1))) \to \frac{D_{\dR}(\Ind_{K/\Q}(V_{\varphi\chi}^{\otimes -1}(1)))}{D_{\dR}^{0}(\Ind_{K/\Q}(V_{\varphi\chi}^{\otimes -1}(1)))}
\]
 denote the Bloch--Kato logarithm map
and \[
\exp^*: \rH^1(\Q_p, \Ind_{K/\Q}(V_{\varphi\chi}^{\otimes -1}(1))) \to D_{\dR}^{0}(\Ind_{K/\Q}(V_{\varphi\chi}^{\otimes -1}(1)))
\]
the dual exponential map which may be defined as the dual of the exponential map (the inverse of $\log$)
with respect to the perfect pairings
\[
( \ , \ )_{\chi}: \rH^1(\Q_p, \Ind_{K/\Q}(V_{\varphi\chi}^{\otimes -1}(1) )) \times \rH^1(\Q_p, \Ind_{K/\Q}(V_{\varphi\chi}^{\otimes -1}(1)) ) \to \CO_{\chi}[1/p]
\]
\[
[ \ , \ ]: D_{\dR}(\Ind_{K/\Q}(V_{\varphi\chi}^{\otimes -1}(1)))\times D_{\dR}(\Ind_{K/\Q}(V_{\varphi\chi}^{\otimes -1}(1))) \to \CO_{\chi}[1/p]
\]
induced by the perfect pairing $\tilde{\bT}\otimes_{\Lambda_{\ac},\chi}\CO_{\chi} \times   \tilde{\bT}\otimes_{\Lambda_{\ac},\chi}\CO_{\chi} \to \CO_{\chi}(1),$
the base change of (\ref{iwasawa-etale-pairing}). 

Since $\Ind_{K/\Q}(V_{\varphi\chi}^{\otimes -1}(1))$ is symplectic self-dual, $D_{\dR}^{0}(\Ind_{K/\Q}(V_{\varphi\chi}^{\otimes -1}(1)))$ is a maximal isotropic line for $[\ ,\ ]$ (cf.~\cite{BKNO}); in particular $[\ ,\ ]$ induces a perfect pairing between $D_{\dR}^{0}(\Ind_{K/\Q}(V_{\varphi\chi}^{\otimes -1}(1)))$ and the quotient $D_{\dR}(\Ind_{K/\Q}(V_{\varphi\chi}^{\otimes -1}(1)))/D_{\dR}^{0}(\Ind_{K/\Q}(V_{\varphi\chi}^{\otimes -1}(1)))$.

Let $\omega \in D_{\dR}^{0}(\Ind_{K/\Q}(V_{\varphi\chi}^{\otimes -1}(1)))$ be a non-zero element. Since $\exp^*$ takes values in the line $\CO_{\chi}[1/p]\cdot \omega$, and since, by the above, $[\ z,\omega\ ]$ depends only on the class of $z\in D_{\dR}(\Ind_{K/\Q}(V_{\varphi\chi}^{\otimes -1}(1)))$ in $D_{\dR}(\Ind_{K/\Q}(V_{\varphi\chi}^{\otimes -1}(1)))/D_{\dR}^{0}(\Ind_{K/\Q}(V_{\varphi\chi}^{\otimes -1}(1)))$, $\omega$ normalises $\log$ and $\exp^*$ to
 \[
\log_{\omega}: \rH^1_{\rf}(\Q_p, \Ind_{K/\Q}(V_{\varphi\chi}^{\otimes -1}(1))) \to \CO_{\chi}[1/p], \quad \exp_{\omega}^*: \rH^1(\Q_p, \Ind_{K/\Q}(V_{\varphi\chi}^{\otimes -1}(1))) \to \CO_{\chi}[1/p],
\]
characterised by $[\log(x),\omega]=\log_{\omega}(x)$ and $\exp^*(y)=\exp_{\omega}^*(y)\cdot \omega$ for $x \in \rH^1_{\rf}(\Q_p, \Ind_{K/\Q}(V_{\varphi\chi}^{\otimes -1}(1)))$ and $y \in \rH^1(\Q_p, \Ind_{K/\Q}(V_{\varphi\chi}^{\otimes -1}(1)))$.
\subsection{Main result}

\begin{thm}\label{p-adic-L-with-global-class}
Let the notation be as above.
For $\omega \in  D_{\dR}^{0}(\Ind_{K/\Q}(V_{\varphi\chi}^{\otimes -1}(1))) \setminus\{0\}$,
if  $\hat{\varepsilon}_p(\Ind_{K/\Q}(\varphi\chi))=-\varepsilon$,
then we have  
\[
\chi(\mathscr{L}_{p,v_{\varepsilon},t}(\varphi))=\exp^*_{\omega}(\chi(v_{-\varepsilon})) \cdot \log_{\omega}(\chi(z^{t, \ac}_{p^{\infty} \ff})). 
\]
\end{thm}
\begin{proof}
By Definition~\ref{def, RpL}, we have
$$
(\loc_p(z^{t,\ac}_{p^{\infty} \ff}), v_{-\varepsilon} )=\mathscr{L}_{p,v_{\varepsilon},t}(\varphi). 
$$
Applying $\chi$, and noting that $\loc_p$ commutes with specialisation along $\chi$, we obtain
$$
(\loc_p(\chi(z^{t,\ac}_{p^{\infty} \ff})), \chi(v_{-\varepsilon}) )_{\chi}=\chi(\mathscr{L}_{p,v_{\varepsilon},t}(\varphi)). 
$$
By Lemma~\ref{lie-in-f}, $x:=\loc_p(\chi(z^{t,\ac}_{p^{\infty} \ff}))$ lies in $\rH^1_{\rf}(\Q_p, \Ind_{K/\Q}(V_{\varphi\chi}^{\otimes -1}(1)))$. By the compatibility of the local Tate pairing $(\ ,\ )_{\chi}$ with the Bloch--Kato logarithm and dual exponential maps (see e.g.\ \cite{K}), $(x,y)_{\chi}=[\log(x), \exp^*(y)]$ for $y \in \rH^1(\Q_p, \Ind_{K/\Q}(V_{\varphi\chi}^{\otimes -1}(1)))$. Taking $y=\chi(v_{-\varepsilon})$, we have
$$
[\log(\chi(z^{t, \ac}_{p^{\infty} \ff})), \exp^*(\chi(v_{-\varepsilon})) ]=\chi(\mathscr{L}_{p,v_{\varepsilon},t}(\varphi)), 
$$ 
and by the normalisation above the left-hand side equals $\exp^*_{\omega}(\chi(v_{-\varepsilon}))\cdot \log_{\omega}(\chi(z^{t, \ac}_{p^{\infty} \ff}))$, concluding the proof. 
\end{proof}

\end{document}